\documentclass[a4paper]{amsart}
\usepackage{xcolor} 
\usepackage{graphicx}
\usepackage[toc]{appendix} 
\usepackage{amsmath, amssymb}
\usepackage{amsfonts} 
\usepackage{amsthm} 
\usepackage{systeme}
\usepackage[top=3cm, bottom=3cm,inner=1.5cm, outer=1.5cm]{geometry}	
\usepackage{color}   
\usepackage{hyperref}
\hypersetup{
    colorlinks=true, 
    linktoc=all,    
    linkcolor=black,  
    citecolor = black,
}
\usepackage{manfnt}

\newtheorem{theorem}{Theorem}[section] 
\newtheorem{lemma}[theorem]{Lemma}

\newtheorem{remark}[theorem]{Remark}
\newtheorem{prop}[theorem]{Proposition}

\numberwithin{equation}{section}

\newcommand{\abs}[1]{\lvert #1 \rvert}
\newcommand{\sumstack}[1]{{\substack{#1}}}
\newcommand{\nsmod}[1]{\text{ }\mathrm{mod}\text{ $#1$}} 
\newcommand{\fq}[0]{\mathbb{F}_q}

\DeclareMathOperator*{\GLtwo}{GL_2}
\DeclareMathOperator*{\SLtwo}{SL_2}
\DeclareMathOperator*{\vol}{Vol}
\theoremstyle{definition}

\title{Averages of Artin characters and their relation to the Davenport--Heilbronn theorem }
\author{Victor Ahlquist}
\address{Department of Mathematical Sciences, Chalmers University of Technology and the University \newline
	\rule[0ex]{0ex}{0ex}\hspace{8pt} of Gothenburg, SE-412 96 Gothenburg, Sweden}
\email{vicahlqu@chalmers.se} 
\date{}
\begin{document}

\begin{abstract}
    We bound the error term in the counting function of cubic extensions of $\mathbb{Q}$ by $\mathcal{O}_\epsilon\left(X^{\theta+\epsilon}\right)$, with $\theta < 2/3$, improving upon results due to Bhargava, Taniguchi, and Thorne. The proof proceeds by relating a certain discriminant-reducing sieve, due to Bhargava, Shankar, and Tsimerman, to averages of Artin characters associated with the Artin $L$-function $\zeta_K(s)/\zeta(s)$, with $K$ a cubic field. Conditional on the Generalised Riemann Hypothesis (GRH) for Dedekind zeta functions $\zeta_K(s)$, we obtain further savings and show that the error term is $\mathcal{O}\left(X^{2/3-1/51+\epsilon}\right)$. Moreover, for the analogous smooth counting problem, we obtain a bound $\mathcal{O}\left(X^{1/2+\epsilon}\right)$ for the error term, conditionally improving results of Shankar, Södergren, and Templier. Our arguments can also handle finitely many splitting conditions.

    The methods we use to study cubic fields are natural extensions of methods which are used to study the one-level density associated with the above family of Artin $L$-functions. In this direction, we improve the admissible support $[-\sigma,\sigma]$, from the previously known admissible value $\sigma < 2/5$, to $\sigma < 1$, conditional on the GRH. This also proves that at least $75\%$ of these $L$-functions are non-vanishing at the central point, conditionally improving upon results of Shankar, Södergren and Templier. 

    Finally, one may study similar questions over a rational function field $\fq(T)$, with $q$ coprime to $2$ and $3$. Here, we obtain a proportion of non-vanishing of at least $75\%$ for the analogous family of $L$-functions. Furthermore, we obtain a bound $\mathcal{O}\left(X^{1/2+\epsilon}\right)$ for the error term in the counting function of cubic function fields, improving previous results of the author. Over function fields, our results are unconditional.
\end{abstract}

\maketitle

\section{Introduction}
Let $X\geq 1$ be a real number, and consider the collection $\mathcal{F}^\pm(X)$ of $S_3$-cubic fields ordered by discriminant, i.e. 
\begin{equation*}
    \mathcal{F}^\pm(X):= \{K: [K:\mathbb{Q}]=3,\,\,0 <\pm \mathrm{Disc}(K)\leq X,\,\,\text{$K$ non-Galois} \},
\end{equation*}
where we include only one field $K$ from every isomorphism class.
The associated counting function
\begin{equation}\label{firstsharpctdef}
    N^\pm(X) := \#\mathcal{F}^\pm(X)
\end{equation}
was first studied by Davenport and Heilbronn \cite{DH}, who proved that
\begin{equation*}
    N^\pm(X)=C_1^\pm X+o(X),
\end{equation*}
as $X\to \infty$, for explicit constants $C_1^\pm$. We remark that the problem of counting $S_3$-cubic fields is essentially equivalent to the problem of counting all cubic fields, as the number of cyclic cubic fields with $0 <\mathrm{Disc}(K) \leq X$ is $\ll X^{1/2}$. Indeed, this follows from more precise asymptotics obtained by Cohn \cite{Cohn}. 

A power-saving bound for the error term above was proven by Belabas--Bhargava--Pomerance \cite{BBP}. The existence of a secondary term in the counting function $N^\pm(X)$ was conjectured by Roberts \cite{Roberts} and proven independently by Bhargava--Shankar--Tsimerman \cite{BST} and Taniguchi--Thorne \cite{TTcube}. To date, the most precise asymptotic is
\begin{equation}\label{BTTasymp}
    N^\pm(X)=C_1^\pm X+C_2^\pm X^{5/6} + \mathcal{O}_\epsilon\left(X^{2/3+\epsilon}\right),
\end{equation}
with $C_2^\pm$ explicit constants, due to Bhargava--Taniguchi--Thorne \cite{BTT}. In fact, they prove a slightly stronger result, where one may replace the $\epsilon$ with a certain power of $\log X$.

We improve the error term in the asymptotic \eqref{BTTasymp}. Our first result in this direction is a quite modest unconditional improvement of the above. 
\begin{theorem}\label{uncondsharpctthm}
    We have
    \begin{equation*}
        N^\pm(X)=C_1^\pm X+C_2^\pm X^{5/6} + \mathcal{O}_\epsilon\left(X^{2/3-c_B+\epsilon}\right),
    \end{equation*}
    where
    \begin{equation*}
        c_B = \frac{1}{288B+15},\text{ and } B= 216\cdot 10^7.
    \end{equation*}
    
\end{theorem}
\begin{remark}
    Because of the recent proof of the quasi-Riemann hypothesis for Dedekind zeta functions \cite{quasidede}, the theorem above is superseded by Theorem \ref{quasiapplthm}, which uses the quasi-Riemann hypothesis as an additional input. Nevertheless, we have elected to keep the above result, and its proof, to showcase the non-necessity of the quasi-Riemann hypothesis for improving the exponent in the error term past the previous best value $2/3$.
\end{remark}

The main ingredient used in the proof of Theorem \ref{uncondsharpctthm} is a relation between a discriminant-reducing sieve used to study cubic fields and certain averages of Artin characters $\lambda_K$, associated with cubic fields $K$. Bounds for such averages are closely related to the locations of the zeros of the Artin $L$-functions $L(s,\rho_K) := \zeta_K(s)/\zeta(s)$. By using a zero-density result of Thorner--Zaman \cite[Theorem 1.1]{TZ}, we are able to bound these averages nontrivially, and obtain Theorem \ref{uncondsharpctthm} above. 

The GRH for $\zeta_K(s)$ asserts that all nontrivial zeros of $\zeta_K(s)$ lie on the line $\Re(s)=1/2$, whence the same holds for zeros of $L(s,\rho_K)$. As a consequence, we are able to significantly improve the error term in Theorem \ref{uncondsharpctthm} under the GRH.
\begin{theorem}\label{sharpgrhthm}
    Assume the GRH for all Dedekind zeta functions attached to quadratic and cubic extensions of $\mathbb{Q}$. Then,
    \begin{equation*}
        N^\pm(X)=C_1^\pm X+C_2^\pm X^{5/6} + \mathcal{O}_\epsilon\left(X^{2/3-1/51+\epsilon}\right).
    \end{equation*}
\end{theorem}
Numerical experiments, performed by Belabas \cite{Belabas} and analysed in Roberts \cite{Roberts}, indicate that the true error term is possibly of square-root size. We are unable to obtain a result of this quality for the sharp counting function $N^\pm(X)$. However, for the analogous smooth counting problem, we obtain a square-root error term.
\begin{theorem}\label{smoothctthm}
    Let $h:[0,\infty)\to\mathbb{R}$ be a $C^\infty$ function with compact support. Denote by $K^\pm_m$ the number of $S_3$-cubic fields of discriminant equal to $\pm m$. Then, under the same conditions as in Theorem \ref{sharpgrhthm}, we have for the smooth counting function
    \begin{equation}\label{firstsmoothctdef}
        N^\pm(h;X) := \sum_{m\geq 1}K^\pm_mh\left(\frac{m}{X}\right)
    \end{equation}
    that
    \begin{equation*}
        N^\pm(h;X) = C_1^\pm \Tilde{h}(1) X+C_2^\pm \Tilde{h}(5/6) X^{5/6} + \mathcal{O}_{\epsilon,h}\left(X^{1/2+\epsilon}\right),
    \end{equation*}
    where 
    \begin{equation*}
        \Tilde{h}(s)= \int_{0}^\infty h(x)x^{s-1}dx
    \end{equation*}
    denotes the Mellin transform of $h$.
\end{theorem}
\begin{remark}
    Shankar--Södergren--Templier \cite[Theorem 6.10]{SST} unconditionally obtained the above asymptotic formula with an error term of size $\mathcal{O}_{\epsilon,h}\left(X^{2/3+\epsilon}\right)$, matching the previous best asymptotic \eqref{BTTasymp} for the sharp counting function. 
\end{remark}
The arguments used to prove Theorems \ref{uncondsharpctthm}, \ref{sharpgrhthm} and \ref{smoothctthm} can also be modified to handle finitely many splitting conditions. Recall that a rational prime $p\in \mathbb{Z}$ can split in five different ways in a cubic extension $K$. Indeed, $p$ is either inert, partially split, completely split, partially ramified, or totally ramified. We denote these splitting types by $(3), (21), (111), (1^21)$, and $(1^3)$, respectively. Let $\Sigma=(\Sigma_p)_p$ be a collection of specifications of splitting types at finitely many primes $p$. Denote by $P_{\Sigma,S}$ the product of all primes $p$ for which $\Sigma_p$ equals the splitting type $S$.

Using the above setup, we may define counting functions $N^\pm(\Sigma;X)$ and $N^\pm(h,\Sigma;X)$ as in \eqref{firstsharpctdef} and \eqref{firstsmoothctdef}, except that we restrict to counting fields which satisfy the splitting conditions described by $\Sigma$. We then have the following result, improving upon \cite[Theorem 1.3]{BTT}.
\begin{theorem}\label{splitcondthmQ}
    Let $\Sigma$ be a collection of splitting types as described above. Then, unconditionally, we have that
    \begin{equation*}
       N^\pm(\Sigma;X) = C_1^\pm(\Sigma)X+C_2^\pm(\Sigma)X^{5/6}+E_1(\Sigma)\mathcal{O}_\epsilon(X^{2/3-c_B+\epsilon}) ,
    \end{equation*}
    under the assumption that $P_{\Sigma,(1^3)}\leq X^{1/100}$. Here $C_1^\pm(\Sigma)$ and $C_2^\pm(\Sigma)$ are the constants from \cite[Theorem 1.3]{BTT}.

    Conditional on the GRH for Dedekind zeta functions associated with quadratic and cubic extensions of $\mathbb
    Q$, we have
    \begin{equation*}
        N^\pm(\Sigma;X) = C_1^\pm(\Sigma)X+C_2^\pm(\Sigma)X^{5/6}+E_2(\Sigma)\mathcal{O}_\epsilon(X^{2/3-1/51+\epsilon}),
    \end{equation*}
    and
    \begin{equation*}
        N^\pm(h,\Sigma;X) = C_1^\pm(\Sigma)\Tilde{h}(1)X+C_2^\pm(\Sigma)\Tilde{h}(5/6)X^{5/6}+E_3(\Sigma)\mathcal{O}_\epsilon(X^{1/2+\epsilon}),
    \end{equation*}
    with no assumptions on the size of $P_{\Sigma,(1^3)}$.

    The error terms $E_i(\Sigma)$ are defined multiplicatively
    \begin{equation*}
        E_i(\Sigma)=\prod_p E_i(\Sigma_p),
    \end{equation*}
    with
    \begin{equation*}
    \begin{split}
        &E_1(\Sigma_p)=\begin{cases}
            p^{2/3+2c_B}, &\text{ if $\Sigma_p\in \{(111),(21),(3), (1^21) \}$,}\\
            p^{3c_B-1}, &\text{ if $\Sigma_p=(1^3)$,}
        \end{cases}\\
        &E_2(\Sigma_p)=\begin{cases}
            p^{12/17}, &\text{ if $\Sigma_p\in \{(111),(21),(3), (1^21) \}$,}\\
            p^{-16/17}, &\text{ if $\Sigma_p=(1^3)$,}
        \end{cases}\\
        &E_3(\Sigma_p)=\begin{cases}
            p, &\text{ if $\Sigma_p\in \{(111),(21),(3), (1^21) \}$,}\\
            1, &\text{ if $\Sigma_p=(1^3)$.}
        \end{cases}
    \end{split}
    \end{equation*}
\end{theorem}
\begin{remark}
    The error term associated with the partially ramified splitting type may be averaged for free, in the sense described in \cite[Theorem 1.4]{BTT}.
\end{remark}
\begin{remark}
    The condition $P_{\Sigma,(1^3)}\leq X^{1/100}$ is somewhat unnatural, and arises because of the way we apply \cite[Theorem 1.1.]{TZ}, see the discussion at the end of Section \ref{splitcondsectsharp}.
\end{remark}
\begin{remark}
    We are not able to improve the error term from \cite[Theorem 1.2]{BTT} for the problem of counting cubic fields with a fundamental discriminant. It is possible that a modification of our methods could handle this case, but we do not pursue this question here.
\end{remark}
A consequence of our unconditional bounds for the error term above is that the so-called Ratios Conjecture of Conrey--Farmer--Zirnbauer \cite{RatiosConjecture}, used for studying averages of quotients of $L$-functions, predicts too small of an error term in the case of cubic Artin $L$-functions $L(s,\rho_K)$ ordered by discriminant, even when substantially modified. Indeed, work of Cho--Fiorilli--Lee--Södergren implies that given a bound for the error term as above, with the exponent of $X$ being strictly smaller than $2/3$, the true error term must be larger than the expected square-root bound, see \cite[Remark 1.5]{CFLS}.

The quasi-Riemann hypothesis for an $L$-function $L(s,f)$ asserts that $L(s,f)$ has no zeros with real part strictly larger than $1-\tau$, for some $0< \tau \leq 1/2$. A quasi-Riemann hypothesis for $L(s,f)$ implies power-saving bounds for the error terms in the summatory functions of various sequences associated with the coefficients of $L(s,f)$. Very recently, an internal OpenAI model proved the quasi-Riemann hypothesis, for Dirichlet $L$-functions (including the Riemann zeta function) \cite[Problem 003]{OpenAI}. Moreover, they have provided a lean-verified proof. Building upon the methods used to prove this result, Cureton--Thorner \cite[Theorem 1.1]{quasidede} used AxiomProver to generate and verify a proof of the quasi-Riemann hypothesis for an arbitrary Dedekind zeta function. More precisely, they proved that $\zeta_K(s)$ has no zeros with real part of $s$ strictly larger than $1-1/32$.

By replacing the use of the zero-density result of Thorner--Zaman \cite[Theorem 1.1]{TZ} with the quasi-Riemann hypothesis for $\zeta_K(s)$, we are able to improve Theorem \ref{uncondsharpctthm} and the unconditional statement from Theorem \ref{splitcondthmQ}. More precisely, we have the following result.
\begin{theorem}\label{quasiapplthm}
    Assume that $\zeta_K(s)$ has no zero with real part strictly larger than $1-\tau$, for some fixed $0<\tau\leq 1/2$, as soon as $K$ is a cubic or quadratic extension of $\mathbb{Q}$. Then,
    \begin{equation}
        N^\pm(X)=C_1^\pm X+C_2^\pm X^{5/6} + \mathcal{O}_\epsilon\left(X^{2/3-c(\tau)+\epsilon}\right),
    \end{equation}
    with 
    \begin{equation*}
        c(\tau)= \frac{\tau}{18+15\tau}.
    \end{equation*}
    Moreover,
    \begin{equation*}
        N^\pm(\Sigma;X) = C_1^\pm(\Sigma)X+C_2^\pm(\Sigma)X^{5/6}+E_4(\Sigma)\mathcal{O}_\epsilon(X^{2/3-c(\tau)+\epsilon}),
    \end{equation*}
    where
    \begin{equation*}
        E_4(\Sigma)=\prod_p E_4(\Sigma_p),
    \end{equation*}
    and
    \begin{equation*}
        E_4(\Sigma_p)=\begin{cases}
            p^{2/3+2c(\tau)}, &\text{ if $\Sigma_p\in \{(111),(21),(3), (1^21) \}$,}\\
            p^{3c(\tau)-1}, &\text{ if $\Sigma_p=(1^3)$,}
        \end{cases}
    \end{equation*}
\end{theorem}
\begin{remark}
    The choice $\tau=1/32$ from \cite{quasidede} yields 
    \begin{equation}
        c(1/32) = \frac{1}{591}.
    \end{equation}
    The choice $\tau = 1/2$, corresponding to the GRH, instead yields 
    \begin{equation*}
        c(1/2) = 1/51,
    \end{equation*}
    which gives the same error term as in the conditional asymptotic for the sharp counting function in Theorem \ref{splitcondthmQ}.
\end{remark}

\subsection{Non-vanishing and the one-level density}

As previously mentioned, the improvements in the various error terms above under the GRH are due to the fact that we have good control over the real parts of all zeros of $L(s,\rho_K)$. A related problem is studying the imaginary parts of these zeros, i.e. how the zeros distribute along the line $\Re(s)=1/2$. The methods used to prove the above theorems are all extensions of methods that we use to study the distribution of zeros which lie close to the real axis, so-called low-lying zeros, of the family
\begin{equation*}
    \mathcal{L}^\pm(X) = \{L(s,\rho_K): K\in \mathcal{F}^\pm(X)\}.
\end{equation*}
To be more precise, this distribution may be studied using the one-level density
\begin{equation*}
    \frac{1}{N^\pm(X)}\sum_{0 < \pm \mathrm{Disc}(K) \leq X}\sum_{\gamma_K}\phi\left(\frac{\log X}{2\pi }\gamma_K\right),
\end{equation*}
where $\gamma_K$ ranges over the ordinates of the zeros of $L(s,\rho_K)$. The function $\phi$ is a Schwartz function whose Fourier transform is supported on the interval $[-\sigma, \sigma]$ for a finite $\sigma < \infty$. By the rapid decay of $\phi$, the one-level density is essentially a summation over ordinates which are close to zero, i.e. which lie close to the central point $1/2$, if one assumes the GRH.

The Katz--Sarnak conjecture \cite{Katz-Sarnak} states that as $X\to \infty$, the one-level density of a natural family of $L$-functions tends to the limit
\begin{equation*}
    \int_{-\infty}^\infty \phi(u)W_G(u)du,
\end{equation*}
where $W_G(u)$ is one of five density functions, depending on the so-called symmetry type $G$ of the family in question. The Katz--Sarnak conjecture has been verified for several different families, if one assumes that the Fourier transform of $\phi$ is compactly supported as above, with $\sigma$ restricted in size; see e.g. \cite{ Hughes-Rudnick, ILS, Ozluk-Snyder, Waxman, Young1} for results concerning a wide variety of families, and see e.g. \cite{Ahlquist1, Bucuretal, CK, Durlanik, Meisner, SSTEqui} for results concerning families associated with field extensions. By proving the Katz--Sarnak conjecture for large enough $\sigma$, one is able to obtain non-vanishing results for the underlying family, under the GRH.  

The one-level density of the family $\mathcal{L}^\pm(X)$ was first studied by Yang \cite{Yang}, who obtained an admissible support of $\sigma < 1/50$. Different aspects of this family were further studied by Cho--Kim \cite{CK}, Shankar--Södergren--Templier \cite{SSTEqui}, and  Cho--Fiorilli--Lee--Södergren \cite{CFLS}. The quality of one-level density results for this family is closely tied to averages of the Artin characters $\lambda_K$ over primes. One way to study such averages is by using asymptotic formulae for the counting functions $N^\pm(\Sigma;X)$. The previous best results \cite[Theorem 1.3]{BTT} for the error term in this counting function yields an unconditional admissible support $\sigma < 2/7$, as remarked in \cite{CFLS}. By restricting to a subfamily and counting with a smooth weight function, Shankar--Södergren--Templier \cite[Theorem 5]{SST} managed to improve this and obtain a support $\sigma < 2/5$ by studying the average of characters more directly.

In some families, by assuming a relevant form of the GRH, one is able to improve the admissible support in the one-level density and thus also obtain better non-vanishing results. However, for the family of non-Galois cubic fields, there has up to this point been no known method for utilising the GRH in a way that improves upon the unconditional results. Our next theorem conditionally establishes the Katz--Sarnak conjecture for the family $\mathcal{L}^\pm(X)$, when $\sigma < 1$.
\begin{theorem}\label{onelevQ}
Assume the GRH for Dedekind zeta functions associated with quadratic and cubic fields. Then, the Katz--Sarnak conjecture for the family $\mathcal{L}^\pm(X)$ holds for $\sigma < 1$. More precisely, for test functions $\phi$ whose Fourier transform is supported in $(-1,1)$, the one-level density of $\mathcal{L}^\pm(X)$ converges to the expected symplectic density
    \begin{equation*}
        \widehat{\phi}(0)-\frac{\phi(0)}{2},
    \end{equation*}
    as $X\to \infty$.
\end{theorem}
One can sometimes obtain conditional non-vanishing results by establishing the Katz--Sarnak conjecture for sufficiently large support, and the strength of the non-vanishing result is then directly dependent on the size of the admissible support. For families of $L$-functions associated with field extensions, not much is known in the way of non-vanishing, as the obtained admissible support is usually too small. The most well-understood family is that of Dedekind zeta functions associated with quadratic fields. This family is closely related to the family of quadratic Dirichlet $L$-functions, and here Özluk and Snyder proved the Katz--Sarnak conjecture with support $\sigma < 2$, and thus obtained a proportion of non-vanishing of $93.75\%$. With some additional optimisation, see e.g. \cite[Appendix A]{ILS}, this number can be improved slightly to approximately $94.27\%$. Results for quadratic Dirichlet $L$-functions over a rational function field $\fq(T)$ of the same quality was obtained by Rudnick \cite{Rudnick}; see also \cite{BF}. Very recently, Koymans--Pagano--Shusterman improved the proportion of non-vanishing to $100\%$ over function fields \cite{Chowlafcnfield}. The best unconditional result over $\mathbb{Q}$ is a proportion $87.5\%$ of non-vanishing, proven by Soundararajan \cite{S}.

Assuming the GRH, any bound larger than $2/5$ for the support in the one-level density of the family $\mathcal{L}^\pm(X)$ would immediately imply the non-vanishing of a positive proportion of these $L$-functions, see e.g. \cite{ILS}. Now, in \cite{SST} it is unconditionally proven that, for large enough $X$, at least $X^{0.67}$ of the $L$-functions $L(s,\rho_K)$, associated with the cubic fields whose absolute discriminants are bounded by $X$ are non-vanishing at the central point. Moreover, under the GRH, this number is, up to a constant depending on $\epsilon$, improved to $X^{1-\epsilon}$ for any $\epsilon > 0$. Either of these results imply non-vanishing for an infinite, yet possibly vanishing proportion of the $L$-functions in question. By instead applying our Theorem \ref{onelevQ}, standard methods \cite[Eqs. (1.40), (1.46)]{ILS} yield the following conditional extension of the non-vanishing results of \cite{SST}, establishing a positive proportion of non-vanishing. 
\begin{theorem}\label{nonvanishQ}
    Under the same conditions as in Theorem \ref{onelevQ}, at least $75\%$ of the Artin $L$-functions in $\mathcal{L}^\pm$ are non-vanishing at the central point. Specifically,
    \begin{equation*}
        \liminf_{X\to \infty} \frac{1}{\#\mathcal{F}^\pm(X)}\sum_{K\in \mathcal{F}^{\pm}(X)}\mathbf{1}_{L(1/2,\rho_K) \neq 0}\geq 0.75. 
    \end{equation*}
\end{theorem}
\begin{remark}
    As the Dedekind zeta function of an $S_3$-cubic field factors as $\zeta_K(s)=\zeta(s)L(s,\rho_K)$, we have the same lower bound for the proportion of non-vanishing for cubic Dedekind zeta functions.
\end{remark}

\subsection{Results over rational function fields}
Let $q$ be a fixed prime power coprime to $2$ and $3$ and consider the rational function field $\fq(T)$. The fields $\mathbb{Q}$ and $\fq(T)$ are both global fields. Moreover, they are in a sense the simplest such fields and they share many similarities. In particular, one may consider the set of $S_3$-cubic extensions of $\fq(T)$. Such an extension $K$ necessarily satisfies $\abs{\mathrm{Disc}(K)}=q^{2n}$, for some $n\geq 1$. We define, for $X=q^{2n}$, the collection
\begin{equation*}
    \mathcal{F}^{S_\infty}(X) = \{K: \text{$K$ $S_3$-cubic},\,\, \abs{\mathrm{Disc}(K)} = X,\,\, P_\infty \text{ is of type $S_\infty$} \}.
\end{equation*} 
Here, $S_\infty\in \{(111),(21),(3),(1^21),(1^3)\}$ is one of the five splitting types, and $P_\infty$ is the so-called prime at infinity. This non-Archimedean prime has a similar relation to the inclusion $\mathbb{F}_q[T]\subseteq \fq(T)$, as the infinite prime over $\mathbb{Q}$ has to the inclusion $\mathbb{Z}\subseteq \mathbb{Q}$. In particular, specifying the splitting behaviour of $P_\infty$ in a cubic extension of $\fq(T)$ is analogous to fixing the sign of the discriminant $\mathrm{Disc}(K)$ for a cubic extension $K$ of $\mathbb{Q}$.

The counting function of cubic extensions of $\fq(T)$ was first studied by Datskovsky--Wright \cite{DWcubic}, who isolated a main term of size $X$, similar to the number-field case. See also the recent work of Bhargava--Shankar--Wang for a different approach \cite{BSW}. In the function field context, the secondary term was first isolated by Zhao \cite{Zhao} using algebro-geometric methods, with a claimed\footnote{See \cite[Section 1.4]{Kural} for an overview of a gap in Zhao's thesis.} bound for the error term of order $o\big(X^{5/6}\big)$.

Recently, the existence of the secondary term was proven independently by Kural \cite{Kural}, with an error term of quality $\mathcal{O}\left(X^{3/4+\epsilon}\right)$, and by the author \cite{Ahlquist2}, with an error term $\mathcal{O}\left(X^{2/3+\epsilon}\right)$. Now, the arguments used to prove Theorem \ref{splitcondthmQ} generalise to the function field case and improve the argument from \cite{Ahlquist2}. More precisely, let
\begin{equation*}
    N^{S_\infty}(\Sigma;X)=\# \{K: \text{$K$ $S_3$-cubic}, \,\, \abs{\mathrm{Disc}(K)} = X,\,\, P_\infty \text{ of type $S_\infty$},\,\,\, K \text{ of type $\Sigma$} \},
\end{equation*}
where $\Sigma=(\Sigma_P)_P$ is a collection of splitting type specifications at finitely many prime polynomials $P\in \fq[T]$. Then, we have the following theorem.
\begin{theorem}\label{fcnfieldct}
    Let $X=q^{2n}$, with $n\geq 1$. Then\footnote{Note that $C_2^{S_\infty}(\Sigma,n)$ depends on $n\nsmod{3}$, cf. \cite[Theorem 1.1]{Ahlquist2}.},
    \begin{equation*}
        N^{S_\infty}(\Sigma;X) = C_1^{S_\infty}(\Sigma)X+C_2^{S_\infty}(\Sigma,n)X^{5/6} + E(\Sigma)\mathcal{O}_{\epsilon,q}\left(X^{1/2+\epsilon}\right),
    \end{equation*}
    where $C_1^{S_\infty}(\Sigma)$ and $C_2^{S_\infty}(\Sigma,n)$ are as in \cite[Theorem 7.2]{Ahlquist2}. Moreover,
    \begin{equation*}
        E(\Sigma)=\prod_PE(\Sigma_P),
    \end{equation*}
    with
    \begin{equation*}
        E(\Sigma_P)=\begin{cases}
            \abs{P}, &\text{ if $\Sigma_P\in \{(111),(21),(3), (1^21) \}$,}\\
            1, &\text{ if $\Sigma_P=(1^3)$,}
        \end{cases}
    \end{equation*}
    where $\abs{P}=q^{\deg P}$.
\end{theorem}
\begin{remark}
    As in Theorem \ref{splitcondthmQ}, the error term associated with the partially ramified splitting type may be averaged for free.
\end{remark}
The reader may note that the unconditional Theorem \ref{fcnfieldct} is of similar strength as the estimates from Theorem \ref{splitcondthmQ} for the smooth counting function, which are conditional on the  GRH. The reason for this is twofold. First, the GRH is a theorem over function fields, proven by Weil. Second, the fact that the possible sizes of discriminants over $\fq(T)$ are $q$-powers means that counting problems over $\fq(T)$ are essentially smooth counting problems.

We also study the Artin $L$-functions $L(s,\rho_K)$ associated with cubic function field extensions $K$. Here, we have the following unconditional results, generalising the analogous results over $\mathbb{Q}$ and improving upon \cite[Theorem 1.4]{Bucuretal} and the computations of \cite[Section 6]{Ahlquist2}.
\begin{theorem}\label{onelevfq}
    The Katz--Sarnak conjecture for Artin $L$-functions $L(s,\rho_K)$ associated with $\mathcal{F}^{S_\infty}$ holds for $\sigma < 1$. More precisely, for test functions $\phi$ whose Fourier transform is supported in $(-1,1)$, the one-level density converges to the expected symplectic density
    \begin{equation*}
        \widehat{\phi}(0)-\frac{\phi(0)}{2},
    \end{equation*}
    as $X=q^{2n}\to \infty$.
\end{theorem}
Just as over $\mathbb{Q}$, the above theorem suffices to obtain a non-vanishing result.

\begin{theorem}\label{nonvanishfq}
    Let $q$ be as above and let $X=q^{2n}$ for a positive integer $n$. Then, at least $75\%$ of the Artin $L$-functions $L(s,\rho_K)$ associated with the fields in $\mathcal{F}^{S_\infty}$ are non-vanishing at the central point. Specifically,
    \begin{equation*}
        \liminf_{X=q^{2n}\to \infty} \frac{1}{\#\mathcal{F}^{S_\infty}(X)}\sum_{K\in \mathcal{F}^{S_\infty}(X)}\mathbf{1}_{L(1/2,\rho_K) \neq 0}\geq 0.75. 
    \end{equation*}
\end{theorem}

\subsection{Proof overview}
We now provide an overview of the arguments used to prove the above theorems. The proof of Theorem \ref{onelevQ} is the most central argument, as it forms the basis for the proofs of the improved asymptotic formulae for the counting functions of cubic fields. Therefore, we start by describing how one proves this result.

\subsubsection{Studying the one-level density}

As we shall see in Section \ref{sectonelev} studying the one-level density when $\sigma < 1$ is essentially equivalent to studying double averages of the Artin character $\lambda_K(p)$ over $K$ and $p$, where $p$ is as large as $\abs{\mathrm{Disc}(K)}^{1-\epsilon}$. Proving cancellation in such averages amounts to proving that there are approximately half as many pairs $(p,K)$ of primes $p$ and cubic fields $K$, where $p$ splits completely in $K$, compared to the number of pairs $(p,K)$ where $p$ is inert in $K$, see \eqref{artinprimedef}. One way to approach this problem is to use asymptotic formulae for the number of cubic fields $K$, with a prescribed splitting behaviour at some fixed prime. However, such formulae usually have too large error terms in the $p$-aspect to be of use when $p$ is large. Indeed, even the error terms for the smooth counting function in Theorem \ref{splitcondthmQ} would be insufficient for reaching support $\sigma < 1$, whence we use a different approach. The setup is quite similar to that of \cite{BTT}, which builds upon the work of \cite{LDTT}.

First, we use the Levi--Delone--Faddeev correspondence to study maximal cubic rings instead of cubic fields. We drop the maximality condition by using an inclusion-exclusion sieve, and instead study cubic orders which are nonmaximal at all primes dividing some number $n$. One of the previous obstacles in studying the one-level density is that one needs to study such orders with $n$ very large. By using the GRH, we are able to show that the contribution from $n\geq X^\delta$, for any small $\delta >0$, is negligible, whence we need only consider small $n$, see Lemma \ref{lemmaGRH1}. A similar idea is also the key for studying the sharp counting function of cubic fields, with or without the GRH.

Next, we would like to use Shintani zeta functions to study these nonmaximal rings. However, these zeta functions also count rings which are not necessarily orders. In \cite{SST}, this issue is handled by imposing a splitting condition, requiring a single prime to be inert, which excludes all non-orders from being counted. However, as approximately one third of all cubic fields are inert at a specified prime, such an approach would significantly lower the proportion obtained in Theorem \ref{nonvanishQ}. Instead, we use the GRH once again, this time for quadratic Dedekind zeta functions, in order to precisely estimate the contribution from non-orders, see Proposition \ref{redcountprop}.

Having performed these necessary preparations, we are able to use Shintani zeta functions to study the one-level density. After using the functional equation for such zeta functions, our problem essentially reduces to counting cubic orders which are totally ramified at the prime $p$. At this point, one could use the fact that the number of such forms with discriminant bounded by $Y$ is $\ll_\epsilon Y/p^{2-\epsilon}$. In fact, this would allow for support $\sigma < 2/3$ in Theorem \ref{onelevQ}.

Instead of applying such a bound, we extend results of Datskovsky--Wright for counting orders inside cubic fields, see Lemmas \ref{totramsubordct} and \ref{3adiclemma}, in order to restrict ourselves to cubic fields which are totally ramified at $p$. We then fibre over the possible quadratic resolvents such a field can have and use a theorem due to Shankar--Thorne \cite[Theorem 11]{ShankarThorne}, building upon \cite{CohenThorne} and \cite{CohenMorra}, for studying the number of such fields, with arbitrary splitting conditions. In particular, their result allows us to isolate a contribution which essentially does not depend on $p$, and which we can bound quite straightforwardly, and another part which is essentially proportional to $\lambda_K(p)$, such that its average over $p$ is small under the GRH. This suffices to prove Theorem \ref{onelevQ}.

\subsubsection{Studying the counting function of cubic fields}
The first step towards extending the arguments described above to the problem of counting cubic fields is to relate this counting problem to averages of Artin characters. For this, we use a non-maximality sieve combined with a discriminant-reducing identity of Bhargava--Shankar--Tsimerman \cite[Eq. (69)]{BST}. In particular, this identity shows that counting cubic fields is tantamount to counting binary cubic forms $x$, weighted by the number of roots $\eta_a(x)$ of $x$ modulo $a$, as $a$ varies. More specifically, one considers such averages twisted by $\mu(a)$. 

The first key observation is the fact that, for $x$ which are not zero modulo $p$, $\eta_p(x)=\lambda_p(x)+1$, where $\lambda_p(x)$ is defined analogously to \eqref{artinprimedef}, where the splitting type of $x$ now refers to the splitting of $x$ modulo $p$. In particular, if $x$ is a maximal order corresponding to the cubic field $K$, then $\lambda_p(x)=\lambda_K(p)$. Next, we rewrite our sieve in terms of averages over $\lambda_d(x)$, twisted by $\mu(d)$, with $d$ squarefree. Such averages are then closely related to partial sums associated with $1/L(s,\rho_K)$. In particular, assuming that one has control over the zeros of $L(s,\rho_K)$, one can show cancellation within such sums, with the GRH implying cancellation with a square root error term. Said differently, we obtain cancellation between coefficients of different Shintani zeta functions. 

The above observation essentially suffices to prove our asymptotics for the sharp counting function. Indeed, we split the range of summation into two ranges, depending on the size of $d$. For the range where $d$ is small, we simply apply \cite[Theorem 3.1]{BTT}, after rewriting this part of the summation in terms of non-negative counting functions. Our improvements for the error term in the sharp case instead come from the large $d$-range. In the conditional case, we use the GRH to control averages of $\mu(d)\lambda_d(x)$ when $x$ is a cubic order, and in the unconditional case, we instead use a zero-density result of Thorner--Zaman \cite{TZ}. The contribution from reducible $x$ is handled separately. This is carried out in Sections \ref{sharpGRHsect} and \ref{sharpnoGRHsect}. 

In Section \ref{smoothctallfieldsect} we handle the smooth case. Here, we can do better compared to the sharp case, as we are able to use the GRH when $d$ is small as well. For small $d$, we use Mellin inversion and Shintani zeta functions to study sums over $x$, weighted by $\mu(d)\lambda_d(x)$. After shifting the contour to the left, we apply the functional equation and are able to use \cite[Theorem 11]{ShankarThorne}, as when studying the one-level density, to show substantial cancellation when summing over $d$. We then obtain the desired square-root error term in Theorem \ref{smoothctthm}.

Finally, we remark that Theorem \ref{splitcondthmQ} is obtained by modifying the proof for counting all fields, similar to how \cite[Theorem 1.3]{BTT} is obtained by modifying the proof of \cite[Theorem 1.1]{BTT}. As the proof of \cite[Theorem 1.3]{BTT} is given in great detail, we provide only a sketch of the proof of Theorem \ref{splitcondthmQ}, in Section \ref{splitcondsupsect}, focusing on the novel parts of our argument.

\subsubsection{Studying cubic function field extensions}
Our arguments over $\mathbb{Q}$ carry over without significant difficulties to the function field case, whence we only provide sketches of the proofs. In Section \ref{ffields}, we describe how Theorem \ref{onelevfq} is proved by generalising the proof of Theorem \ref{onelevQ}. The most significant difference is that we circumvent the use of Shintani zeta functions and instead build upon the methods developed by the author in \cite{Ahlquist2}.

In Section \ref{fcnfldctsect}, we describe how to count cubic function fields. The methods generalise from the smooth case over $\mathbb{Q}$. Now, in \cite{Ahlquist2}, the use of the functional equation for the Shintani zeta function over $\mathbb{Q}$ is replaced by the use of certain dual sums over dual forms. One may separate this contribution into a degenerate, and a non-degenerate contribution. Unfortunately, in \cite{Ahlquist2}, the degenerate contribution is handled in a way which prevents an improvement of the error term past $\mathcal{O}_\epsilon\left(X^{2/3+\epsilon}\right)$. This issue is rectified in Section \ref{degencontribsect} using an explicit parametrisation of the degenerate forms. The rest of the argument over $\mathbb{Q}$ generalises mutatis mutandis.

\subsection{Prospects for improvement}
As mentioned before the statement of Theorem \ref{smoothctthm}, numerical experiments indicate that the error term in the sharp counting function of cubic field extensions of $\mathbb{Q}$ is of square-root size. If this is the case, then the error terms in Theorem \ref{uncondsharpctthm}, Theorem \ref{quasiapplthm} and Theorem \ref{sharpgrhthm} are far from optimal.


As explained in the previous section, the reason why we are able to obtain a square-root error term in the smooth case is because we are able to exploit cancellation between coefficients of different Shintani zeta functions. However, when studying the sharp counting function, we use \cite[Theorem 3.1]{BTT}, which requires the coefficients of the zeta functions for which one seeks to obtain partial sum estimates to be nonnegative. Hence, we do not obtain any cancellation. We have not been able to modify Landau's method, used in \cite{LDTT}, which \cite[Theorem 3.1]{BTT} builds upon, to allow for coefficients which are not necessarily nonnegative. Such a modification would likely improve the error term one obtains when studying the sharp counting function, both in the unconditional and the conditional case.

Finally, we remark that even if one is able to make the necessary modifications to Landau's method, obtaining a square-root error term in the sharp counting function, even under the GRH, seems out of reach at the present moment. Indeed, even for the counting function of all cubic forms, the current best error term is of size $\mathcal{O}\left(X^{3/5}\right)$, see e.g. \cite[Theorem 3.1]{BTT}, with $m=1$.

\subsection{Conventions}
Given a set $M$ and functions $f,g: M\to \mathbb{C}$, we write $f = \mathcal{O}(g)$ if there is a constant $C$ such that $\abs{f}\leq C\abs{g}$ on $M$. We will also use the notation $f\ll g$ if $f=\mathcal{O}(g)$. If $f\ll g$ and $g\ll f$, then we write $f\asymp g$. The dependence of the constant $C$ on a variable is indicated through a subscript, unless such a dependence clearly holds, in which case we often leave out the subscript. In particular, when studying the one-level density, all constants are allowed to depend on the test function $\phi$, and thus also on the support $\sigma$. When studying the smooth counting function of cubic fields, all constants are allowed to depend on the smooth function $h$. We will write $f = o(1)$ if $f(X)\to 0$ as $X\to \infty$. Moreover, we make use of a somewhat non-standard notation and write $f\sim g$ if $f-g = o(1)$. 

Next, for a function $\psi:\mathbb{R}\to \mathbb{C}$, we denote its Fourier transform by $\widehat{\psi}$. Here, we use the convention that
\begin{equation*}
    \widehat{\psi}(u)=\int_{\mathbb{R}}\psi(x)e^{-2\pi i x u}dx.
\end{equation*}
If $G$ is a finite abelian group, and $f: G\to \mathbb{C}$, then $\widehat{f}$ is its finite Fourier transform, defined by
\begin{equation*}
    \widehat{f}(\chi) = \frac{1}{\#G}\sum_{g\in G}f(g)\chi(g),
\end{equation*}
for a character $\chi$ of $G$. With this choice of definition, one may extract $f$ from its Fourier transform through the identity
\begin{equation*}
    f(g)=\sum_{\chi\in \widehat{G}}\widehat{f}(\chi)\overline{\chi}(g).
\end{equation*}

Throughout, we shall make use of several arithmetic functions. We denote by $\sigma(n)$ the sum of divisors of $n$, and by $\omega(n)$ the number of prime factors of $n$. We let $\mathrm{sqf}(n)$ denote the squarefree kernel of $n$, i.e. the product of all primes dividing $n$. 

\subsection{Acknowledgements}
I would like to thank Anders Södergren for valuable comments and for encouraging me to extend the scope of a previous draft of this paper, to include the case of cubic extensions of $\mathbb{Q}$.
\subsection{AI statement}
ChatGPT was used to perform literature searches during the research process. Moreover, after the first version of this preprint was made publicly available on arXiv, ChatGPT $5.6$ was used to proofread the text. In particular, this led to the identification of a technical error in a previous proof of Proposition \ref{zerodenspropcube}.

\section{Preliminaries}
Here, and until Section \ref{ffields}, we restrict ourselves to the case of number fields for simplicity.

Let $K$ be an $S_3$-cubic extension of $\mathbb{Q}$. For $\Re(s) > 1$, the Dedekind zeta function associated with $K$ is given by the Euler product
\begin{equation}\label{dedekindef}
    \prod_p\prod_{\mathcal{P}\mid p} \left(1-\abs{\mathcal{P}}^{-s}\right)^{-1},
\end{equation}
where $\mathcal{P}$ is a prime ideal of $K$, and $\abs{\mathcal{P}}:=\abs{\mathcal{O}_K/\mathcal{P}}$. By general theory, see e.g. \cite[Ch. VII.10]{Neukirch}, the Dedekind zeta function splits into a product
\begin{equation*}
    \zeta_K(s)=\zeta(s)L(s,\rho_K),
\end{equation*}
where $L(s,\rho_K)$ is the Artin $L$-function associated with the unique irreducible two-dimensional representation of $S_3$. As the Artin conjecture is known for these $L$-functions, $L(s,\rho_K)$ is an entire function.

Associated with each $L(s,\rho_K)$ is an Artin character $\lambda_K$, defined by
\begin{equation*}
    L(s,\rho_K)=\sum_{n\geq 1}\frac{\lambda_K(n)}{n^s}.
\end{equation*}
We will only need to study this character at squarefree numbers $n$. Now, $\lambda_K(n)$ is multiplicative in $n$, and at primes, we see from \eqref{dedekindef} that
\begin{equation}\label{artinprimedef}
    \lambda_K(p)=\begin{cases}
        2, &\text{ for $p$ of type $(111)$ in $K$,}\\
        0, &\text{ for $p$ of type $(21)$ in $K$,}\\
        -1, &\text{ for $p$ of type $(3)$ in $K$,}\\
        1, &\text{ for $p$ of type $(1^21)$ in $K$,}\\
        0, &\text{ for $p$ of type $(1^3)$ in $K$.}
    \end{cases}
\end{equation}
This agrees with the definition of the character $a_K(p)$ at primes $p$, where $a_K$ is defined by the relation
\begin{equation*}
    -\frac{L'(s,\rho_K)}{L(s,\rho_K)} = \sum_{n\geq 1} \frac{\Lambda(n)a_K(n)}{n^s},
\end{equation*}
where $\Lambda(n)$ is the von-Mangoldt function. We will also be interested in sums of characters $\lambda_K(n)$ twisted by $\mu(n)$. For squarefree $n$, one also concludes from \eqref{dedekindef} that $\mu(n)\lambda_K(n) = \mu_K(n)$, where
\begin{equation*}
    \frac{1}{L(s,\rho_K)}=\sum_{n\geq 1}\frac{\mu_K(n)}{n^s}.
\end{equation*}

Any cubic field $K$ contains a unique distinguished ring, namely its ring of integers, $\mathcal{O}_K$. The ring $\mathcal{O}_K$ is free over $\mathbb{Z}$ of rank $3$. In general a ring $A$, which is free over $\mathbb{Z}$ of rank $3$, is known as a cubic ring. We say that $A$ is a cubic order if it is a domain. The ring $\mathcal{O}_K$ is a maximal order, in the sense that it contains every other cubic ring contained in $K$. Every cubic ring $A$, whose discriminant is nonzero, is contained inside a maximal cubic ring $A'$. Moreover, $A'$ is the ring of integers of a cubic field if and only if $A$ is an order. In general, a maximal cubic ring is either isomorphic to $\mathcal{O}_K$ for a cubic $K$, $\mathcal{O}_{L}\times \mathbb{Z}$ for a quadratic $L$, or isomorphic to $\mathbb{Z}^3$.

Cubic rings are closely related to (binary) cubic forms, i.e. homogeneous polynomials of the form
\begin{equation*}
    f(u,v)=au^3+bu^2v+cuv^2+dv^3.
\end{equation*}
We denote the space of such forms with coefficients in some ring $A$ by $V(A)$. On the space $V(A)$, one may define a twisted $\GLtwo(A)$-action by letting $\gamma\in \GLtwo(A)$ act by
\begin{equation*}
    (\gamma f)(u,v) = \frac{1}{\det \gamma}f\big((u,v)\gamma\big),
\end{equation*}
see also \cite[Section 2]{TT1} for an explicit description of how $\gamma$ acts on the coefficients of $f$. The Delone--Faddeev correspondence, see e.g. \cite[Section 2]{BST} for a detailed description, asserts that there is a discriminant-preserving bijection between the set of orbits of this group action and the set of isomorphism classes of cubic rings. Moreover, the ring corresponding to a cubic form is an order if and only if the form is irreducible. 

Given a cubic form $f\in V(\mathbb{Z})$, one may consider its reduction modulo a prime $p$. If not all coefficients are divisible by $p$, the reduced form will either have three distinct roots, exactly one root, no roots, a double root, or a triple root in $\mathbb{F}_p$. We say that $f$ has splitting type $(111),(21),(3),(1^21)$, or $(1^3)$, respectively. When $f$ corresponds to a maximal order $\mathcal{O}_K$, the splitting type of $f$ agrees with the splitting type of $p$ in $\mathcal{O}_K$.

We remark that maximality is a local condition. Specifically, we say that $A$ is maximal at $p$ if it is not contained in another cubic ring $A'$ with index divisible by $p$. Using a criterion due to Davenport and Heilbronn, see e.g. \cite[Section 3]{BST}, one can check maximality at $p$ by reducing the coefficients modulo $p^2$.

In order to study orbits of cubic forms, one often makes use of Shintani zeta functions associated with the space of binary cubic forms, first studied by Shintani \cite{Shintani}. The simplest such Shintani zeta functions have the form
\begin{equation*}
    \xi^{\pm}(s)=\sum_\sumstack{f\in \GLtwo(\mathbb{Z})\setminus V(\mathbb{Z})\\ \pm \mathrm{Disc}(f) > 0} \frac{\abs{\mathrm{Disc}(f)}^{-s}}{\#\mathrm{Stab}(f)}.
\end{equation*}
Here, we remark that under the Delone--Faddeev correspondence, $\mathrm{Stab}(f)\simeq \mathrm{Aut}(A)$ for the ring $A$ corresponding to $f$.

As described in \cite[Section 2.4]{BTT}, given a $\GLtwo(\mathbb{Z}/m\mathbb{Z})$-invariant function $\Phi_m: V(\mathbb{Z}/m\mathbb{Z})\to \mathbb{C}$, one may form the Shintani zeta function
\begin{equation}\label{shintanigeneral}
    \xi^{\pm}(s,\Phi_m)=\sum_\sumstack{f\in \GLtwo(\mathbb{Z})\setminus V(\mathbb{Z})\\ \pm \mathrm{Disc}(f) > 0} \frac{\Phi_m(f)\abs{\mathrm{Disc}(f)}^{-s}}{\#\mathrm{Stab}(f)} =: \sum_{n\geq 1}\frac{a^{\pm}(\Phi_m,n)}{n^s}.
\end{equation}
As worked out in \cite{DW, Wright, TT1}, these zeta functions have a meromorphic continuation to $\mathbb{C}$ with at most simple poles at $s=1$ and $s=5/6$. Moreover, the values of the residues at these poles are known explicitly.

In order to state the functional equation, we need the concept of a dual zeta function. First, the dual lattice $V(\mathbb{Z})^*$ embeds as the subspace of $V(\mathbb{Z})$ consisting of forms with the middle two coefficients divisible by three, i.e. integral binary cubic forms whose reduction modulo $3$ is of type $(1^3)$, or $(0)$. The discriminant of a dual form is then defined using this embedding.

The Fourier transform of $\Phi_m$ is defined by
\begin{equation*}
    \widehat{\Phi_m}(f)=\frac{1}{m^4}\sum_{g\in V(\mathbb{Z}/m\mathbb{Z})}\Phi_m(g)\mathrm{exp}\left(2\pi i\frac{[f,g]}{m}\right),
\end{equation*}
where
\begin{equation*}
    [f,g]=f_1g_1+f_2g_2/3+f_3g_3/3+f_4g_4,
\end{equation*}
and the $f_i,g_i$ are the coefficients of $f$ and $g$, see \cite[Section 3]{TTOrb}. We lift $\widehat{\Phi_m}$ to a function defined on $V(\mathbb{Z})$, which is zero outside of $V(\mathbb{Z})^*\subseteq V(\mathbb{Z})$. The dual zeta function is then defined by
\begin{equation*}
    \widehat{\xi^{\pm}}(s,\Phi_m)=\xi^{\pm}(s,\widehat{\Phi}_m).
\end{equation*}
Then, see e.g. \cite[Theorem 2.4]{BTT} for our choice of normalisation, we have with 
\begin{equation*}
\begin{split}
   & \Delta^+(s)=\left(\frac{2^43^6}{\pi^4}\right)^{s/2}\Gamma\left(\frac{s}{2}\right)\Gamma\left(\frac{s}{2}+\frac{1}{2}\right)\Gamma\left(\frac{s}{2}-\frac{1}{12}\right)\Gamma\left(\frac{s}{2}+\frac{1}{12}\right),\\
    &\Delta^-(s)=\left(\frac{2^43^6}{\pi^4}\right)^{s/2}\Gamma\left(\frac{s}{2}\right)\Gamma\left(\frac{s}{2}+\frac{1}{2}\right)\Gamma\left(\frac{s}{2}+\frac{5}{12}\right)\Gamma\left(\frac{s}{2}+\frac{7}{12}\right),
\end{split}
\end{equation*}
and
\begin{equation*}
    T=\begin{pmatrix}
        \sqrt{3} & 1 \\ \sqrt{3} & -1
    \end{pmatrix},\,\,\,\,
\end{equation*}
that the functional equation
\begin{equation}\label{genfcnal}
    \Delta(1-s)\cdot T\cdot\xi(1-s,\Phi_m) = m^{4s}\begin{pmatrix}
        3 &0\\0 & -3
    \end{pmatrix}\cdot \Delta(s)\cdot T\cdot \xi(s,\widehat{\Phi_m})
\end{equation}
holds, where
\begin{equation*}
\begin{split}
    \xi(s,\Phi_m) = \begin{pmatrix}
        \xi^+(s,\Phi_m)\\ \xi^-(s,\Phi_m)
    \end{pmatrix}, \,\,\,\, \text{ and }\,\,\,\Delta(s) = \begin{pmatrix}
        \Delta^+(s) & 0\\ 0 & \Delta^-(s)
    \end{pmatrix}.
\end{split}
\end{equation*}

\section{The one-level density}\label{sectonelev}
Let $\mathcal{F}^{\pm}(X)$ be the family of non-Galois cubic fields, whose discriminants have absolute value bounded by $X$ and sign $\pm$. The one-level density of the Artin $L$-functions associated with the fields in $\mathcal{F}^{\pm}(X)$ is then the average
\begin{equation*}
    \frac{1}{\#\mathcal{F}^{\pm}(X)}\sum_{K\in \mathcal{F}^{\pm}(X)}D_\phi(K)\,\,\,\,\text{ with }\,\,\,\, D_\phi(K)=\sum_{\gamma_K}\phi\left(\frac{\log X}{2\pi}\gamma_K\right),
    \end{equation*}
where $\gamma_K$ ranges over the ordinates, with multiplicity, of the zeros of $L(s,\rho_K)$. The function $\phi$ is a real even Schwartz function. 

The one-level density is often studied using the so-called explicit formula. For this family, the explicit formula \cite[Lemma 3.1]{CFLS}, and a change of variables, implies that
\begin{equation*}
    D_\phi(K) = \frac{\widehat{\phi}(0)}{L}\log\abs{D_K}-\frac{2}{L}\sum_{p}\sum_{e=1}^\infty\frac{\log p}{p^{e/2}}\widehat{\phi}\left(\frac{e\log p}{L}\right)a_K(p^e) + o(1).
\end{equation*}
Here, $L=\log X$, $D_K = \mathrm{Disc}(K)$ and $a_K(n)$ is the $n$th coefficient in the Dirichlet series defining $-L'(s,\rho_K)/L(s,\rho_K)$. More precisely, 
\begin{equation*}
    a_K(p^e) = \begin{cases}
        2, &\text{ if $p$ splits according to $(111)$ in $K$,}\\
        2\delta_{2\mid e}, &\text{ if $p$ splits according to $(21)$ in $K$,}\\
        3\delta_{3\mid e}-1, &\text{ if $p$ splits according to $(3)$ in $K$,}\\
        1, &\text{ if $p$ splits according to $(1^21)$ in $K$,}\\
        0, &\text{ if $p$ splits according to $(1^3)$ in $K$.},\\
    \end{cases}
\end{equation*}
see e.g. \cite[Section 1]{CFLS}.

In order to compute the one-level density, we must average $D_\phi(K)$ over $K$. A straightforward application of partial summation proves that
\begin{equation*}
    \frac{\widehat{\phi}(0)}{L}\cdot\frac{1}{\#\mathcal{F}^{\pm}(X)}\sum_{K\in \mathcal{F}^{\pm}(X)}\log \abs{D_K} = \widehat{\phi}(0) +o(1).
\end{equation*}
Hence, without any condition on the support $\sigma<\infty$ of $\widehat{\phi}$, we have that
\begin{equation}\label{onelevexpansion}
    \frac{1}{\#\mathcal{F}^{\pm}(X)}\sum_{K\in \mathcal{F}^{\pm}(X)}D_\phi(K)= \widehat{\phi}(0)-\frac{2}{L\#\mathcal{F}^{\pm}(X)}\sum_{K\in \mathcal{F}^{\pm}(X)}\sum_{p}\sum_{e=1}^\infty\frac{\log p}{p^{e/2}}\widehat{\phi}\left(\frac{e\log p}{L}\right)a_K(p^e) + o(1),
\end{equation}
as $X\to \infty$. It is immediate that the terms $e\geq 3$ may be absorbed in the error term above. Moreover, at the cost of a negligible error, we may replace $a_K$ with a function that equals zero on the ramified splitting types. Indeed, the ramified primes are precisely those which divide the discriminant, and for any $D$ with $\abs{D}\leq X$, we have that
\begin{equation}\label{rambound}
\begin{split}
    \sum_{p\mid D}&\frac{\log p}{p^{1/2}} = \sum_\sumstack{p\mid D\\ p\leq \log X}\frac{\log p}{p^{1/2}}+\sum_\sumstack{p\mid D\\ p> \log X}\frac{\log p}{p^{1/2}}\\&\ll \log\log X\sum_\sumstack{ p\leq \log X}\frac{1}{p^{1/2}} +\frac{1}{(\log X)^{1/2}}\sum_\sumstack{p\mid D}\log p \ll (\log X)^{1/2}\log\log X + (\log X)^{1/2},
\end{split}
\end{equation}
which is $o(\log X)$. Hence, the total contribution of the ramified primes to the one-level density is $o(1)$.

For the terms with $e=2$ in \eqref{onelevexpansion}, we will use asymptotic formulae for the counting functions of the sets
\begin{equation*}
    \mathcal{F}^{\pm}_{p,T}(X) = \{K\in \mathcal{F}^{\pm}(X): p \text{ has splitting type $T$}\},
\end{equation*}
proved in \cite[Theorem 1.3]{BTT}. Then, similarly to \cite[Lemma 3.4]{CFLS}, one shows that
\begin{equation*}
    -\frac{2}{L\#\mathcal{F}^{\pm}(X)}\sum_{K\in \mathcal{F}^{\pm}(X)}\sum_{p}\frac{\log p}{p}\widehat{\phi}\left(\frac{2\log p}{L}\right)a_K(p^2) = -\frac{\phi(0)}{2}+o(1),
\end{equation*}
assuming that $\sigma < 1$. We remark here that the condition $\sigma < 1$ is at this point not a serious obstacle. Indeed, if instead of using bounds for $\#\mathcal{F}^{\pm}_{p,T}(X)$ one uses the GRH for quadratic Dirichlet $L$-functions, as well as the GRH for $L(s,\rho_K)$, then one can bound the error term above as $o(1)$ for any finite $\sigma$.

In order to prove the Katz--Sarnak prediction for $\sigma < 1$, it remains to show that
\begin{equation*}
    \frac{2}{L\#\mathcal{F}^{\pm}(X)}\sum_{K\in \mathcal{F}^{\pm}(X)}\sum_{p}\frac{\log p}{p^{1/2}}\widehat{\phi}\left(\frac{\log p}{L}\right)\theta_{K}(p) = o(1),
\end{equation*}
where 
\begin{equation*}
    \theta_K(p)= \begin{cases}
        2, &\text{ if $p$ splits according to $(111)$ in $K$,}\\
        -1, &\text{ if $p$ splits according to $(3)$ in $K$,}\\
        0, & \text{ else.}
    \end{cases}
\end{equation*}
The part of the sum involving small $p\leq X^{2/7-\epsilon}$ has been shown to be $o(1)$ in \cite{CFLS}. Hence, to make some parts of our argument simpler, we may restrict ourselves to the tail sum, $p\geq X^{1/4},$ say. We interchange the order of summation to rewrite the above as
\begin{equation}\label{bound2win}
    \frac{2}{L\#\mathcal{F}^{\pm}(X)}\sum_{p\geq X^{1/4}}\frac{\log p}{p^{1/2}}\widehat{\phi}\left(\frac{\log p}{L}\right)\sum_{K\in \mathcal{F}^{\pm}(X)}\theta_{K}(p).
\end{equation}

\subsection{Sieving for maximality}
By the Delone--Faddeev correspondence, evaluating the sum
\begin{equation*}
    \sum_{K\in \mathcal{F}^{\pm}(X)}\theta_{K}(p)
\end{equation*}
is equivalent to counting maximal cubic orders with weight $\theta_K(p)$ up to isomorphism. By an inclusion-exclusion argument, we have that
\begin{equation}\label{fieldsieve}
    \sum_{K\in \mathcal{F}^{\pm}(X)}\theta_{K}(p)= \sum_{n=1}^\infty \mu(n)\sum_{f\in \mathcal{F}^{\pm}(W_n;X)}\theta_{K_f}(p):= \sum_{n\geq 1}\mu(n)N^{\pm}(W_n, \theta_p;X),
\end{equation}
where $\mathcal{F}^{\pm}(W_n;X)$ is the collection of binary cubic forms which are nonmaximal at every prime dividing $n$, with discriminant of sign $\pm$ and size bounded by $X$, whose corresponding ring is contained in an $S_3$-cubic field $K_f$. Now, standard estimates, see e.g. \cite[Proposition 4.6]{BTT}, allow us to bound
\begin{equation}\label{nonmaxbd}
    \#\mathcal{F}^{\pm}(W_n;X)\ll_\epsilon \frac{X}{n^{2-\epsilon}}.
\end{equation}
We will use this bound in the next section in order to bound the terms with $n$ large.

\subsection{Bypassing the sieve}
The numbers $N^{\pm}(W_n, \theta_p;X)$ become harder to estimate accurately the larger $n$ is. In this section, by way of the GRH, we manage to restrict ourselves to considering $n$ of very small size $\ll X^{\delta}$. 

To be more precise, our goal in this section is to bound the contribution of
\begin{equation*}
    \sum_{n> X^{\delta}}\mu(n)N^{\pm}(W_n, \theta_p;X),
\end{equation*}
to \eqref{bound2win}, for an arbitrary, small $\delta > 0$. The total contribution to \eqref{bound2win} from such terms equals
\begin{equation}\label{firsttail}
    \frac{2}{L\#\mathcal{F}^{\pm}(X)}\sum_{p\geq X^{1/4}}\frac{\log p}{p^{1/2}}\widehat{\phi}\left(\frac{\log p}{L}\right)\sum_{n> X^{\delta}}\mu(n)\sum_{f\in \mathcal{F}^{\pm}(W_n;X)}\theta_{K_f}(p).
\end{equation}
Interchanging the order of summation gives us an expression
\begin{equation*}
    \frac{2}{L\#\mathcal{F}^{\pm}(X)}\sum_{n> X^{\delta}}\mu(n)\sum_{f\in \mathcal{F}^{\pm}(W_n;X)}\sum_{p\geq X^{1/4}}\frac{\log p}{p^{1/2}}\widehat{\phi}\left(\frac{\log p}{L}\right)\theta_{K_f}(p),
\end{equation*}
which we now bound using the following lemma.
\begin{lemma}\label{lemmaGRH1}
    Assume the GRH for Dedekind zeta functions associated with non-Galois cubic fields. Then,
    \begin{equation}\label{lemmasumbdptwise}
        \frac{2}{L\#\mathcal{F}^{\pm}(X)}\sum_{n> X^{\delta}}\mu(n)\sum_{f\in \mathcal{F}^{\pm}(W_n;X)}\frac{1}{\#\mathrm{Stab}(f)}\sum_{p\geq X^{1/4}}\frac{\log p}{p^{1/2}}\widehat{\phi}\left(\frac{\log p}{L}\right)\theta_{K_f}(p)=o(1).
    \end{equation}
\end{lemma}
\begin{proof}
    Assuming the GRH, a strong form of Chebotarev's density theorem \cite[Eq. (5.109)]{IK} implies that for sufficiently large $x$,
    \begin{equation*}
        \sum_{p\leq x} \theta_{K_f}(p)\log p \ll x^{1/2}\log\big(x\abs{\mathrm{Disc}(f)}\big).
    \end{equation*}
    Denoting the left-hand side above by $S(x)$, summation by parts shows that
    \begin{equation*}
        \sum_{p\geq X^{1/4}}\frac{\log p}{p^{1/2}}\widehat{\phi}\left(\frac{\log p}{L}\right)\theta_{K_f}(p) =-\frac{S(X^{1/4})}{X^{1/8}}\widehat{\phi}\left(\frac{\log X^{1/4}}{L}\right) -\int_{X^{1/4}}^\infty S(t)\frac{d}{dt}\left(\frac{1}{t^{1/2}}\widehat{\phi}\left(\frac{\log t}{L}\right)\right) dt \ll X^{\epsilon},
    \end{equation*}
    as $\widehat{\phi}$ has compact support. Combined with the bound \eqref{nonmaxbd}, this gives the desired result.
\end{proof}

\subsection{Landau's method and Shintani zeta functions}
It remains to bound
\begin{equation}\label{smallnsum}
    \frac{2}{L\#\mathcal{F}^{\pm}(X)}\sum_{p\geq X^{1/4}}\frac{\log p}{p^{1/2}}\widehat{\phi}\left(\frac{\log p}{L}\right)\sum_{n\leq X^{\delta}}\mu(n)\sum_{f\in \mathcal{F}^{\pm}(W_n;X)}\theta_{K_f}(p).
\end{equation}
Recall that if $f$ is the cubic form corresponding to a maximal order $\mathcal{O}_K$, then the splitting behaviour of $p$ in $K$ can be deduced from the structure of the roots of $f\nsmod{p}$. For example, $K$ is of type $(111)$ at $p$ if $f\nsmod{p}$ factors into three linear factors, and one has similar statements for the other splitting types. Using this observation, we may define $\theta_f(p)$ for an arbitrary cubic form, in a way that agrees with $\theta_{K_f}(p)$ for maximal orders. If $f$ is maximal at $p$, then a localisation argument shows that $\theta_f(p)$ agrees with $\theta_{K_f}(p)$. Now, potential terms $f$ which are nonmaximal at $p$ satisfy $p^2\mid \mathrm{Disc}(f)$, whence the contribution of such terms to the one-level density can be discarded at the cost of an error $o(1)$ to \eqref{bound2win}, by an argument similar to \eqref{rambound}. Hence, we may replace \eqref{smallnsum} by
\begin{equation}\label{bound2winsmalln}
    \frac{2}{L\#\mathcal{F}^{\pm}(X)}\sum_{p\geq X^{1/4}}\frac{\log p}{p^{1/2}}\widehat{\phi}\left(\frac{\log p}{L}\right)\sum_{n\leq X^\delta} \mu(n)\sum_{f\in \mathcal{F}^{\pm}(W_n;X)}\theta_{f}(p).
\end{equation}

The innermost sum above equals
\begin{equation*}
    \Tilde{A}_0(\theta_p\otimes \Psi_{n}; X) :=\sum_{m\leq X}a^{\pm}_1(\theta_p\otimes\Psi_{n},m).
\end{equation*}
Here, $\Psi_{n}$ is the indicator function of forms nonmaximal at $n$ and 
\begin{equation*}
    a^{\pm}_1(\theta_p\otimes\Psi_{n},m) := \sum_\sumstack{f \leftrightarrow \text{ $S_3$-cubic order} \\ \pm \mathrm{Disc}(f)=m} \theta_p(f)\Psi_{n}(f).
\end{equation*}
By an $S_3$-cubic order, we mean a suborder of the ring of integers of an $S_3$-cubic field. Except for the restriction that $f$ correspond to an $S_3$-cubic order, this is the $m$th coefficient of the Shintani zeta function $\xi^{\pm}(s,\theta_p\otimes \Psi_{n})$.

In order to study these partial sums, we use a uniform version of a classical method due to Landau, following \cite{LDTT}. As we shall need to carefully utilise cancellation coming from $\theta_p$, we cannot cite the results of \cite{LDTT} and use them out of the box. Instead, we follow their general setup, and make deviations suitable for our specific problem.

The starting point of Landau's method is to consider the smoothed partial sums
\begin{equation*}
    \Tilde{A}_{k}(\theta_p\otimes \Psi_{n}; X) :=\frac{1}{\Gamma(k+1)}\sum_{m\leq X}a^{\pm}_1(\theta_p\otimes\Psi_{n},m)(X-m)^k,
\end{equation*}
for integers $k\geq 0$. One then extracts information about $\Tilde{A_0}$ using finite differencing. Given any function $F:\mathbb{R}\to \mathbb{C}$, and $y\geq 1$, we define the finite-differencing operator $\Delta^k_y(X)$ by
\begin{equation*}
    \Delta^k_yF(X) = \sum_{\nu=0}^k(-1)^{k-\nu}\binom{k}{\nu}F(X+\nu y).
\end{equation*}
We shall make frequent use of the trivial bound
\begin{equation}\label{trivialfindiff}
    \Delta^k_yF(X) \ll_k \sup_{t\in [X,X+ky]} \abs{F(t)}.
\end{equation}
In the sequel, all implicit constants are allowed to depend on $k$, without indicating this. A short calculation, see the proof of \cite[Lemma 8]{LDTT}, shows that
\begin{equation*}
    \Delta^k_y \Tilde{A}_k(\theta_p\otimes \Psi_{n}; X) = \Tilde{A}_0(\theta_p\otimes \Psi_{n}; X)y^k+ \frac{1}{\Gamma(k+1)}\sum_{\nu=0}^k(-1)^{k-\nu}\binom{k}{\nu}\sum_{X<m\leq X+\nu y}a^{\pm}_1(\theta_p\otimes\Psi_{n},m)(X+\nu y-m)^k,
\end{equation*}
so that one may recover $\Tilde{A_0}$ from $\Tilde{A_k}$ assuming that the tail sum can be controlled.

Using the relation above, we may write
\begin{equation*}
    \sum_{f\in \mathcal{F}^{\pm}(W_n;X)}\theta_{f}(p) = y^{-k}\Delta^k_y \Tilde{A}_k(\theta_p\otimes \Psi_{n}; X)-\frac{y^{-k}}{\Gamma(k+1)}\sum_{\nu=0}^k(-1)^{k-\nu}\binom{k}{\nu}\sum_{X<m\leq X+\nu y}a^{\pm}_1(\theta_p\otimes\Psi_{n},m)(X+\nu y-m)^k.
\end{equation*}
We now use the GRH once again to bound the contribution from the tail sum above to \eqref{bound2winsmalln}. Specifically, we prove the following lemma.
\begin{lemma}\label{GRHlemma2}
    Assume the GRH for non-Galois cubic fields. Then, for $y\ll X^{1-\eta}$, with $\eta > \delta$ fixed and small, we have that
    \begin{equation*}
        \frac{2}{L\#\mathcal{F}^{\pm}(X)}\sum_{p\geq X^{1/4}}\frac{\log p}{p^{1/2}}\widehat{\phi}\left(\frac{\log p}{L}\right)\sum_{n\leq X^\delta} \mu(n)\frac{y^{-k}}{\Gamma(k+1)}\sum_{\nu=0}^k(-1)^{k-\nu}\binom{k}{\nu}\sum_{X<m\leq X+\nu y}a^{\pm}_1(\theta_p\otimes\Psi_{n},m)(X+\nu y-m)^k = o(1).
    \end{equation*}
\end{lemma}
\begin{proof}
    Interchanging the order of summation in the expression above yields
\begin{equation*}
    \frac{2}{L\#\mathcal{F}^{\pm}(X)}\sum_{n\leq X^\delta} \mu(n)\frac{y^{-k}}{\Gamma(k+1)}\sum_{\nu=0}^k(-1)^{k-\nu}\binom{k}{\nu}\sum_{X<m\leq X+\nu y}(X+\nu y-m)^k\sum_{p\geq X^{1/4}}\frac{\log p}{p^{1/2}}\widehat{\phi}\left(\frac{\log p}{L}\right)a^{\pm}_1(\theta_p\otimes\Psi_{n},m).
\end{equation*}
Now, 
\begin{equation*}
    \sum_{p\geq X^{1/4}}\frac{\log p}{p^{1/2}}\widehat{\phi}\left(\frac{\log p}{L}\right)a^{\pm}_1(\theta_p\otimes\Psi_{n},m) = \sum_\sumstack{f \leftrightarrow \text{ $S_3$-cubic order} \\ \pm \mathrm{Disc}(f)=m}\Psi_{n}(f)\sum_{p\geq X^{1/4}}\frac{\log p}{p^{1/2}}\widehat{\phi}\left(\frac{\log p}{L}\right) \theta_p(f).
\end{equation*}
Similar to the proof of Lemma \ref{lemmaGRH1}, one may bound the prime sum as being $\ll X^{\epsilon}$. Combining this with the fact that the number of cubic forms with discriminant in $(X,X+\nu y]$ is $o(X^{1-\delta})$, for $y$ as in the statement of the Lemma (and $\delta$ small), we obtain the desired result.
\end{proof}
\begin{remark}
    The necessity of Lemma \ref{GRHlemma2} comes from the fact that the coefficients of our zeta function are not nonnegative. Indeed, nonnegativity is crucial for the rest of our arguments, as we shall obtain substantial cancellation when averaging over $p$. For zeta functions with nonnegative coefficients, one could instead make use of the observation described at the end of the proof of \cite[Theorem 5]{LDTT} in order to disregard the tail sum.
\end{remark} 
In the sequel, we will take $y=X^{1-\eta}$, with $\eta$ small (depending only on $\sigma$ and $\delta$). Let us write $\xi_{S_3}^{\pm}(s,\theta_p\otimes \Psi_{n})$ for the zeta function generated by the coefficients $a^{\pm}_1(\theta_p\otimes\Psi_{n},m)$. Then, one instance of Perron's formula, see e.g. \cite[Eq. (33)]{BTT}, gives that
\begin{equation*}
    \Tilde{A}_k(\theta_p\otimes \Psi_{n}; X) = \frac{1}{2\pi i}\int_{(2)}\xi_{S_3}^{\pm}(s,\theta_p\otimes \Psi_{n})\frac{X^{s+k}}{s(s+1)...(s+k)}ds.
\end{equation*}
We now complete $\xi_{S_3}^{\pm}(s,\theta_p\otimes \Psi_{n})$ into a proper Shintani zeta function $\xi^{\pm}(s,\theta_p\otimes \Psi_{n})$, by including the coefficients corresponding to reducible forms, and $C_3$-forms. We write 
\begin{equation*}
    \xi_{S_3^c}^{\pm}(s,\theta_p\otimes \Psi_{n}):= \xi^{\pm}(s,\theta_p\otimes \Psi_{n})-\xi_{S_3}^{\pm}(s,\theta_p\otimes \Psi_{n}),
\end{equation*}
so that
\begin{equation}\begin{split}
     \Tilde{A}_k(\theta_p\otimes \Psi_{n}; X) &= \frac{1}{2\pi i}\int_{(2)}\xi^{\pm}(s,\theta_p\otimes \Psi_{n})\frac{X^{s+k}}{s(s+1)...(s+k)}ds-\frac{1}{2\pi i}\int_{(2)}\xi_{S_3^c}^{\pm}(s,\theta_p\otimes \Psi_{n})\frac{X^{s+k}}{s(s+1)...(s+k)}ds\\ 
     &=:I_1^\pm(n,p,X)-I_2^\pm(n,p,X)
\end{split}
\end{equation}

We summarise our findings in the following lemma.
\begin{lemma}\label{onelevallforms}
    Conditional on the GRH for non-Galois cubic Dedekind zeta functions, and assuming that either $\sigma < 1$, or that the GRH for Dedekind zeta functions associated with quadratic fields holds, we have that 
    \begin{equation*}
        \frac{1}{\#\mathcal{F}^{\pm}(X)}\sum_{K\in \mathcal{F}^{\pm}(X)}D_\phi(K) \sim \widehat{\phi}(0)-\frac{\phi(0)}{2}-\frac{2}{L\#\mathcal{F}^{\pm}(X)}\sum_{p\geq X^{1/4}}\frac{\log p}{p^{1/2}}\widehat{\phi}\left(\frac{\log p}{L}\right)\sum_{n\leq X^\delta} \mu(n)y^{-k}\Delta^k_y \big(I_1^\pm(n,p,X)-I_2^\pm(n,p,X)\big).
    \end{equation*}
    
\end{lemma}

\section{Subrings, reducible rings, and $C_3$-rings}\label{quadringsect}
In this section, we estimate the contribution to the one-level density coming from $I_2^\pm(n,p,X)$. Specifically, from Lemma \ref{onelevallforms}, we see that we should study
\begin{equation}\label{correctiononelevterm}
    \frac{2}{L\#\mathcal{F}^{\pm}(X)}\sum_{p\geq X^{1/4}}\frac{\log p}{p^{1/2}}\widehat{\phi}\left(\frac{\log p}{L}\right)\sum_{n\leq X^\delta} \mu(n)y^{-k}\Delta^k_yI_2^\pm(n,p,X).
\end{equation}

\subsection{Subrings of maximal orders}
We will study $I_2^\pm(n,p,X)$ by switching to maximal overrings, and using exact formulae for the number of subrings of a maximal ring. Such formulae were proven by Datskovsky--Wright \cite{DW}, see also \cite[Lemma 3.2]{Nakagawa} for the precise statement that we need. More precisely, if $A$ is a maximal cubic ring, then writing 
\begin{equation*}
    \eta_{\mathrm{Frac}(A)}(s) = \sum_{A'\subseteq A} \frac{1}{[A:A']^s},
\end{equation*}
where $A'$ runs over the cubic subrings of $A$, we have that (\cite[Lemma 3.2]{Nakagawa})
\begin{equation}\label{suborderctdw}
\begin{split}
    \eta_K(s) &= \frac{\zeta_K(s)}{\zeta_K(2s)}\zeta(2s)\zeta(3s-1), \text{ $K$ a cubic field},\\
    \eta_{K\oplus \mathbb{Q}}(s) &= \frac{\zeta_K(s)}{\zeta_K(2s)}\zeta(s)\zeta(3s-1), \text{ $K$ a quadratic field},\\
    \eta_{ \mathbb{Q}^3}(s) &= \frac{\zeta(s)^3}{\zeta(2s)^2}\zeta(3s-1).
\end{split}
\end{equation}
Recall also that $\mathrm{Disc}(A')=[A:A']^2\mathrm{Disc}(A)$. We remark that all the zeta functions above grow at most polynomially in the vertical direction.

Using \eqref{suborderctdw} shows that the contribution of $C_3$-cubic fields to $I^+_2(n,p,X)$ is
\begin{equation}\label{c3int}
   \ll  \frac{1}{2\pi i}\int_{(2)}\sum_{K\text{ $C_3$-cubic}}\frac{1}{\abs{\mathrm{Disc}(K)}^s}\frac{\zeta_K(2s)\zeta(4s)}{\zeta_K(4s)}\zeta(6s-1)\frac{\eta_{K,n}(2s)}{\prod_{q\mid n}\big(1+\eta_{K,q}(2s)\big)}\frac{X^{s+k}}{s(s+1)...(s+k)}ds,
\end{equation}
where we define 
\begin{equation*}
    1+\eta_{K,q}(2s)
\end{equation*}
to be the local factor of $\eta_K$ at a prime $q$, and
\begin{equation*}
    \eta_{K,n} := \prod_{q\mid n}\eta_{K,q}.
\end{equation*}
We may cut off the summation over $K$ at $\abs{\mathrm{Disc}(K)}\leq X$, as integrating the other terms yields $0$, as one sees by shifting the contour by letting $\Re(s)\to+\infty$. Next, we have that
\begin{equation*}
   \bigg\lvert \frac{\zeta_K(2s)\zeta(4s)}{\zeta_K(4s)}\zeta(6s-1)\frac{\eta_{K,n}(2s)}{\prod_{q\mid n}\big(1+\eta_{K,q}(2s)\big)} \bigg\rvert \ll_\epsilon n^\epsilon,
\end{equation*}
when $\Re(s)\geq 1/2+\epsilon$. Moreover, a result of Cohn \cite{Cohn} asserts that the number of $C_3$-cubic fields of discriminant $\leq X$ is asymptotic to  $D_{C_3}X^{1/2}$ for some constant $D_{C_3}$. Combining these two observations, we see that, after shifting the contour of the integral over the remaining terms to $\Re(s)=1/2+\epsilon$, \eqref{c3int} is
\begin{equation}\label{c3intbd}
    \ll X^{k+1/2+\epsilon}.
\end{equation}
Using the trivial bound \eqref{trivialfindiff} then implies that the total $C_3$-contribution to \eqref{correctiononelevterm} is bounded by
\begin{equation*}
    \ll \frac{1}{X}\sum_{p}\frac{\log p}{p^{1/2}}\widehat{\phi}\left(\frac{\log p}{L}\right)\sum_{n\leq X^\delta} \mu(n)y^{-k}X^{k+1/2+\epsilon}\ll X^{\eta k-1/2}\sum_{p\leq X^{\sigma}}\frac{\log p}{p^{1/2}}\ll X^{\eta k+(\sigma-1)/2},
\end{equation*}
which is $o(1)$ when $\sigma < 1$, assuming that $\eta$ is small enough. We remark here that the condition $\sigma < 1$ is superfluous if one assumes the GRH for $\zeta_K$ with $K$ $C_3$-cubic by similar arguments as in the proof of Lemma \ref{lemmaGRH1} and \ref{GRHlemma2}.

\subsection{The reducible contribution}\label{redringssect}

We turn to the contribution from rings which are not domains. We first consider rings whose ring of fractions is $\mathbb{Q}^3$, henceforth these are referred to as $\mathbb{Q}^3$-rings. The contribution from $\mathbb{Q}^3$-rings to $I^+_2(n,p,X)$, is bounded by
\begin{equation*}
    \ll \frac{1}{2\pi i}\int_{(2)}\frac{\zeta(2s)^3}{\zeta(4s)^2}\zeta(6s-1)\frac{\eta_{\mathbb{Q}^3,n}(2s)}{\prod_{q\mid n}\big(1+\eta_{\mathbb{Q}^3,q}(2s)\big)}\frac{X^{s+k}}{s(s+1)...(s+k)}ds.
\end{equation*}
Here, we simply shift this integral to the line $\Re(s)=1/2+\epsilon$ and bound as above. The contribution from the shifted integral is $o(1)$ for $\sigma < 1$.

The remaining contribution comes from rings associated with quadratic fields. If $f$ is a form corresponding to a ring $R(f)$, then unless $R(f)$ is nonmaximal at $p$, $\theta_p(f)$ equals $2$ if $p$ splits in the corresponding quadratic field and it equals zero otherwise. Using \eqref{nonmaxbd} and that $p\geq X^{1/4}$, one sees that the contribution from rings $R(f)$ nonmaximal at $p$ is negligible.

It is well-known that the set of quadratic fields is in bijection with the set of fundamental discriminants $d$, not equal to $1$. Hence, at the cost of a negligible error, we may replace $\theta_p$ by
\begin{equation*}
1+\chi_d(p),
\end{equation*}
where $\chi_d$ is the Kronecker character associated with the fundamental discriminant $d$. This agrees with $\theta_p$ at forms which are maximal at $p$. Writing $K_d$ for the quadratic field of discriminant $d$, we see that the remaining quadratic contribution to $I^\pm_2(n,p,X)$ is thus
\begin{equation}\label{quadint}
  \sum_\sumstack{d\neq 1 \text{ fundamental}\\ \mathrm{sgn}(d)=\pm1}\frac{1}{2\pi i}\int_{(2)}\frac{\big(1+\chi_d(p)\big)}{\abs{d}^s}\frac{\zeta_{K_d}(2s)}{\zeta_{K_d}(4s)}\zeta(2s)\zeta(6s-1)\frac{\eta_{K_d\oplus \mathbb{Q},n}(2s)}{\prod_{q\mid n}\big(1+\eta_{K_d\oplus \mathbb{Q},q}(2s)\big)}\frac{X^{s+k}}{s(s+1)...(s+k)}ds.
\end{equation}

Now, one may factor
\begin{equation*}
    \zeta_{K_d}(s) = \zeta(s)L(s,\chi_d),
\end{equation*}
where the local factor at a prime $q$ of $L(s,\chi_d)$ equals
\begin{equation*}
    \begin{cases}
        (1-q^{-s})^{-1}, \text{ if $q$ splits in $K_d$, } \\
        (1+q^{-s})^{-1}, \text{ if $q$ is inert in $K_d$, }\\
        1, \text{ if $q$ ramifies in $K_d$, }
    \end{cases}
\end{equation*}
and we write this as 
\begin{equation*}
    \big(1-\chi_d(q)q^{-s}\big)^{-1}.
\end{equation*}
Now, $\chi_d(q)$ is multiplicative in its argument. Moreover $\chi_d(q)$ is equal to the Kronecker symbol $\left(\frac{d}{q}\right)$.

In order to study the integral above, we use the following lemma, which is a straightforward consequence of the GRH.
\begin{lemma}\label{grhquadlemma}
    Assuming the GRH for quadratic Dirichlet $L$-functions, we have that
    \begin{equation*}
        \sum_\sumstack{d\neq 1 \text{ fundamental}\\ \mathrm{sgn}(d)=\pm1,\\ \abs{d}\leq X} \chi_d(m)\ll X^{1/2+\epsilon}m^\epsilon,
    \end{equation*}
    when $m$ is a non-square.
\end{lemma}
\begin{proof}
The argument is standard and we, therefore, only provide a sketch. We begin by noting that the sum may be restricted to $d$ satisfying $(d,m)=1$, as the character is zero on $d$ not satisfying these conditions. For such a $d$, we write
\begin{equation*}
    \chi_d(m)=\left(\frac{d}{m}\right) = \left(\frac{d}{m_0}\right),
\end{equation*}
where $m_0$ is the squarefree part of $m$. One version of quadratic reciprocity states that
\begin{equation}\label{quadrec}
    \left(\frac{d}{m_0}\right) = \left(\frac{m_0(-1)^{(m_0'-1)/2}}{d}\right) = \psi_{m_0}(d),
\end{equation}
where $m_0'$ is the odd part of $m_0$. Here, $\psi_{m_0}$ is a quadratic Dirichlet character of modulus $4m_0$ if $m_0$ is even, and else of modulus $m_0$.

We recall that there are essentially three cases for $d$ modulo $16$. More specifically, we either have that $d$ is squarefree with $d\equiv 1 \nsmod 4$, or $d=4D$ with $D$ squarefree and $D\equiv 2,3\nsmod{4}$. Applying the GRH for quadratic Dirichlet $L$-functions (or equivalently for Dedekind zeta functions associated with quadratic fields), one can bound averages of $\psi_{m_0}(d)$ over squarefree integers $d$ satisfying congruence conditions modulo $16$, by summing against certain Dirichlet characters. For example, one handles the condition $d\equiv 1 \nsmod{4}$ by writing
\begin{equation*}
    \mathbf{1}_{d\equiv 1 \nsmod{4}} = \frac{1}{2}\sum_{\chi \nsmod{4}}\chi(d),
\end{equation*}
where the summation ranges over Dirichlet characters modulo $4$. Using these methods, the case $m_0 \neq 2$ is straightforward, and one obtains a bound $\ll X^{1/2+\epsilon}$. When $m_0=2$, 
\begin{equation*}
    \chi_d(2)=\begin{cases}
        1 , &\text{ if $d\equiv \pm 1 \nsmod{8}$,}\\
        -1, &\text{ if $d\equiv \pm 3 \nsmod{8}$,}\\
        0, &\text{ else.}
    \end{cases}
\end{equation*}
Hence, the only nonzero contribution is from $d\equiv 1 \nsmod{4}$. Now, $\chi_d(2)$ is a primitive Dirichlet character of modulus $8$, with respect to the variable $d$, whence its product with either of the two characters $\nsmod{4}$ is nonprincipal. Hence, by using the GRH we obtain the bound in the statement of the lemma. We omit the details.
\end{proof}
Using the above lemma, we can accurately bound the contribution from $I^\pm_2(n,p,X)$ to the one-level density.
\begin{prop}\label{redcountprop}
    Assuming the GRH for Dedekind zeta functions associated with quadratic fields, and that $\sigma <1$, we have that \eqref{correctiononelevterm} equals
    \begin{equation*}
        \frac{2}{L\#\mathcal{F}^{\pm}(X)}\sum_{p}\frac{\log p}{p^{1/2}}\widehat{\phi}\left(\frac{\log p}{L}\right)\sum_{n\leq X^\delta} \mu(n)y^{-k}\Delta^k_y\frac{\pi^2}{24}\frac{X^{1+k}}{(k+1)!n^2}\prod_{q\mid n}\big(2-q^{-2}\big)+o(1)
    \end{equation*}
\end{prop}
\begin{proof}
   By our previous computations, it suffices to study the contribution coming from rings associated with quadratic fields. A short computation, using \eqref{suborderctdw}, shows that the zeta function counting suborders of $K_d$ is explicitly given by
    \begin{equation}\label{prodallsubord}
        \prod_{q\nmid d}\frac{1}{(1-q^{-2s})(1-q^{1-6s})}\left(1+\frac{1+\chi_d(q)}{q^{2s}}+\frac{1}{q^{4s}}\right)\prod_{q\mid d} \frac{(1+q^{-2s})}{(1-q^{-2s})(1-q^{1-6s})}.
    \end{equation}
     Hence, the local factor at $q$ of the zeta function counting suborders of $K_d$ with index divisible by $n$ is
     \begin{equation*}
         \frac{1}{(1-q^{-2s})(1-q^{1-6s})}\left(1+\frac{1+\chi_d(q)}{q^{2s}}+\frac{1}{q^{4s}}\right)-1,
     \end{equation*}
     if $q\nmid d$, but $q\mid n$. Else, if $q\mid (d,n)$, the local factor is
     \begin{equation*}
         \frac{(1+q^{-2s})}{(1-q^{-2s})(1-q^{1-6s})}-1,
     \end{equation*}
     and in the case that $q\nmid n$, then the local factor agrees with the local factor of \eqref{prodallsubord}. 
     
     We now want to use Lemma \ref{grhquadlemma} in order to bound the contribution coming from the various $\chi_d(q)$ above. First, the only possibility for the argument being square is if $q=p$, as we multiply by the factor $1+\chi_d(p)$ from \eqref{quadint}. Because of the accompanying factor $p^{-2s}$ coming from $q=p$, these terms provide a negligible contribution to \eqref{quadint}, as $p\geq X^{1/4}$. The remaining contribution from terms involving $\chi_d$ is handled using Lemma \ref{grhquadlemma}, by shifting the integral involving these terms to $\Re(s)=1/2+\epsilon$, whereafter the shifted integral can be bounded by the same expression as in \eqref{c3intbd}.
     
     What remains is to consider the contribution from terms not involving Kronecker characters, i.e. from the Euler product
     \begin{equation*}
     \begin{split}
        \zeta(2s)\zeta(6s-1) \prod_{q\nmid dn}&\left(1+\frac{1}{q^{2s}}+\frac{1}{q^{4s}}\right)\prod_\sumstack{q\mid d\\ q\nmid n} (1+q^{-2s})\\&\times\prod_\sumstack{q\mid n\\ q\nmid d}\left(\left(1+q^{-2s}+q^{-4s}\right)-(1-q^{-2s})(1-q^{1-6s})\right)\prod_\sumstack{q\mid d\\ q\mid n}\left(1+q^{-2s}-(1-q^{-2s})(1-q^{1-6s})\right)
     \end{split}
     \end{equation*}
     We may rewrite the above into a form more amenable to summation over $d$. Specifically, the above equals
    \begin{equation}\label{constfactor}
     \begin{split}
        \zeta(2s)\zeta(6s-1) \prod_{q\nmid n}&\left(1+\frac{1}{q^{2s}}+\frac{1}{q^{4s}}\right)\prod_\sumstack{q\mid d\\ q\nmid n}\frac{(1+q^{-2s})}{\left(1+q^{-2s}+q^{-4s}\right)} \\&\times\prod_\sumstack{q\mid n}\left(2q^{-2s}+q^{-4s}+q^{1-6s}-q^{1-8s}\right)\prod_\sumstack{q\mid d\\ q\mid n}\frac{\left(2+q^{1-4s}-q^{1-6s}\right)}{\left(2+q^{-2s}+q^{1-4s}-q^{1-6s}\right)}.
     \end{split}
     \end{equation}
    We now wish to sum this over $d$ using multiplicativity. We then keep the sign of $d$ fixed. Moreover, we only sum over fundamental discriminants $d$, i.e. $d\equiv1\nsmod{4}$, or $d$ of the form $4d'$, with $d'\equiv2,3 \nsmod{4}$. To sum over one of these types of $d$, we may sum over all integers $d = 2^ed'$ of a fixed sign, with $e=0,2,3$, $d'$ odd squarefree, and weighted by $\mathbf{1}_{d' \equiv c \nsmod{4}}$, where $c =1,3$, unless $e=3$ in which case the only requirement is that $d'$ should be odd. Now,
    \begin{equation*}
        \mathbf{1}_{d' \equiv c \nsmod{4}} = \frac{1}{2}\sum_{\chi \nsmod{4}}\chi(d')\overline{\chi}(c),
    \end{equation*}
    where the summation ranges over the two distinct Dirichlet characters modulo $4$. By way of the GRH, one sees that the contribution from the nontrivial character is negligible for $\sigma <1$, similarly to how we bounded the contribution from $\chi_d$ above. We may thus focus on the contribution from the trivial character $\chi_0$. We have by multiplicativity that
    \begin{equation*}
    \begin{split}
        \frac{1}{2^{\delta_{3\neq e}+es}}&\sum_\sumstack{d'\geq 1 \text{ squarefree} \\ (d',2)=1}\frac{1}{d'^s}\prod_\sumstack{q\mid 2^ed'\\ q\nmid n}\frac{(1+q^{-2s})}{\left(1+q^{-2s}+q^{-4s}\right)}\prod_\sumstack{q\mid 2^ed'\\ q\mid n}\frac{\left(2+q^{1-4s}-q^{1-6s}\right)}{\left(2+q^{-2s}+q^{1-4s}-q^{1-6s}\right)}
        \\&=\frac{1}{2^{\delta_{3\neq e}+es}}\prod_\sumstack{q\mid 2^e\\ q\nmid n}\frac{(1+q^{-2s})}{\left(1+q^{-2s}+q^{-4s}\right)}\prod_\sumstack{q\mid 2^e\\ q\mid n}\frac{\left(2+q^{1-4s}-q^{1-6s}\right)}{\left(2+q^{-2s}+q^{1-4s}-q^{1-6s}\right)}
        \\&\times\prod_\sumstack{ q\nmid n\\q\neq 2}\left(1+\frac{q^{-s}(1+q^{-2s})}{\left(1+q^{-2s}+q^{-4s}\right)}\right)\prod_\sumstack{ q\mid n\\q\neq 2}\left(1+\frac{q^{-s}\left(2+q^{1-4s}-q^{1-6s}\right)}{\left(2+q^{-2s}+q^{1-4s}-q^{1-6s}\right)}\right)
    \end{split}
    \end{equation*}
Including the prefactors from \eqref{constfactor} yields
    \begin{equation*}
    \begin{split}
\frac{\zeta(2s)\zeta(6s-1)}{2^{\delta_{3\neq e}+es}n^{2s}}&\prod_\sumstack{q\mid 2\\ q\nmid n}(1+q^{-2s})^{\delta_{e\geq 1}}(1+q^{-2s}+q^{-4s})^{\delta_{e=0}}\prod_\sumstack{q\mid 2\\ q\mid n}\left(2+q^{1-4s}-q^{1-6s}\right)^{\delta_{e\geq 1}}\left(2+q^{-2s}+q^{1-4s}-q^{1-6s}\right)^{\delta_{e=0}}
        \\&\times\prod_\sumstack{ q\nmid n\\q\neq 2}\left(1+q^{-s}+q^{-2s}+q^{-3s}+q^{-4s}\right)\prod_\sumstack{ q\mid n\\q\neq 2}\left(2+2q^{-s}+q^{-2s}+q^{1-4s}+q^{1-5s}-q^{1-6s}-q^{1-7s}\right)
    \end{split}
    \end{equation*}

Next, we want to shift the contour in the integral \eqref{quadint} to $\Re(s)=1/2+\epsilon$ and pick up the residue at $s=1$. To do this, we first write
\begin{equation*}
    \prod_\sumstack{ q\nmid n\\q\neq 2}\left(1+q^{-s}+q^{-2s}+q^{-3s}+q^{-4s}\right) = \prod_\sumstack{ q\nmid n\\q\neq 2}\frac{1-q^{-5s}}{1-q^{-s}} = \frac{\zeta(s)}{\zeta(5s)}\prod_\sumstack{ q\mid 2n}\frac{1-q^{-s}}{1-q^{-5s}}.
\end{equation*}
The shifted integral is negligible. The residue at $s=1$ contributes
\begin{equation*}
\begin{split}
    \frac{\zeta(2)}{2^{\delta_{3\neq e}+e}n^2}&\frac{X^{1+k}}{(k+1)!}\prod_\sumstack{q\mid 2\\ q\nmid n}(1+1/4)^{\delta_{e\geq 1}}(1+1/4+1/16)^{\delta_{e=0}}\prod_\sumstack{q\mid 2\\ q\mid n}\left(2+1/8-1/32\right)^{\delta_{e\geq 1}}\left(2+1/4+1/8-1/32\right)^{\delta_{e=0}}
        \\&\times\prod_\sumstack{ q\mid 2n}\frac{1-q^{-1}}{1-q^{-5}}\prod_\sumstack{ q\mid n\\q\neq 2}\left(2+2q^{-1}+q^{-2}+q^{-3}+q^{-4}-q^{-5}-q^{-6}\right).
\end{split}
\end{equation*}
Noting that
\begin{equation*}
    2+2q^{-1}+q^{-2}+q^{-3}+q^{-4}-q^{-5}-q^{-6} = (2-q^{-2})(1+q^{-1}+q^{-2}+q^{-3}+q^{-4}),
\end{equation*}
we may write the above as
\begin{equation*}
\begin{split}
    \frac{\zeta(2)}{2^{\delta_{3\neq e}+e}n^2}&\frac{X^{1+k}}{(k+1)!}\prod_\sumstack{q\mid 2\\ q\nmid n}(1+1/4)^{\delta_{e\geq 1}}(1+1/4+1/16)^{\delta_{e=0}}\prod_\sumstack{q\mid 2\\ q\mid n}\left(2+1/8-1/32\right)^{\delta_{e\geq 1}}\left(2+1/4+1/8-1/32\right)^{\delta_{e=0}}
        \\&\times\prod_\sumstack{ q\mid 2}\frac{1-q^{-1}}{1-q^{-5}}\prod_{q\mid n/(2,n)}\left(2-q^{-2}\right).
\end{split}
\end{equation*}
Summing this over all possible local specifications at the prime $2$ yields
\begin{equation*}
\begin{split}
    \zeta(2)&\frac{X^{1+k}}{(k+1)!n^2}\bigg(\frac{1}{4}\prod_\sumstack{q\mid 2\\ q\nmid n}(1+1/4)\prod_\sumstack{q\mid 2\\ q\mid n}\left(2+1/8-1/32\right)+\frac{1}{2}\prod_\sumstack{q\mid 2\\ q\nmid n}(1+1/4+1/16)\prod_\sumstack{q\mid 2\\ q\mid n}\left(2+1/4+1/8-1/32\right)\bigg)
        \\&\times \frac{1}{2}\cdot\frac{1}{1-1/32} \prod_{q\mid n/(2,n)}\left(2-q^{-2}\right).
\end{split}
\end{equation*}
A case distinction depending on whether $2\mid n$ or not shows that in all cases, the above simplifies to
\begin{equation*}
\begin{split}
    \frac{\zeta(2)}{2}&\frac{X^{1+k}}{(k+1)!n^2}\prod_{q\mid n}\left(2-q^{-2}\right).
\end{split}
\end{equation*} 
Lastly, recalling that quadratic rings are weighted by a factor $1/2$ in the Shintani zeta function, we obtain the final result. 
\end{proof}
\begin{remark}
    One could give a different proof of the above proposition, by instead using the approach of Lemma \ref{redshintlemma}. As the method of using counting functions of suborders is of use in the next section, we choose to give the above proof here.
\end{remark}

\section{Shintani zeta functions and total ramification}
In this section, we study $I^{\pm}_1(n,p,X)$. Recall that
\begin{equation*}
     I^{\pm}_1(n,p,X)= \frac{1}{2\pi i}\int_{(2)}\xi^{\pm}(s,\theta_p\otimes \Psi_{n})\frac{X^{s+k}}{s(s+1)...(s+k)}ds.
\end{equation*}
A discriminant-reducing identity \cite[Proposition 7.1]{BTT} asserts that
\begin{equation}\label{Shintanidiscred}
    \xi^{\pm}(s,\theta_p\otimes \Psi_{n}) = n^{-4s}\sum_{abc=n}\mu(b)a^{2s}\xi^{\pm}(s,\theta_p\otimes \eta_{ab}),
\end{equation}
where $\eta_{ab}$ counts the number of roots of a form when reduced modulo $ab$. Hence, we may write
\begin{equation*}
    I^{\pm}_1(n,p,X)= \frac{1}{2\pi i}\sum_{abc=n}\mu(b)\int_{(2)}n^{-4s}a^{2s}\xi^{\pm}(s,\theta_p\otimes \eta_{ab})\frac{X^{s+k}}{s(s+1)...(s+k)}ds.
\end{equation*}
We will make use of the Datskovsky--Wright diagonalisation so that the functional equation becomes simpler to apply. We write
\begin{equation*}
    J^{\pm}_1(n,p,X)= \frac{1}{2\pi i}\sum_{abc=n}\mu(b)\int_{(2)}n^{-4s}a^{2s}\big(\sqrt{3}\xi^{+}(s,\theta_p\otimes \eta_{ab})\pm \xi^{-}(s,\theta_p\otimes \eta_{ab})\big)\frac{X^{s+k}}{s(s+1)...(s+k)}ds,
\end{equation*}
and note that studying $I^{\pm}_1(n,p,X)$ is tantamount to studying $J^{\pm}_1(n,p,X)$. We simplify the notation somewhat by writing
\begin{equation*}
    \xi_{\pm}(s,\theta_p\otimes \eta_{ab}) := \sqrt{3}\xi^{+}(s,\theta_p\otimes \eta_{ab})\pm \xi^{-}(s,\theta_p\otimes \eta_{ab}).
\end{equation*}
Next, we shift the contour above to the line $(\epsilon)$ and thus obtain
\begin{equation}\label{Jtot}
\begin{split}
    J^{\pm}_1(n,p,X)&=\sum_{abc=n}\mu(b)\left(\frac{X^{k+1}a^2}{(k+1)!n^4}\mathrm{Res}_{s=1}\xi_{\pm}(s,\theta_p\otimes \eta_{ab})+\frac{X^{k+5/6}a^{5/3}}{5/6(5/6+1)...(5/6+k)n^{10/3}}\mathrm{Res}_{s=5/6}\xi_{\pm}(s,\theta_p\otimes \eta_{ab})\right)\\
    &  +\sum_{abc=n}\mu(b)\cdot\frac{1}{2\pi i}\int_{(\epsilon)}n^{-4s}a^{2s}\xi_{\pm}(s,\theta_p\otimes \eta_{ab})\frac{X^{s+k}}{s(s+1)...(s+k)}ds.
\end{split}
\end{equation}

\subsection{The contribution from the residues}\label{ressect}
We now study the contribution of the various residues to \eqref{Jtot}. The residues of the Shintani zeta functions $\xi_{\pm}(s,\theta_p\otimes \eta_{ab})$ were found in \cite{TT1}, and are summarised in \cite[Theorem 2.4]{BTT}. More specifically, we have that
\begin{equation*}
    \mathrm{Res}_{s=1} \xi_{\pm}(s,\theta_p\otimes \eta_{ab}) = \frac{\pi^2(\sqrt{3}\pm 1)}{24}\big(1+p^{-1}+p^{-2}\big)^{-1}\left(1-p^{-2}\right)\left(1-p^{-3}\right)2^{\omega(ab)}.
\end{equation*}
The function, $\xi_-(s,\theta_p\otimes \eta_{ab})$ does not have a pole at $s=5/6$, while
\begin{equation*}
    \mathrm{Res}_{s=5/6} \xi_{+}(s,\theta_p\otimes \eta_{ab}) = \frac{2\sqrt{3}\pi^2\zeta(1/3)}{9\Gamma(2/3)^3}p^{-1/3}\left(1-p^{-1}\right)\left(1-p^{-2/3}\right)\prod_{q\mid ab}\left(1+q^{-1/3}\right).
\end{equation*}

Now, the contribution from the residue at $s=5/6$ to the one-level density is $o(1)$ when $\sigma <1$. Indeed,
\begin{equation*}
\begin{split}
    \frac{1}{L\#\mathcal{F}^{\pm}(X)}&\sum_{p\geq X^{1/4}}\frac{\log p}{p^{1/2}}\widehat{\phi}\left(\frac{\log p}{L}\right)\sum_{n\leq X^\delta} \mu(n)y^{-k}\Delta^k_y \sum_{abc=n}\frac{X^{k+5/6}a^{5/3}p^{-1/3}}{n^{10/3}}\ll X^{k\eta-1/6}\sum_{p\geq X^{1/4}}\frac{\log p}{p^{5/6}}\widehat{\phi}\left(\frac{\log p}{L}\right)
    \\&\ll X^{k\eta+ (\sigma-1)/6}.
\end{split}
\end{equation*}

The contribution from the residue at $s=1$ is non-negligible, therefore, we compute the non-negligible part, which is
\begin{equation*}
    \frac{X^{k+1}}{(k+1)!n^4}\cdot\frac{\pi^2(\sqrt{3}\pm 1)}{24}\sum_{abc=n}\mu(b)a^22^{\omega(ab)} = \frac{X^{k+1}}{(k+1)!}\cdot\frac{\pi^2(\sqrt{3}\pm 1)}{24}\cdot \frac{1}{n^2}\prod_{q\mid n}\left(2-q^{-2}\right). 
\end{equation*}
After undoing the Datskovsky--Wright diagonalisation, we see that this cancels against the main term coming from Proposition \ref{redcountprop}.

\subsection{Estimating the shifted integral}
We now turn our attention to the shifted integral in \eqref{Jtot}. Using a result by Mori \cite{Mori}, see also \cite[Theorem 11]{TTOrb}, we have that
\begin{equation}\label{thetapFouriertransf}
  \widehat{\theta_p} = \frac{1}{p}\left(1-\frac{1}{p}\right)\mathbf{1}_{p,(1^3)}-\frac{1}{p^2}\left(1-\frac{1}{p}\right)\mathbf{1}_{p,(1^21)} + \frac{1}{p^3}\left(2\cdot\mathbf{1}_{p,(111)}-\mathbf{1}_{p,(3)}\right),
\end{equation}
and we write
\begin{equation*}
    \widehat{\theta_p} =: \frac{\mathbf{1}_{p,(1^3)}}{p} + r_p,
\end{equation*}
for an appropriate function $r_p$.
Applying the functional equation \eqref{genfcnal}, we may rewrite the shifted integral as
\begin{equation*}
    \frac{\pm 3}{2\pi i}\int_{(1-\epsilon)}n^{4s-4}a^{2-2s}(abp)^{4s}\frac{\Delta^{\pm}(s)}{\Delta^{\pm}(1-s)}\xi_{\pm}(s,\widehat{\theta_p}\otimes \widehat{\eta_{ab}})\frac{\Gamma(1-s)}{\Gamma(k+2-s)}X^{k+1-s}ds,
\end{equation*}
where we also recall that the Shintani zeta function above only counts forms which have at least a triple root modulo $3$. 
Next, we separate the part involving $r_p$ from the integral above and study
\begin{equation*}
    \frac{\pm 3}{2\pi i}\int_{(1-\epsilon)}n^{4s-4}a^{2-2s}(abp)^{4s}\frac{\Delta^{\pm}(s)}{\Delta^{\pm}(1-s)}\xi_{\pm}(s,r_p\otimes \widehat{\eta_{ab}})\frac{\Gamma(1-s)}{\Gamma(k+2-s)}X^{k+1-s}ds.
\end{equation*}
Changing variables $s\mapsto 1-s$ yields
\begin{equation*}
    \frac{\pm 3}{2\pi i}\int_{(\epsilon)}n^{4s}a^{2s}(abp)^{4-4s}\frac{\Delta^{\pm}(1-s)}{\Delta^{\pm}(s)}\xi_{\pm}(1-s,r_p\otimes \widehat{\eta_{ab}})\frac{X^{k+s}}{s(s+1)...(s+k)}ds.
\end{equation*}
We now shift this integral back to $\Re(s)=1+\epsilon$. First, using Mori's results once again, we find that the inverse transform $r_p'$ of $r_p$ is
\begin{equation*}
    \ll p\mathbf{1}_{p,(0)}+\mathbf{1}_{p,(1^3)}+\mathbf{1}_{p,(1^21)}+\frac{1}{p}\left(\mathbf{1}_{p, (111)}+\mathbf{1}_{p,(21)}+\mathbf{1}_{p,(3)}\right).
\end{equation*}
Hence, applying the functional equation again, the integral above equals
\begin{equation*}
    \frac{1}{2\pi i}\int_{(\epsilon)}n^{-4s}a^{2s}\xi_{\pm}(s,r'_p\otimes \eta_{ab})\frac{X^{k+s}}{s(s+1)...(s+k)}ds.
\end{equation*}
We now shift the contour to $\Re(s)=1+\epsilon$, picking up residues at $s=1$ and $s=5/6$. The shifted integral can be reinterpreted as a sum, using Perron's formula. The bound for $r_p'$ above shows that the shifted integral contributes only a term $o(1)$ to the one-level density. Indeed, the contributions coming from the indicator functions of unramified splitting types are bounded trivially. One bounds the contribution coming from $\mathbf{1}_{p,(0)}$ by using that
\begin{equation*}
    \xi_{\pm}(s,\mathbf{1}_{p,(0)}\otimes \eta_{ab}) = p^{-4s}\xi_{\pm}(s,\eta_{ab}).
\end{equation*}
One bounds the contribution from the totally ramified splitting type by using \cite[Proposition 4.5]{BTT} and for $\mathbf{1}_{p,(1^21)}$, one interchanges the order of summation in the one-level density to first sum over $p$, similar to \eqref{rambound}. Finally, we remark that the contribution from the various residues are bounded as in Section \ref{ressect}, using \cite[Theorem 2.4]{BTT}.

We now turn our attention to the remaining parts of the shifted integral, namely
\begin{equation*}
    \frac{\pm 3p^{-1}}{2\pi i}\int_{(1-\epsilon)}n^{4s-4}a^{2-2s}(abp)^{4s}\frac{\Delta^{\pm}(s)}{\Delta^{\pm}(1-s)}\xi_{\pm}(s,\mathbf{1}_{p,(1^3)}\otimes \widehat{\eta_{ab}})\frac{\Gamma(1-s)}{\Gamma(k+2-s)}X^{k+1-s}ds.
\end{equation*}
The proof of Theorem \ref{onelevQ} is concluded once we can prove that the contribution from this term to the one-level density is $o(1)$ for $\sigma <1$. Now, again by Mori's results, see also \cite[Proposition 1]{TTOrb}, we have that
\begin{equation*}
    \widehat{\eta_{ab}} = \prod_{q\mid ab}\left((1+q^{-1})\mathbf{1}_{q,(0)}+q^{-1}\mathbf{1}_{q,(1^3)}\right) = \sum_{\ell\mid ab}\frac{\sigma(ab/\ell)}{ab}\mathbf{1}_{(ab/\ell),(0)}\mathbf{1}_{\ell,(1^3)},
\end{equation*}
for $ab$ coprime to $3$. For the case $q=3$, we need the following lemma, whose proof is a slight extension of the proof of \cite[Proposition 1]{TTOrb}.
\begin{lemma}\label{eta3fourier}
    Let $y \in V(\mathbb{Z})^*$. Then, $\widehat{\eta_3}(y)$ depends only on the projection of $y$ to the space
    \begin{equation*}
        V_3:= \frac{\mathbb{Z}}{3\mathbb{Z}}\times \frac{3\mathbb{Z}}{9\mathbb{Z}}\times \frac{3\mathbb{Z}}{9\mathbb{Z}}\times \frac{\mathbb{Z}}{3\mathbb{Z}}.
    \end{equation*}
    In fact, $\GLtwo(\mathbb{F}_3)$ acts on $V_3$, and $\widehat{\eta_3}(y)$ depends only on the orbit of $y$ under this action. More precisely,
    \begin{equation*}
        \widehat{\eta_3}(y) = \begin{cases}
            1+3^{-1}, &\text{ if $y\in \mathcal{O}_0:=\{0\}\subseteq V_3$},\\
            3^{-1}, &\text{ if $y$ lies in the orbit $\mathcal{O}_1$ of $(1,0,0,0)\in V_3$,}\\
            0, &\text{ else.}
        \end{cases}
    \end{equation*}
\end{lemma}
\begin{proof}
    Let 
    \begin{equation*}
        \langle x,y\rangle =\mathrm{exp}\left(2\pi i\frac{[x,y]}{3}\right),
    \end{equation*}
    where $[x,y]=x_1y_1+x_2y_2/3+x_3y_3/3+x_4y_4$ satisfies $[gx,g^Ty]=[x,y]$, cf. \cite[Section 3]{TTOrb}. This property already implies that the Fourier transform only depends on the orbit. Denote the orbit of $(1,0,0,0)$ by $\mathcal{O}_1$. Now, similarly to the proof of \cite[Proposition 1]{TTOrb}, one may consider the Fourier transform of the indicator function of $\mathcal{O}_1$. We find that
    \begin{equation*}
        3^4\widehat{\mathbf{1}_{\mathcal{O}_1}}(x) = \sum_{y\in V_3\cap \mathcal{O}_1}\langle x,y \rangle.
    \end{equation*}
    A short computation shows that the stabiliser of $(1,0,0,0)$ has order $6$, whence the above equals
    \begin{equation*}
        \frac{1}{6}\sum_{g\in \GLtwo(\mathbb{F}_3)}\langle x,g^{-T}(1,0,0,0) \rangle = \frac{1}{6}\sum_{g\in \GLtwo(\mathbb{F}_3)}\langle gx,(1,0,0,0) \rangle =\frac{1}{12}\sum_{g\in \GLtwo(\mathbb{F}_3)}\sum_{t=\pm 1}\langle gx,(t,0,0,0) \rangle.
    \end{equation*}
    Now, the innermost sum above equals $2$ if and only if $gx$ has a root at $v=0$, else it equals $-1$. For every root of $x$, $gx$ will have a root at $v=0$ for precisely $\#\GLtwo(\mathbb{F}_3)/4=12$ choices of $g$. Hence,
    \begin{equation*}
        3^4\widehat{\mathbf{1}_{\mathcal{O}(1)}}(x) = 2\eta_3(x)+ (\eta_3(x)-4)=3\eta_3(x)-4.
    \end{equation*}
    Applying the inverse transform gives the desired formula for $\widehat{\eta_3}(x)$.
\end{proof}
\begin{remark}\label{remarkorbit}
    Unlike the case $q\neq 3$, it is not true that the orbit of $(1,0,0,0)$ consists of all the $(1^3)$-points modulo $3$. Instead, this orbit consists only of elements of the type $(\alpha^3,3\alpha^2\gamma,3\alpha\gamma^2,\gamma^3)$ for $\alpha,\gamma$ not both congruent to zero modulo $3$.
\end{remark}
\begin{remark}\label{remarkdisc3}
    For $y$ which are maximal, lying in $\mathcal{O}_1$ is equivalent to the $3$-part of $\mathrm{Disc}(y)$ equalling $3^5$, see \cite[Proposition 8.20]{TT1}.
\end{remark}
In the general case we then have that
\begin{equation*}
    \widehat{\eta_{ab}} = \sum_\sumstack{\ell\mid ab/(ab,3)\\ \ell_3\mid (3,ab)}\frac{\sigma(ab/\ell\ell_3)}{ab}\mathbf{1}_{ab/(\ell(3,ab)),(0)}\mathbf{1}_{\ell,(1^3)}\mathbf{1}_{(3,ab)/\ell_3,\mathcal{O}_0}\mathbf{1}_{\ell_3,\mathcal{O}_1}.
\end{equation*}
Using the above, we may expand
\begin{equation*}
    \xi_{\pm}(s, \mathbf{1}_{p,(1^3)}\otimes \widehat{\eta_{ab}}) = \sum_\sumstack{\ell\mid ab/(ab,3)\\ \ell_3\mid (3,ab)}\frac{\sigma(ab/\ell\ell_3)}{ab}\xi_{\pm}(s, \mathbf{1}_{p\ell,(1^3)} \otimes \mathbf{1}_{(3,ab)/\ell_3,\mathcal{O}_0}\otimes \mathbf{1}_{\ell_3,\mathcal{O}_1} \otimes \mathbf{1}_{ab/(\ell(3,ab)),(0)}).
\end{equation*}
We now recall that the zeta function above only counts forms with at least a triple root modulo $3$. Hence, we write the above as
\begin{equation*}
\begin{split}
   \sum_\sumstack{\ell\mid ab/(ab,3)\\ \ell_3\mid (3,ab)}&\frac{\sigma(ab/\ell\ell_3)}{ab}\xi_{\pm}(s, \mathbf{1}_{3p\ell/(3,ab),(1^3)} \otimes \mathbf{1}_{(3,ab)/\ell_3,\mathcal{O}_0}\otimes \mathbf{1}_{\ell_3,\mathcal{O}_1} \otimes \mathbf{1}_{ab/(\ell(3,ab)),(0)})\\&=\sum_\sumstack{\ell\mid ab/(ab,3)\\ \ell_3\mid (3,ab)}\left(\frac{\ell\ell_3}{ab}\right)^{4s}\frac{\sigma(ab/\ell\ell_3)}{ab}\xi_{\pm}(s, \mathbf{1}_{3p\ell/\ell_3,(1^3)} \otimes \mathbf{1}_{\ell_3,\mathcal{O}_1} ).
\end{split}
\end{equation*}
We shift the integral involving this zeta function to $\Re(s)=1+\epsilon$. Note that we also pick up a residue at $s=1$, but we shall not need to compute this residue explicitly. 

Next, we separate
\begin{equation*}
    \xi_{\pm}(s, \mathbf{1}_{3p\ell/\ell_3,(1^3)} \otimes \mathbf{1}_{\ell_3,\mathcal{O}_1} ) = \xi_{\pm}^{S_3}(s, \mathbf{1}_{3p\ell/\ell_3,(1^3)} \otimes \mathbf{1}_{\ell_3,\mathcal{O}_1} )+\xi_{\pm}^{S_3^c}(s, \mathbf{1}_{3p\ell/\ell_3,(1^3)} \otimes \mathbf{1}_{\ell_3,\mathcal{O}_1} ).
\end{equation*}
We may bound the integral over the zeta function associated with non-$S_3$ fields using crude bounds. Indeed, we may bound absolutely using the general fact that for a zeta function $Z(s)$, with $s$ lying inside its domain of convergence, one has
\begin{equation*}
    Z(s)= s\int_{1}^\infty \frac{A_Z(t)}{t^{s+1}}dt,
\end{equation*}
where $A_Z(t)$ is the summatory function of $Z$. The summatory function associated with $\xi_{\pm}^{\mathrm{S_3^c}}(s, \mathbf{1}_{p,(1^3)}\otimes \widehat{\eta_{ab}})$ is bounded in absolute value by  the summatory function $A_{p,S_3^c}(t)$ of $\xi_{\pm}^{\mathrm{S_3^c}}(s, \mathbf{1}_{p,(1^3)})$. Combining \eqref{suborderctdw} with \cite[Lemma 4.1]{BTT} shows that the $C_3$-contribution to $A_{p,S_3^c}(t)$ is
\begin{equation}\label{c3sumbd}
    \ll t^{1/2}/p^{1-\epsilon}.
\end{equation}
Next, we study the contribution from forms with precisely one root. Without loss of generality, we may consider orbits under $\SLtwo(\mathbb{Z})$ instead of $\GLtwo(\mathbb{Z})$. A change of variables brings these cubic forms into the shape
\begin{equation}\label{redshape}
    u(a_1u^2+a_2uv+a_3v^2),
\end{equation}
with $0 \leq a_2<2\abs{a_3}$ and $a_3 >  0$. This form having a triple root modulo $p$ translates into the conditions $p\mid a_3$ and $p\mid a_2$. The discriminant of such a form equals
\begin{equation*}
    \ a_3^2(a_2^2-4a_1a_3).
\end{equation*}
Through elementary computations, one checks that the number of $a_1,a_2,a_3$ satisfying these conditions is 
\begin{equation*}
    \ll t/p^3.
\end{equation*}
In fact, if we also require that $a_2^2\neq 4a_1a_3$, then one may compute the number more precisely as
\begin{equation}\label{redcount}
     d_1\frac{t}{p^3}+\mathcal{O}\left(\frac{t^{2/3}}{p^2}\right),
\end{equation}
for a constant $d_1$. See also Lemma \ref{redshintlemma} for a more accurate result. For future reference, we note that the above asymptotic also bounds the corresponding residue at $s=1$ by $\ll p^{-3}$. We remark that one studies forms with three roots similarly.

These bounds are already sufficient to show that the contribution from the shifted integral over $\xi_{\pm}^{\mathrm{S_3^c}}(s, \mathbf{1}_{p,(1^3)}\otimes \widehat{\eta_{ab}})$ to the one-level density is negligible, for $\sigma < 1$. In fact, it turns out the zeta function $\xi_{\pm}^{S_3^c}(s, \mathbf{1}_{3p\ell/\ell_3,(1^3)} \otimes \mathbf{1}_{\ell_3,\mathcal{O}_1} )$ can be analytically continued past the line $\Re(s)=1$. Indeed, this essentially follows from the estimate \eqref{redcount},  in addition to an argument taking care of the $\mathbb{Q}^3$-rings. Having established the above bounds, and the analytic continuation to $\Re(s) > 2/3$, we conclude that the contribution to the one-level density, from the part of the residue we picked up at $s=1$ coming from $\xi_{\pm}^{\mathrm{S_3^c}}(s, \mathbf{1}_{p,(1^3)}\otimes \widehat{\eta_{ab}})$, is also negligible when $\sigma < 1$.

One could perform similar estimations for studying the zeta function associated with $S_3$-fields, using that the number of cubic fields which are totally ramified at $p$ of discriminant bounded by $t$ is $\ll t/p^2$, and this would be sufficient to allow for support $\sigma < 2/3$. However, to allow for support $\sigma < 1$, we need to study this term in a different manner.

\subsection{Counting subrings with local conditions}
We now study $\xi_{\pm}^{S_3}(s, \mathbf{1}_{3p\ell/\ell_3,(1^3)} \otimes \mathbf{1}_{\ell_3,\mathcal{O}_1} )$ by switching to maximal overrings. For this, we need to extend \eqref{suborderctdw} to handle certain local conditions. First, given an $S_3$-cubic field $K$ and a squarefree $d$, we wish to study the zeta function
\begin{equation}\label{zetasubram}
    \sum_\sumstack{\mathcal{O}\subseteq \mathcal{O}_K }\frac{\mathbf{1}_{d,(1^3)}}{[\mathcal{O}_K:\mathcal{O}]^s}.
\end{equation}
Here, $\mathbf{1}_{(1^3),d}$ counts only orders which correspond to forms of type $(1^3)$ modulo every $q\mid d$. By multiplicativity, studying this zeta function amounts to studying it locally at some prime. For primes coprime to $d$, we obtain the same local factors as in \eqref{suborderctdw} so that we do not need to consider these. Hence, we restrict ourselves to a prime $q\mid d$.

Our goal is then to find the number of index-$q^{m}$ suborders of $K$ with splitting type $(1^3)$ at $q$. We study this number by using a criterion from the proof of \cite[Lemma 13]{BST} together with \cite[Eq. (10)]{BST}. Let $\langle 1,\omega,\theta\rangle$ be a normal basis of $\mathcal{O}_K$, i.e. a basis such that $\omega\theta \in \mathbb{Z}$, and let $y = (a,b,c,d)$ be the binary cubic form corresponding to this choice of basis. The fourth coordinate $d$ of $y$ should not be confused with the modulus of $\mathbf{1}_{d,(1^3)}$. Consider the $\mathbb{Z}$-module
\begin{equation*}
    \mathbb{Z} \oplus \mathbb{Z}(q^i\omega)\oplus \mathbb{Z}(q^j\theta)
\end{equation*}
and note that without loss of generality, we may assume that $i\geq j$. We call such a ring an $(i,j)$-subring. This module forms a ring if and only if
\begin{equation*}
    d\equiv 0 \nsmod{q^{i-2j}},
\end{equation*}
and the form corresponding to the ring with basis $\langle 1,q^i\omega,q^j\theta\rangle$ is $(q^{2i-j}a,q^ib,q^jc,q^{2j-i}d)$. As $i\geq 1$ and $i\geq j$, this new form automatically has at least a double root at $v=0$. Hence, it is of type $(1^3)$ if and only if $i\geq2j$, $d\not\equiv 0\nsmod q^{i-2j+1}$, and either $c\equiv0 \nsmod{q}$, or $j\geq 1$.

Now, we also need to consider the question of when making a change of basis to $\langle1,\omega',\theta'\rangle$ gives a different subring $\langle 1,q^i\omega',q^j\theta'\rangle$.  If one writes $\omega' = \alpha\omega+\beta\theta$ and $\theta' = \gamma\omega+\delta\theta$, then $\langle 1,q^i\omega',q^j\theta'\rangle$ is a ring if and only if the form corresponding to the normal basis $\{1,\omega',\theta'\}$ has a root at $(0:1) \nsmod{q^{i-2j}}$, i.e. if and only if 
\begin{equation}\label{welldefeqsubord}
    y(\gamma,\delta) = \gamma^3a+\gamma^2 \delta b+\gamma \delta^2c+\delta^3d\equiv 0\nsmod{q^{i-2j}},
\end{equation}
as this is the fourth coordinate of this form. Note that the above condition is precisely the condition that $(\gamma:\delta)$ is a root of $y$ modulo $q^{i-2j}$. A short computation shows that the two rings $\langle 1,q^i\omega',q^j\theta'\rangle$ and $\langle 1,q^i\omega,q^j\theta\rangle$ are equal precisely when $\gamma\equiv 0\nsmod{q^{i-j}}$. In particular, this shows that subrings associated with different roots are not equal, when $i > 2j$. Hence, we see that it suffices to consider
\begin{equation*}
    \begin{pmatrix}
        \alpha & \beta \\ 
        \gamma & \delta
    \end{pmatrix} \nsmod{q^{i-j}}.
\end{equation*}
When $i> j$, there are $q^{i-j}(q^{i-j})^2(1-q^{-1})^2$ matrices in $\GLtwo(\mathbb{Z}/q^{i-j}\mathbb{Z})$ with $\gamma\equiv 0\nsmod{q^{i-j}}$. As
\begin{equation*}
    \# \GLtwo(\mathbb{Z}/q^{i-j}\mathbb{Z}) = q^{4i-4j}(1-q^{-1})(1-q^{-2}),
\end{equation*}
we see that there are $q^{i-j}(1+q^{-1}) =q^{j}(1+q^{-1}) $ $(i,j)$-subrings when $i=2j$. For the case $i > 2j$, we note that there are $q^{j}q^{i-j}(q^{i-j})^2(1-q^{-1})^2$ matrices in $\GLtwo(\mathbb{Z}/q^{i-j}\mathbb{Z})$ which fix the root $(0:1) \nsmod{q^{i-2j}}$, as these are precisely the matrices with $\gamma\equiv 0\nsmod{q^{i-2j}}$. Hence, there are $q^j$ distinct subrings associated with each root of $y$.

Finally, we elaborate on the connection between roots modulo $q$ and roots modulo a power $q^{\ell}$. If $y$ is a form, then each simple root lifts uniquely to a root modulo $q^{\ell}$, by Hensel's lemma. If $y$ has a root of at least multiplicity $2$ at $(0:1)$, then it is not possible for the fourth coordinate to be congruent to $0 \nsmod{q^2}$ as $y$ is maximal, see \cite[Lemma 13]{BST}. We remark that using the information above, one may reprove \eqref{suborderctdw}.

We now compute the local factor at $q\mid d$ of the zeta function above by considering the various possibilities for $i$ and $j$. First, there is the trivial case when $j=i=0$, which contributes to the sum only if $\mathcal{O}_K$ is totally ramified at $q$. Next, if $j=0$ and $i \geq 1$, then the conditions on $c$ and $d$ mean that the form has at least a double root at $u=0$, which should lift to a root modulo $q^{i-2j}$. As $y$ is maximal, this is only possible if $i=1$.

It remains to handle the case when $j\geq 1$. Then there is no condition on the coefficient $c$. In the case $i=2j$, the condition on $d$ is equivalent to the form not having a root at $u=0$. Making the change of basis described above, the last coefficient of the resulting form is $y(\gamma,\delta)$, so that the condition on the fourth coefficient translates to the condition that $(\gamma,\delta)$ should not be a root of $y$. We then obtain $q^{j}(1+q^{-1}-q^{-1}\eta_q(y))$ distinct subrings.

In the final case $i > 2j$, we require that 
\begin{equation*}
    y(\gamma,\delta)\equiv 0\nsmod{q^{i-2j}},\,\,\, y(\gamma,\delta)\not\equiv 0 \nsmod{q^{i-2j+1}}.
\end{equation*}
A simple root lifts uniquely so that there are $q^{j}(1-q^{-1})$ distinct subrings associated with every such root. For the multiple roots, there is no lift to higher prime powers, so that these only contribute when $i=2j+1$, and the second condition above is then automatically satisfied. A computation now proves the following lemma.
\begin{lemma}\label{totramsubordct}
    The local factor at a prime $q\mid d$ of the zeta function \eqref{zetasubram} is given by
    \begin{equation*}
    \frac{1}{(1-q^{1-3s})(1-q^{-s})}\cdot 
    \begin{cases}
        1- q^{-2s}, &\text{ for $q$ of splitting type $(1^3)$ in $K$,}\\
        q^{-s}\big(1-q^{-s}+q^{-2s}(q-1)\big), &\text{ for $q$ of splitting type $(1^21)$ in $K$,}\\
        q^{-3s}(1-q^{-s})(q+1), &\text{ for $q$ of splitting type $(3)$ in $K$,}\\
        q^{-3s}(q-q^{-s}), &\text{ for $q$ of splitting type $(21)$ in $K$,}\\
        q^{-3s}\big(q-2+2q^{1-s}-q^{-s}\big), &\text{ for $q$ of splitting type $(111)$ in $K$,}\\
    \end{cases}
    \end{equation*}
\end{lemma}
\begin{proof}
    The proof is a straightforward computation, using the discussion above. We consider the case when $q$ has splitting type $(111)$ in $K$, and omit the other cases. 
    
    In the case $(111)$, the form $y$ has three distinct roots modulo $q$. Then, the discussion above shows that only $j\geq 1$ contributes to the sum, and that the local factor is given by
    \begin{equation*}
        \sum_{j\geq 1}q^{-js}\left(q^{-2js}q^j(1-2q^{-1})+3\sum_{i\geq 2j+1}q^{-is}q^{j}(1-q^{-1})\right),
    \end{equation*}
    which gives the desired factor. The other splitting types are handled similarly.
\end{proof}

At $q=3$, we see from Lemma \ref{eta3fourier} that we need to consider a slightly more refined suborder-counting zeta function. Specifically, we modify \eqref{zetasubram} into the zeta function
\begin{equation}\label{suborderzeta3}
    \sum_\sumstack{\mathcal{O}\subseteq \mathcal{O}_K }\frac{\mathbf{1}_{(1^3),d}\mathbf{1}_{\mathcal{O}_1,e}}{[\mathcal{O}_K:\mathcal{O}]^s},
\end{equation}
where $e=1$ or $3$, with the definition that $\mathbf{1}_{\mathcal{O}_1,1}$ always equals $1$ and that $\mathbf{1}_{\mathcal{O}_1,3}$ is the indicator function of forms which are of type $(1^3)$ at $3$, and whose reduction to
\begin{equation}\label{projspace3}
    \frac{\mathbb{Z}}{3\mathbb{Z}} \times \frac{3\mathbb{Z}}{9\mathbb{Z}}\times \frac{3\mathbb{Z}}{9\mathbb{Z}}\times \frac{\mathbb{Z}}{3\mathbb{Z}}
\end{equation}
is nonzero, and of the form $(\alpha^3,3\alpha^2\gamma,3\alpha\gamma^2,\gamma^3)$ for $\alpha,\gamma \in \mathbb{F}_3$, see \eqref{eta3fourier} and Remark \ref{remarkorbit}.

As above, if $\mathcal{O}_K$ is represented by the form $(a,b,c,d)$ then its $(i,j)$-subring is represented by the form $(3^{2i-j}a,3^ib,3^jc,3^{2j-i}d)$. The form $(a,b,c,d)$ lying in $\mathcal{O}_1$ is equivalent to $3^5\mid \mathrm{Disc}(K)$, see Remark \ref{remarkdisc3}. Assuming $i\geq 1$ and $i\geq j$, we have that $3\mid 3^{2i-j}$ so that we only need to study which forms reduce to $(0,0,0,\gamma^3)=(0,0,0,\gamma)$ in \eqref{projspace3}, with $\gamma \neq 0$. Hence, we need that $i\geq 2j$ and that
\begin{equation*}
    3^ib, 3^j c\equiv 0 \nsmod{9},\,\,\, d\equiv 0 \nsmod{3^{i-2j}}, \,\,\, d\not\equiv 0 \nsmod{3^{i-2j+1}}.
\end{equation*}
As before, $j=0$ requires that $i=1$ (unless $i=0$), by maximality. Then, we need $3\mid b$, $9\mid c$, $3\mid d$ and $9\nmid d$. Maximal forms satisfying these conditions correspond to cubic fields where discriminant is divisible by $3^4$, cf. \cite[Proposition 8.20]{TT1}.

When $j=1$, the only possibility is that $i=2$, or $i=3$, by maximality. When $i=2$, the conditions above become that $3\mid c$ and $3\nmid d$. In the case $i=3$, the condition translates to having at least a double root modulo $3$ at $u=0$. When $j\geq 2$ and $i\geq 2j$ there are no conditions on $b$ or $c$. One obtains the following lemma.
\begin{lemma}\label{3adiclemma}
    When $K$ is totally ramified at $3$, the local factor at $3$ of the zeta function \eqref{suborderzeta3}, with $e=3$ equals
    \begin{equation*}
        \frac{1}{1-3^{1-3s}}\cdot \begin{cases}
            1+3^{-s}, \text{ if $3^5\mid\mid \mathrm{Disc}(K)$},\\
            3^{-s}+3^{1-3s}, \text{ if $3^4\mid\mid \mathrm{Disc}(K)$}, \\
            3^{1-3s}+3^{1-4s},\text{ if $3^3\mid\mid \mathrm{Disc}(K)$},
        \end{cases}
    \end{equation*}
    and the case when $3^2\mid\mid \mathrm{Disc}(K)$ is not possible. For the other splitting types, the local factor is given by
    \begin{equation*}
        \begin{cases}
            3^{-3s}+3^{1-4s}+3^{1-6s}\big(2+3^{1-s}-3^{1-2s}\big)\big(1-3^{-s}\big)^{-1}\big(1-3^{1-s}\big)^{-1}, &\text{ for $K$ of type $(1^21)$},\\
            3^{-3s}\big(1+3^{2-3s}\big)\big(1-3^{-s}\big)^{-1}, &\text{ for $K$ of type $(3)$},\\
            3^{-3s}\big(1-3^{-s}+6\cdot3^{-3s}\big)\big(1-3^{-s}\big)^{-1}\big(1-3^{1-3s}\big)^{-1}, &\text{ for $K$ of type $(21)$},\\
            3^{-3s}\big(1-3^{-s}+18\cdot 3^{-4s}\big)\big(1-3^{-s}\big)^{-1}\big(1-3^{1-3s}\big)^{-1}, &\text{ for $K$ of type $(111)$}.
        \end{cases}
    \end{equation*}
\end{lemma}
\begin{proof}
The only case that cannot be treated analogously to Lemma \ref{totramsubordct} is counting rings with $i=2$ and $j=1$. To treat this case, we note the following facts. First, any binary cubic form is $\GLtwo(\mathbb{Z})$-equivalent to one which satisfies $3\mid c$. Second, if the maximal form $y$ has its third coordinate divisible by $3$, then $gy$, with 
\begin{equation*}
    g = \begin{pmatrix}
        \alpha & \beta \\ 
        \gamma & \delta
    \end{pmatrix},
\end{equation*}
does as well, if and only if either $y$ is of type $(1^3)$ at $3$, or $3\mid \gamma$. The $(1^3)$-case can be handled as above, so we may assume that $K$ is not totally ramified at $3$. If $K$ is of type $(1^21)$ at $3$, then we see that $3$ divides the fourth coordinate as soon as it divides the third, contradicting $d\not\equiv 0 \nsmod{3}$. In the unramified cases, we see that there is one subring of type $(2,1)$.
\end{proof}

\subsection{Bounding the dual $S_3$-contribution}\label{dualS3onelevsect}
We now return to studying $\xi_{\pm}^{S_3}(s, \mathbf{1}_{3p\ell/\ell_3,(1^3)} \otimes \mathbf{1}_{\ell_3,\mathcal{O}_1} )$. We do this by switching to maximal overrings, with the help of Lemmas \ref{totramsubordct} and \ref{3adiclemma}. First, a short calculation using Lemma \ref{totramsubordct} shows that we may assume that the rings counted by $\xi_{\pm}^{S_3}(s, \mathbf{1}_{3p\ell/\ell_3,(1^3)} \otimes \mathbf{1}_{\ell_3,\mathcal{O}_1} )$ are maximal at $p$, as those which are both nonmaximal and of type $(1^3)$ are negligible, when $\sigma < 2$. Hence, up to a small error, we may write
\begin{equation}\label{s3fcndecomponelev}
    \xi_{\pm}^{S_3}(s, \mathbf{1}_{3p\ell/\ell_3,(1^3)} \otimes \mathbf{1}_{\ell_3,\mathcal{O}_1} ) = \sum_{K \text{ $S_3$-cubic}}\frac{1}{\abs{\mathrm{Disc}(K)}^s}\sum_{m=1}^\infty \frac{\lambda_m(p,\ell,\ell_3,K)}{m^{2s}},
\end{equation}
where the coefficient $\lambda_m(p,\ell,\ell_3,K)$ counts the number of index $m$-suborders of $\mathcal{O}_K$ where $3p\ell$ is of type $(1^3)$ and $\ell_3$ is of type $\mathcal{O}_1$. This coefficient depends only on $m$ and the splitting of all primes dividing $m$ in $K$, by \eqref{suborderctdw} and Lemmas \ref{totramsubordct} and \ref{3adiclemma}. Moreover, by our earlier discussion, we may assume that the coefficient above is nonzero only if $p$ is totally ramified in $K$. We remark that to be precise, one should weigh the coefficients above with a weight depending on $\mathrm{sgn}(\mathrm{Disc}(K))$. We separate the two cases for the discriminant and bound each one individually.

Next, we sum over all possible splitting types in $K$ of all primes dividing $m$. More precisely, we write the above as
\begin{equation*}
    \sum_\sumstack{K \text{ $S_3$-cubic}\\\mathrm{sgn}(\mathrm{Disc}(K)) = \pm 1}\frac{1}{\abs{\mathrm{Disc}(K)}^s}\sum_{m=1}^\infty m^{-2s} \sum_{(S_i)_{q\mid m}}\lambda_m(\ell,\ell_3,(S_i)_{q\mid m})\mathbf{1}_{p \text{ of type $(1^3)$ in $K$}}\mathbf{1}_{\text{$m$ of type $(S_i)_{q\mid m}$ in $K$}},
\end{equation*}
where $(S_i)_{q\mid m}$ is a list of splitting types associated with each prime dividing $m$. We remark that there are $\ll_\epsilon m^\epsilon$ configurations of splitting types for each $m$. When $q=3$, $S_i$ is not necessarily a splitting type, but may instead be a collection of local specifications. Interchanging the order of summation yields
\begin{equation*}
    \sum_{m=1}^\infty m^{-2s} \sum_{(S_i)_{q\mid m}}\lambda_m(\ell,\ell_3,(S_i)_{q\mid m}) \sum_\sumstack{K \text{ $S_3$-cubic}\\ \mathrm{sgn}(\mathrm{Disc}(K)) = \pm 1}\frac{\mathbf{1}_{p \text{ of type $(1^3)$ in $K$}}}{\abs{\mathrm{Disc}(K)}^s}\mathbf{1}_{\text{$m$ of type $(S_i)_{q\mid m}$ in $K$}}.
\end{equation*}
To study the innermost sum we apply a result by Shankar--Thorne, see \cite[Theorem 11]{ShankarThorne} and the remarks preceding its proof. More precisely, write $P_{S}$ for the product of all primes dividing $m$ associated with the splitting type $S$, so that $\mathrm{sqf}(m)=P_{(111)}P_{(21)}P_{(3)}P_{(1^21)}P_{(1^3)}$, assuming $3\nmid m$ (the case $3\mid m$ is treated similarly). Then,
\begin{equation}\label{Shankarthorneappl}
\begin{split}
    &\sum_\sumstack{K \text{ $S_3$-cubic}\\ \mathrm{sgn}(\mathrm{Disc}(K))=\pm1}\frac{\mathbf{1}_{p \text{ of type $(1^3)$ in $K$}}}{\abs{\mathrm{Disc}(K)}^s}\mathbf{1}_{\text{$m$ of type $(S_i)_{q\mid m}$ in $K$}} = \sum_\sumstack{\text{$P_{(1^21)}\mid d$ fundamental}\\ \mathrm{sgn}(d) = \pm 1 \\ \text{ $K_d$ inert at $P_{(21)}$}\\K_d \text{ split at $P_{(111)}P_{(3)}$}\\ (d,P_{(1^3)})=1}\frac{1}{\abs{d}^s}\sum_{r\mid P_{(3)}}\mu(r)\frac{1}{c_d3^{\omega(rP_{(111)})}}
    \\ \times &\left(\frac{1}{2}M_{1,d}(s)\prod_{q\mid pP_{(1^3)}}\left(\frac{1+\chi_{-3d}(q)}{q^{2s}}\right)\prod_{q\nmid 3dmp}\left(1+\frac{1+\chi_{-3d}(q)}{q^{2s}}\right)+\sum_{E\in \mathcal{L}_3(rP_{(111)},d) }M_{2,E}(s)\frac{\theta_E(pP_{(1^3)})}{(pP_{(1^3)})^{2s}}\prod_{ q\nmid 3dmp}\left(1+\frac{\theta_E(q)}{q^{2s}}\right)\right).
\end{split}
\end{equation}
Here, $c_d$ equals $1$ if $d < -3$, else it equals $3$, $M_{1,d}(s)$ and $M_{2,E}(s)$ are $3$-Euler factors explicitly given in \cite[Table 1]{ShankarThorne}. The set $\mathcal{L}_3(e,d)$ consists of those cubic fields whose discriminant equals $-d$, $-d/3$, $-3d$, or $-27d$ times a square dividing $e^2$. Finally, recall that the multiplicative character $\theta_E(q)$ equals $2$ if $q$ splits completely in $E$, it equals $-1$ if $q$ is inert in $E$, otherwise, it equals zero. 

The character $\theta_E$ is essentially equal to the Artin character of $E$, at primes, and this is what allows us to surpass the support $\sigma < 2/3$. Indeed, we may now use the GRH once again, and sum the characters $\theta_E(p)$ and $\chi_{-3d}(p)$ over $p$. More precisely one has e.g. that
\begin{equation}\label{lastGRHonelev}
    \sum_{p\leq X^\sigma}\widehat{\phi}\left(\frac{\log p}{L}\right)\frac{\theta_E(p)p^{2s}\log p}{p^{3/2}}\ll X^{\sigma+\epsilon} \log\abs{\mathrm{Disc}(E)},
\end{equation}
when $\Re(s) \in [1,1+\epsilon]$, using the GRH and summation by parts, cf. Lemma \ref{lemmaGRH1}. One obtains a bound of the same strength for the average over $\chi_{-3d}(p)$, when $d\neq -3$, and one bounds the contribution from $d=-3$ separately by shifting the contour. Moreover, from \cite[Proposition 4.5]{BTT}, it follows that
\begin{equation*}
    \sum_{\abs{d}\leq t} \#\mathcal{L}_3(e,d) \ll_\epsilon te^{\epsilon}.
\end{equation*}
Combining this with the bound $\lambda_m(\ell,\ell_3,(S_i)_{q\mid m})\ll_\epsilon m^{1/3+\epsilon}$, it follows that the contribution involving the characters $\theta_E(p)$ and $\chi_{-3d}(p)$ to the one-level density can be bounded by 
\begin{equation*}
    \ll X^{\sigma-1+\epsilon} \sum_{n\leq X^\delta}\sum_{abc=n}\mu(b)(ab)^{4}\sum_\sumstack{\ell\mid ab/(ab,3)\\ \ell_3\mid (3,ab)}\left(\frac{\ell\ell_3}{ab}\right)^{4}\sum_{m=1}^\infty m^{1/3+\epsilon-2}\ll X^{\sigma-1+\epsilon+5\delta},
\end{equation*}
which is negligible when $\sigma <1$, for small enough $\delta$ and $\epsilon$ as usual.

The remaining terms, which we do not average over $p$ yet, are
\begin{equation*}
\begin{split}
    &\sum_\sumstack{\text{$P_{(1^21)}\mid d$ fundamental}\\ \mathrm{sgn}(d) = \pm 1 \\ \text{ $K_d$ inert at $P_{(21)}$}\\K_d \text{ split at $P_{(111)}P_{(3)}$}}\frac{1}{\abs{d}^s}\sum_{r\mid P_{(3)}}\mu(r)\frac{1}{c_d3^{\omega(rP_{(111)})}}
    \frac{1}{2p^{2s}}M_{1,d}(s)\prod_{q\mid P_{(1^3)}}\left(\frac{1+\chi_{-3d}(q)}{q^{2s}}\right)\prod_{q\nmid 3dmp}\left(1+\frac{1+\chi_{-3d}(q)}{q^{2s}}\right).
\end{split}
\end{equation*}
By requiring that $d$ be coprime to $m/P_{(1^21)}$, we may replace the splitting conditions on $d$ by multiplying with the factors 
\begin{equation*}
 \prod_{q\mid P_{(111)}P_{(3)}}\left(\frac{1+\chi_d(q)}{2}\right)\prod_{q\mid P_{(21)}}\left(\frac{1-\chi_d(q)}{2}\right).   
\end{equation*}
We now handle the sum over $d$ using similar methods as in Section \ref{redringssect}, whence we will leave out most of the details. First, by using Lemma \ref{grhquadlemma}, we may ignore the contribution from terms involving a character $\chi_d$ or $\chi_{-3d}$, assuming that $\sigma  < 1$. What remains is 
\begin{equation*}
    \sum_{r\mid P_{(3)}}\mu(r)\frac{1}{3^{\omega(rP_{(111)})}}
    \frac{1}{2p^{2s}P_{(1^3)}^{2s}}\sum_\sumstack{\text{$P_{(1^21)}\mid d$ fundamental}\\ \mathrm{sgn}(d) = \pm 1 }\frac{c_d^{-1}}{\abs{d}^s}M_{1,d}(s)\prod_{q\nmid 3dmp}\left(1+\frac{1}{q^{2s}}\right).
\end{equation*}
This term is handled as in the proof of Proposition \ref{redcountprop}. One sums over $d$ using multiplicativity and obtains $\zeta(s)$ multiplied by a finite sum of Euler products which converges absolutely when $\Re(s) > 1/2$. One then shifts the contour from $(1+\epsilon)$ to $(1/2+\epsilon)$. 

We claim that the total residue contribution coming from the simple pole at $s=1$ cancels the previous residue contribution we obtained from shifting the integral, involving $\xi_{\pm}(s, \mathbf{1}_{3p\ell/\ell_3,(1^3)} \otimes \mathbf{1}_{\ell_3,\mathcal{O}_1} )$, to the line $\Re(s)=1+\epsilon$, up to a negligible factor. This follows from us having already bounded the residue contribution coming from the non-$S_3$ part of $\xi_{\pm}(s, \mathbf{1}_{3p\ell/\ell_3,(1^3)} \otimes \mathbf{1}_{\ell_3,\mathcal{O}_1} )$ by bounding the corresponding summatory function uniformly in $p$, see \eqref{c3sumbd} and \eqref{redcount}. Moreover, we have shown that the $S_3$-contribution equals \eqref{Shankarthorneappl}, up to terms whose summatory functions are negligible with respect to $p$. In fact, \eqref{lastGRHonelev} implies such averaged bounds for the summatory functions of the terms in \eqref{Shankarthorneappl} which involve the fields $E$. Hence, the non-negligible parts of the residues we obtained by shifting $\xi_{\pm}(s, \mathbf{1}_{3p\ell/\ell_3,(1^3)} \otimes \mathbf{1}_{\ell_3,\mathcal{O}_1} )$ must then all be concentrated within the terms from \eqref{Shankarthorneappl}, which do not involve $E$ or any of the nonprincipal quadratic characters, which proves our claim.

Finally, a short computation confirms that the shifted integral is negligible when $\sigma < 1$. Indeed, we have that 
\begin{equation*}
    \sum_{n\leq X^{\delta}}X^{-1/2+\epsilon+k}y^{-k}\sum_{p\leq X^\sigma} \widehat{\phi}\left(\frac{\log p}{L}\right)\frac{p^{1+2\epsilon}}{p^{3/2}p^{1-2\epsilon}}\ll X^{\sigma/2+\eta k+\delta-1/2+\epsilon}
\end{equation*}
is $o(1)$ when $\sigma < 1$, for sufficiently small $\eta$ and $\delta$. This concludes the proof of Theorem \ref{onelevQ}.

\section{Function fields}\label{ffields}
In this section, we give a very brief sketch on how to prove Theorem \ref{onelevfq} by implementing the proof of Theorem \ref{onelevQ} in the function field setting. Many details are left out, and we focus on highlighting the differences between $\mathbb{Q}$ and $\mathbb{F}_q$.

\subsection{Preliminaries and notation}
Throughout this section, $q$ will denote a prime power coprime to $2$ and $3$. Furthermore, $\fq$ is the finite field with $q$ elements, and $\fq(T)$ is its rational function field in the indeterminate $T$, i.e. the field of fractions of the polynomial ring $R:=\fq[T]$. The primes in $\fq(T)$ are all non-Archimedean, and they may be identified with the irreducible polynomials $P\in R$, together with the so-called prime at infinity $P_\infty \in \fq[T^{-1}]$. We denote the completion of $\fq(T)$ at $P_\infty$ by $K_\infty$.

Over a function field, the Dedekind zeta function $\zeta_L(s)$ of a (geometric) finite extension $L/\fq(T)$ takes the form
\begin{equation*}
    \zeta_L(u) = \frac{P_L(u)}{(1-qu)(1-u)},
\end{equation*}
where $u=q^{-s}$, and $P_L(u)$ is a polynomial. Moreover, the polynomial $P_L$ is of degree $2g_L$, where $g_L$ is the genus of $L$. The genus may be computed using the Riemann--Hurwitz formula, and the difference $2g_L-\log_q(\abs{\mathrm{Disc}(L)})$ is constant. The GRH, which is a theorem over function fields, asserts that all zeros of $P_L(u)$ have absolute value $q^{-1/2}$. We refer to \cite{Rosen} for a more thorough introduction to the theory of function fields. 
\subsection{Setting up the one-level density}
Over function fields, we define the one-level density slightly differently to over number fields. When $L$ is a finite extension of $\fq(T)$, the polynomial $P_L(u)$ has exactly $2g_L$ roots in $\mathbb{C}$ of the form $q^{-1/2}e^{i\psi_L}$, with $\psi_L\in \mathbb{R}$. The one-level density for the family of cubic fields with splitting type $S_\infty$ at $P_\infty$ is defined as 
\begin{equation*}
    \frac{1}{\#\mathcal{F}^{S_\infty}(X)}\sum_{L\in \mathcal{F}^{S_\infty}(X)}D_\phi(L)\,\,\,\,\text{ with }\,\,\,\, D_\phi(L)=\sum_{\psi_L}\phi\left(\frac{2g_L}{2\pi}\psi_L\right),
\end{equation*}
where $X=q^{2n}$ for $n\geq 1$. We remark that $\psi_L$ ranges over all elements in $\mathbb{R}$ such that $q^{-1/2}e^{i\psi_L}$ is a root of $P_L(u)$, not only over their representatives in $\mathbb{R}/2\pi \mathbb{Z}$. By applying the function field version of the explicit formula, see e.g. \cite[Section 6.2]{Ahlquist2}, one sees that the one-level density equals
\begin{equation*}
    \widehat{\phi}(0)-\frac{\phi(0)}{2}-\frac{2}{2g_L\#\mathcal{F}^{S_\infty}(X) }\sum_{m\geq 1}\widehat{\phi}\left(\frac{m}{2g_L}\right)q^{-m/2}m\sum_{\deg P=m}\sum_{L\in \mathcal{F}^{S_\infty}(X)}\theta_L(P)+o(1),
\end{equation*}
where $\theta_L$ equals $2$ if $P$ is completely split, $-1$ if $P$ is inert, and else it equals zero, cf. the computations in Section \ref{sectonelev}. Our goal is now to bound the contribution from the sum involving $\theta_L(P)$ as $o(1)$ when $\sigma < 1$.

\subsection{Sieving for maximality and geometry of numbers}\label{dualfcnfldonelev}
We now study the sum over $\theta_L(P)$. First, by applying the Delone-Faddeev correspondence, we may sum over maximal forms instead of cubic fields. Next, we sieve for maximality and write
\begin{equation*}
    \sum_{L\in \mathcal{F}^{S_\infty}(X)}\theta_L(P) = \sum_{F}\mu(F)\sum_{x\in \mathcal{F}_{\text{nm}}^{S_\infty}(F;X)}\theta_x(P),
\end{equation*}
where $F$ ranges over all monic squarefree polynomials, and $\mathcal{F}_{\text{nm}}^{S_\infty}(F;X)$ contains (isomorphism classes of) forms which are nonmaximal at all primes dividing $F$. Now, similar to the computations over $\mathbb{Q}$, the sum over $F$ satisfying $\abs{F}\geq X^{\delta}$ is seen to be negligible, by using the GRH. Hence, we may study
\begin{equation*}
    \sum_{\abs{F}\leq X^{\delta}}\mu(F)\sum_{x\in \mathcal{F}_{\text{nm}}^{S_\infty}(F;X)}\theta_P(x)
\end{equation*}
At this point, we also add the non-$S_3$ contribution, and the contribution corresponding to non-orders. As there are very few $C_3$-cubic fields, these are negligible. Moreover, the contribution from rings corresponding to subrings of $R^3$ is negligible when $\sigma < 1$, just as over $\mathbb{Q}$. Indeed, the local factors in the zeta functions counting suborders of a field remain the same, except that one replaces $p$ by $P$. Hence, we may write
\begin{equation*}
    \sum_{\abs{F}\leq X^{\delta}}\mu(F)\sum_{x\in \mathcal{G}_{\text{nm}}^{S_\infty}(F;X)}\theta_P(x)-\sum_{\abs{F}\leq X^{\delta}}\mu(F)\sum_{x\in \mathcal{Q}_{\text{nm}}^{S_\infty}(F;X)}\theta_P(x),
\end{equation*}
where $\mathcal{G}_{\text{nm}}^{S_\infty}(F;X)$ drops the $S_3$-cubic condition on the forms and contains all forms which are nonmaximal at $F$ of discriminant $X$, and its subset $\mathcal{Q}_{\text{nm}}^{S_\infty}(F;X)$ contains only forms whose corresponding maximal ring is the product of $R$ and the ring of integers in a quadratic field. All forms $x$ above are weighted by the inverse size of their stabilisers. An analogue of \eqref{Shintanidiscred} in the function field setting implies that the above equals
\begin{equation}\label{subquad}
    \sum_{\abs{F}\leq X^{\delta}}\mu(F)\sum_{fg\mid F}\mu(g)\sum_{\mathcal{G}_{\text{nm}}^{S_\infty}(X\abs{f}^2/\abs{F}^4)}\eta_{fg}(x)\theta_P(x)-\sum_{\abs{F}\leq X^{\delta}}\mu(F)\sum_{fg\mid F}\mu(g)\sum_{x\in \mathcal{Q}_{\text{nm}}^{S_\infty}(X\abs{f}^2/\abs{F}^4)}\eta_{fg}(x)\theta_P(x).
\end{equation}
The second of the two sums above can be estimated using identical methods to those used in Section \ref{quadringsect}.

We now briefly describe the geometry-of-numbers methods used in \cite{Ahlquist2}, based on the methods developed in \cite{BST} and \cite{BSW}, which we will use to study the sum over $\mathcal{G}_{\text{nm}}^{S_\infty}(F;X)$. The idea is that in order to count isomorphism classes of cubic forms, one may essentially count elements in $\mathcal{H}v_{S_\infty}$, where $v_{S_\infty}$ is a binary cubic form with coefficients in $K_\infty$ such that its corresponding splitting type at $P_\infty$ is $S_\infty$, and $\mathcal{H}$ is a fundamental domain of $ \GLtwo(R)\setminus\GLtwo(K_\infty)$. One averages this over several different forms $hv_{S_\infty}$, with $h $ lying in a compact subset $G_0$ of $\GLtwo(K_\infty)$. After a computation one arrives at, with $B=G_0v_{S_\infty}$,
\begin{equation}\label{ffonelevint}
\begin{split}
   \sum_{\abs{F}\leq X^{\delta}}&\mu(F)\sum_{fg\mid F}\mu(g)\sum_{\mathcal{G}_{\text{nm}}^{S_\infty}(X\abs{f}^2/\abs{F}^4)}\eta_{fg}(x)\theta_P(x) \\&= \sum_{\abs{F}\leq X^{\delta}}\mu(F)\sum_{fg\mid F} \mu(g)\frac{1}{\vol(G_0)}\int_{h\in S_{X\abs{f}^2/\abs{F}^4}} \left(\sum_{x\in hB \cap V(R)}\eta_{fg}(x)\theta_P(x)\right)dh,
\end{split}
\end{equation}
cf. \cite[Eqs. (4.4), (5.3), (7.5)]{Ahlquist2}. Here, $S_X$ is essentially a fundamental domain for $\GLtwo(R)\setminus\GLtwo(K_\infty)_{X\abs{f}^2/\abs{F}^4}$, where the elements in $\GLtwo(K_\infty)_{\abs{f}^2/\abs{F}^4}$ have squared determinant equal to $X\abs{f}^2/\abs{F}^4$. Finally, the measure $dh$ is a Haar measure for $\GLtwo(K_\infty)$. We refer to \cite{Ahlquist2} for the details.

One should view the expression above as the analogue of the integral over Shintani zeta functions that we encountered over $\mathbb{Q}$. The analogue of shifting the contour is to apply the finite Fourier transform. For any modulus $\ell$, the space $V(R/\ell R)$ is self-dual, as $3\nmid q$, so that applying the inverse Fourier transform yields a dual sum over $y\in V(R/\ell R)$, see \cite[Section 5.1]{Ahlquist2}. We may lift the dual forms $y$ defined modulo $\ell$ to forms with coefficients in $R$. 

The contribution from the forms $y$ which are degenerate, i.e. for which $\mathrm{Disc}(y)=0$ corresponds roughly to the residue contributions at $s=1$ and $s=5/6$ that we saw in the computations over $\mathbb{Q}$. A calculation similar to those in \cite[Section 7]{Ahlquist2} shows that the contribution from such forms cancels against the quadratic contribution in \eqref{subquad}, up to a term $o(1)$, when $\sigma < 1$, see also Section \ref{degencontribsect}.

The contribution from forms $y$ which are nondegenerate are the most difficult to bound. We assume henceforth that $f=g=F=1$ to simplify the exposition. General values of $f$ and $g$ are handled similarly to how the function $\eta_{ab}$ was handled over $\mathbb{Q}$. Now, just as over $\mathbb{Q}$, the relevant part of the Fourier transform of $\theta_P$ is 
\begin{equation*}
    \abs{P}^{-1}\mathbf{1}_{(1^3),P}.
\end{equation*}
Using the unfolding method to handle the nondegenerate $y$, cf. the argument preceding \cite[Proposition 5.1]{Ahlquist2} shows that, up to constants, the contribution from nondegenerate $y$ to the integral above is a finite sum of terms
\begin{equation}\label{oscintff}
    \abs{P}^3\sum_\sumstack{y\in \GLtwo(R) \setminus V(R)\\ \mathrm{Disc}(y) \neq 0 \\ y \in (1^3)\nsmod P}\frac{\#\mathrm{Aut}(y)}{\abs{\mathrm{Disc}(y)}}\int_{\text{$v\sim y$ }} \chi_\infty([v^{(i)},v])\mathbf{1}_{v\in C}\mathbf{1}_{\abs{\mathrm{Disc}(v)}=\abs{\lambda}^4\abs{\mathrm{Disc}(y)}/\abs{P}^4}dv,
\end{equation}
where $v\sim y$ means that $v$ and $y$ lie in the same $\GLtwo(K_\infty)$-orbit, i.e. that their local behaviour at $P_\infty$ agrees, and where $\mathrm{Aut}(y)$ is the automorphism group of the corresponding local field. Furthermore, in the expression above, $C$ is a certain box and $\chi_\infty$ is the character on $K_\infty$ defined by
\begin{equation*}
    \chi_\infty\left(\sum_{\nu=-N}^\infty a_\nu T^{-\nu}\right) = \exp\left(-\frac{2\pi i}{p}\mathrm{Tr}_{\fq/\mathbb{F}_p}\left(a_{1}\right)\right),
\end{equation*}
where $p\mid q$ is the characteristic. The pairing $[x,y] := x_1y_1+x_2y_2/3+x_3y_3/3+x_4y_4$ is the usual $\GLtwo$-equivariant pairing. Finally, for $i\leq i_0\in \mathbb{Z}_{\geq 1}$, $v^{(i)}$ is an element in the same $\GLtwo(K_\infty)$-orbit as $v_{S_\infty}$, and $\abs{\lambda}^4\asymp X\abs{f}^2/\abs{F}^4$.

Now, the integral above only depends on the value of $\abs{\mathrm{Disc}(y)}$. Hence, we are interested in evaluating 
\begin{equation}\label{totramsumff}
    \sum_\sumstack{y\in \GLtwo(R) \setminus V(R)\\ \mathrm{Disc}(y) = Z \\ y \in (1^3)\nsmod P}1,
\end{equation}
for different $Z$. We study this by first switching to maximal rings using Lemma \ref{zetasubram}, which continues to hold over function fields. Similar to the situation over $\mathbb{Q}$ we need only consider the case of $S_3$-cubic fields, as the other maximal rings provide only a negligible contribution when $\sigma < 1$.

We also need a function field version of \cite[Theorem 11]{ShankarThorne}. This can be proven, with a slightly different approach depending on whether $\fq$ contains all third roots of unity or not, by using the methods of \cite{CohenMorra}. We leave out the details, but remark that the situation over function fields is much simpler, as there are no $3$-adic complications, as $3\nmid q$.

Just as over $\mathbb{Q}$, the zeta function associated with \eqref{totramsumff} contains one part which does not depend on $P$, and one part which oscillates with $P$. The oscillating part is negligible after summation over $P$, and the non-oscillating part can be computed to give a contribution equal to a constant multiplied by $Z/(\abs{P}^2\#\mathrm{Aut}(y))$. Hence, we obtain a main-term contribution to \eqref{oscintff}, which up to a constant equals
\begin{equation*}
        \abs{P}\int_{v\in C} \chi_\infty([v^{(i)},v])dv.
\end{equation*}
By orthogonality, this equals zero, as $v^{(i)}$ is not the zero form. This concludes the sketch of Theorem \ref{onelevfq}.

\section{Counting cubic number fields}
We now turn to the problem of counting cubic number fields. By applying bounds for character sums, similar to when we studied the one-level density, we are able to improve the bound for the error term in this counting function. For the sharp counting function, we are not able to modify Landau's method in a way that allows for obtaining cancellation between coefficients of different zeta functions. Instead the savings come from a better estimate of the number of binary cubic forms which are nonmaximal at large integers. When counting fields with a smooth weight function, we are able to conditionally establish a high degree of cancellation, obtaining an error term of quality $\mathcal{O}(X^{1/2+\epsilon})$.

\subsection{The sharp counting function}\label{sharpGRHsect}
  We first study the sharp counting function and prove Theorem \ref{sharpgrhthm}. Similar to \eqref{fieldsieve}, if we let $N^\pm(X)$ denote the number of cubic fields of absolute discriminant bounded by $X$ and of sign $\pm$ and let $Q^\pm(X)$ denote the number of quadratic fields with the same condition on the discriminant, then
\begin{equation}\label{countingSieve}
   N^\pm(X) +\frac{1}{2}Q^\pm(X)=  \sum_{n\geq 1}\mu(n)N^{\pm}(W_n;X) + \mathcal{O}(X^{1/2+\epsilon}),
\end{equation}
where $n$ is squarefree. The counting function $N^{\pm}(W_n;X)$ counts binary cubic forms which are nonmaximal at $n$, weighted by the inverse size of their stabilisers, and the error term bounds the contribution from $C_3$-fields. Next, we apply a discriminant-reducing sieve \cite[Eq. (70)]{BST}, allowing us to expand the above as
\begin{equation}\label{discredsum}
    \sum_{n\geq 1}\mu(n)\sum_{abc=n}\mu(b)N^\pm\big(\eta_{ab};X/(a^2b^4c^4)\big) = \sum_{a,b,c \text{ coprime}}\mu(ac)N^\pm\big(\eta_{ab};X/(a^2b^4c^4)\big).
\end{equation}
where $\eta_{ab}(x)$ counts the number of roots of the form $x$ modulo $ab$, as before. Now, $\eta_{ab}$ is nonnegative, however it turns out that it equals a sum of highly oscillatory functions whose contribution can be bounded using the GRH. 

For a prime $p$, we define
\begin{equation*}
    \eta'_p = \eta_p-1,
\end{equation*}
and extend multiplicatively to squarefree integers. Then,
\begin{equation*}
    \eta_{ab} = \sum_{d\mid ab}\eta'_d.
\end{equation*}
Hence, if we factor $a=a_1\ell_1$ and $b=b_1\ell_2$, \eqref{discredsum} further expands to
\begin{equation}\label{firstdecompsieve}
    \sum_{a_1,b_1,c,\ell_1,\ell_2 \text{ coprime}}\mu(a_1\ell_1c)N^\pm\big(\eta'_{a_1b_1};X/(a_1^2\ell_1^2b_1^4\ell_2^4c^4)\big).
\end{equation}
When studying this counting function, one often splits the summation into two different ranges $n\leq Q$ and $n> Q$. We refine this somewhat, which will allow us to take advantage of cancellation in the $a_1$-variable. More specifically, we separate into the "small" range $a_1b_1c\ell_2\ell_1^{2/3}\leq Q$ and the "large" range $a_1b_1c\ell_2\ell_1^{2/3}>Q$.

\subsubsection{Small modulus}

We first handle the small range. Here, the idea is to apply \cite[Theorem 3.1]{BTT} directly and not use the GRH. A minor problem for the application of this theorem is that $\eta'$ is not nonnegative, however we remedy this by using the relation
\begin{equation*}
    \eta'_{a_1b_1} = \sum_{d\mid a_1b_1}\mu(a_1b_1/d)\eta_{d}
\end{equation*}
so that the sum over the small range becomes
\begin{equation}\label{sharpsmallrangesum}
    \sum_\sumstack{a'_1,b'_1,c,\ell_1,\ell_1',\ell_2,\ell_2' \text{ coprime}\\a'_1b'_1c\ell_1^{2/3}\ell_1'\ell_2\ell_2' \leq Q }\mu(a_1'\ell_1'\ell_1c)\mu(\ell_1'\ell_2')N^\pm\big(\eta_{a'_1b_1'};X/(a_1'^2\ell_1^2\ell_1'^2b_1'^4\ell_2^4\ell_2'^4c^4)\big).
\end{equation}
Now, \cite[Theorem 3.1]{BTT} applies directly, cf. \cite[Eq. (79)]{BTT} and implies that (in the notation of \cite[Theorem 2.4]{BTT})
\begin{equation}\label{BTTapplication}
\begin{split}
&N^\pm\big(\eta_{a'_1b_1'};X/(a_1'^2\ell_1^2\ell_1'^2b_1'^4\ell_2^4\ell_2'^4c^4)\big) = \alpha^{\pm}\prod_{p\mid a_1'b_1'}\left(1+p^{-1}\right)\frac{X}{a_1'^2\ell_1^2\ell_1'^2b_1'^4\ell_2^4\ell_2'^4c^4} + \beta\left(\prod_{p\mid a_1'b_1'}2\right)\frac{X}{a_1'^2\ell_1^2\ell_1'^2b_1'^4\ell_2^4\ell_2'^4c^4}\\&+\gamma^\pm\prod_{p\mid a'_1b_1'}\left(1+p^{-1/3}\right)\frac{X^{5/6}}{(a_1'^2\ell_1^2\ell_1'^2b_1'^4\ell_2^4\ell_2'^4c^4)^{5/6}} +\mathcal{O}\left(\frac{X^{3/5+\epsilon}(a_1'b_1')^{2/5}}{(a_1'^2\ell_1^2\ell_1'^2b_1'^4\ell_2^4\ell_2'^4c^4)^{3/5}}\right).
\end{split}
\end{equation}
Summing the error term in the range $a_1b_1c\ell_2\ell_1^{2/3}\leq Q$ yields an upper bound
\begin{equation}\label{sharpbadsmallerror}
    \ll X^{3/5+\epsilon}Q^{1/5}.
\end{equation}
At a first glance, the tail sums of the various main and secondary terms above seem to be rather large. However, upon noting that
\begin{equation*}
    \sum_{d'\mid d}\mu(d/d')\prod_{p\mid d'}\left(1+p^{-1}\right) = \frac{1}{d},
\end{equation*}
and that
\begin{equation*}
    \sum_{d'\mid d}\mu(d/d')\prod_{p\mid d'}\left(1+p^{-1/3}\right) = \frac{1}{d^{1/3}},
\end{equation*}
we see that we may extend the sum over the first and third term in \eqref{BTTapplication} to infinity at the cost of an error (up to a factor $X^\epsilon$)
\begin{equation*}
    \ll \frac{X}{Q^{3/2}} + \frac{X^{5/6}}{Q},
\end{equation*}
which is $\ll XQ^{-2/3}$ when $Q\ll X^{1/3}$. For the second term of \eqref{BTTapplication}, extending the summation to infinity does not yield an acceptable error. Instead, we simply note that
\begin{equation*}
    \sum_{d'\mid d} \mu(d/d')\prod_{p\mid d'}2 = 1,
\end{equation*}
so that the part of \eqref{sharpsmallrangesum} summing this term becomes
\begin{equation}\label{redsmallsum}
    \beta X\sum_\sumstack{a_1,b_1,c,\ell_1,\ell_2 \text{ coprime}\\a_1b_1c\ell_1^{2/3}\ell_2 \leq Q }\mu(a_1\ell_1c)\frac{1}{a_1^2\ell_1^2b_1^4\ell_2^4c^4},
\end{equation}
where we have once again let $a_1=a_1'\ell_1'$ and $b_1=b_1'\ell_2'$.

\subsubsection{Large modulus}\label{sharplargesect}
We now turn to the range $a_1b_1c\ell_2\ell_1^{2/3}>Q$ where we will apply GRH to bound
\begin{equation}\label{tailsumsharppre}
    \sum_\sumstack{a_1,b_1,c,\ell_1,\ell_2 \text{ coprime}\\ a_1b_1c\ell_1^{2/3}\ell_2 > Q}\mu(a_1\ell_1c)N^\pm\big(\eta'_{a_1b_1};X/(a_1^2\ell_1^2b_1^4\ell_2^4c^4)\big)
\end{equation}
The main observation for being able to use GRH is that $\eta'_p = \lambda_p +p\mathbf{1}_{(0),p}$ so that we may decompose the above into
\begin{equation}\label{tailsumsharp}
    \sum_\sumstack{a_1',b_1',c,\ell_1,\ell_1',\ell_2, \ell_2'\text{ coprime}\\ a'_1b'_1c\ell_1^{2/3}\ell_1'\ell_2\ell_2' > Q}\ell_1'\ell_2'\mu(a_1'\ell_1'\ell_1c)N^\pm\big(\lambda_{a'_1b'_1};X/(a_1'^2\ell_1^2b_1'^4\ell_2^4c^4\ell_1'^6\ell_2'^6)\big).
\end{equation}
The contribution of irreducible cubic forms to this sum is
\begin{equation*}
    \sum_\sumstack{a_1',b_1',c,\ell_1,\ell_1',\ell_2, \ell_2'\text{ coprime}\\ a'_1b'_1c\ell_1^{2/3}\ell_1'\ell_2\ell_2' > Q}\ell_1'\ell_2'\mu(a_1'\ell_1'\ell_1c)\sum_\sumstack{0 < \pm \mathrm{Disc}(x)\leq X/(a_1'^2\ell_1^2b_1'^4\ell_2^4c^4\ell_1'^6\ell_2'^6)\\ \text{$x$ irreducible}}\lambda_{a_1'}(x)\lambda_{b_1'}(x).
\end{equation*}
Now, for the maximal cubic form $x_{\text{max}}$ associated with $x$, one sees that $\mu(a_1')\lambda_{a_1'}(x_{\text{max}})$ equals the $a_1'$th coefficient of $1/L(s,\rho_K)$, where $K$ is the field corresponding to the form $x_{\text{max}}$, see \eqref{artinprimedef}. Hence, a standard application of the GRH for cubic Dedekind zeta functions implies that for a maximal cubic form
\begin{equation}\label{grhoneoverL}
    \sum_{A<a_1'<2A} \mu(a_1')\lambda_{a_1'}(x_\text{max})\ll \abs{\mathrm{Disc}(x)}^\epsilon A^{1/2+\epsilon}. 
\end{equation}
As the splitting behaviour of $x$ and $x_\text{max}$ coincide, except for finitely many primes $\ll \abs{\mathrm{Disc}(x)}^\epsilon$ in number, we may interchange the order of summation and write the above as
\begin{equation}\label{sharplargesum}
     \sum_\sumstack{a_1',b_1',c,\ell_1,\ell_1',\ell_2, \ell_2'\text{ coprime}}\ell_1'\ell_2'\mu(\ell_1'\ell_1c)\sum_\sumstack{A=2^j\\ j\geq 0}\sum_\sumstack{0 < \pm \mathrm{Disc}(x)\leq X/(A^2\ell_1^2b_1'^4\ell_2^4c^4\ell_1'^6\ell_2'^6)\\ \text{$x$ irreducible}}\lambda_{b_1'}(x)\sum_\sumstack{a_1' > Q/b'_1c\ell_1^{2/3}\ell_1'\ell_2\ell_2'\\a_1'^2 \leq X/(\abs{\mathrm{Disc}(x)}\ell_1^2b_1'^4\ell_2^4c^4\ell_1'^6\ell_2'^6)\\ a_1'\in [A,2A]}\mu(a_1')\lambda_{a_1'}(x),
\end{equation}
which under the GRH is
\begin{equation*}
    \ll \frac{X^{1+\epsilon}}{Q^{3/2}}.
\end{equation*}

Next, we consider the reducible contribution to \eqref{tailsumsharppre}. Recall that we only need to consider the contribution from forms with precisely one root. For studying such forms, we make the decomposition $\eta'_p=\Tilde{\lambda}_p+(p-1)\mathbf{1}_{p,(0)}$. By means of the truncated Perron's formula, we then see that the contribution from such reducible forms is up to a negligible error
\begin{equation*}
    \sum_\sumstack{a_1',b_1',c,\ell_1,\ell_1',\ell_2, \ell_2'\text{ coprime}\\ a'_1b'_1c\ell_1^{2/3}\ell_1'\ell_2\ell_2' > Q}\ell_1'\ell_2'\prod_{p\mid \ell_1'\ell_2'}\left(1-\frac{1}{p}\right)\mu(a_1'\ell_1'\ell_1c)\frac{1}{2\pi i}\int_{1-iT}^{1+iT}\left(\frac{X}{a_1'^2\ell_1^2b_1'^4\ell_2^4c^4\ell_1'^6\ell_2'^6}\right)^s\frac{1}{s}\xi^{\pm}_{\text{quad}}(s,\Tilde{\lambda}_{a_1'b_1'})ds,
\end{equation*}
with $T=X^{10}$, say. Here, 
\begin{equation*}
    \xi^{\pm}_{\text{quad}}(s,\Tilde{\lambda}_{a_1'b_1'}) = \sum_\sumstack{x\in \mathrm{GL}_2(\mathbb{Z})\setminus V(\mathbb{Z})\\x \text{ has precisely one root}} \frac{\Tilde{\lambda}_{a_1'b_1'}(x)}{\# \mathrm{Stab}(x)\abs{\mathrm{Disc}(x)}^s},
\end{equation*}
where the summation ranges over $x$ whose discriminant has sign equal to $\pm$. We remark that this zeta function was essentially studied in Section \ref{redringssect}. As we are now studying the range where the variable $n$ from Section \ref{redringssect} is large, we need to use a different approach. 

First, we write $\lambda'_p = \Tilde{\lambda}_p-1$ to create a character which is oscillatory for reducible $x$ whose corresponding fraction field is a quadratic field times $\mathbb{Q}$. Henceforth, we call such $x$ quadratic. The above becomes
\begin{equation}\label{redlastsum}
    \sum_\sumstack{a_1'',b_1'',c,\ell_1,\ell_1',\ell_1'',\ell_2, \ell_2',\ell_2''\text{ coprime}\\ a''_1b''_1c\ell_1^{2/3}\ell_1'\ell_1''\ell_2\ell_2'\ell_2'' > Q}\ell_1'\ell_2'\prod_{p\mid \ell_1'\ell_2'}\left(1-\frac{1}{p}\right)\mu(a_1''\ell_1'\ell_1''\ell_1c)\frac{1}{2\pi i}\int_{1-iT}^{1+iT}\left(\frac{X}{a_1''^2\ell_1^2\ell_1''^2b_1''^4\ell_2^4\ell_2''^4c^4\ell_1'^6\ell_2'^6}\right)^s\frac{1}{s}\xi^{\pm}_{\text{quad}}(s,\lambda'_{a_1''b_1''})ds.
\end{equation}
We remark that the zeta function above is essentially a twisted version of a zeta function studied by Shintani, see \cite{ShintQuad}.
The idea is to evaluate this integral by shifting the contour to $\Re(s)=1/2+\epsilon$. For this, we need the following growth bound for $\xi^{\pm}_{\text{quad}}(s,\lambda'_{a_1''b_1''})$.
\begin{lemma}\label{redshintlemma}
The zeta function $\xi^{\pm}_{\text{quad}}(s,\lambda'_{a_1''b_1''})$ admits an analytic continuation to $\Re(s) > 1/2$, with its only pole at $s=1$ with residue
    \begin{equation}\label{redres}
        \frac{\pi^2}{24} \cdot \frac{\mu(a_1''b_1'')}{(a_1''b_1'')^3}\prod_{p\mid a_1''b_1''}\left(1-\frac{1}{p}\right) = \beta  \frac{\mu(a_1''b_1'')}{(a_1''b_1'')^3}\prod_{p\mid a_1''b_1''}\left(1-\frac{1}{p}\right).
    \end{equation}
Furthermore, at the line $\Re(s)=1/2+\epsilon$, it satisfies a weak Lindelöf-type bound
    \begin{equation}\label{redLindelöf}
        \xi^{\pm}_{\text{quad}}(1/2+\epsilon+it,\lambda'_{a_1''b_1''}) \ll (1+\abs{t}^\epsilon)\abs{a_1''b_1''}^\epsilon,
    \end{equation}
    under the GRH for Dirichlet $L$-functions. Furthermore, the same type of bound holds on the ray $\Im(s)=T$, $\Re(s) \geq 1/2+\epsilon$.
\end{lemma}
We postpone the proof for the sake of making the main argument clearer.

By applying the above lemma, we may shift the contour in the integral from \eqref{redlastsum} to the line $\Re(s)=1/2+\epsilon$. The horizontal contributions are negligible, as $T=X^{10}$ and the contribution from the shifted integral is
\begin{equation*}
    \ll X^{1/2+\epsilon}\sum_\sumstack{a_1'',b_1'',c,\ell_1,\ell_1',\ell_1'',\ell_2, \ell_2',\ell_2''\text{ coprime}\\ a''_1b''_1c\ell_1^{2/3}\ell_1'\ell_1''\ell_2\ell_2'\ell_2'' > Q}\frac{\ell_1'\ell_2'}{(a_1''\ell_1\ell_1''b_1''^2\ell_2^2\ell_2''^2c^2\ell_1'^3\ell_2'^3)^{1+2\epsilon}}\ll X^{1/2+\epsilon}.
\end{equation*}

The contribution from the residue at $s=1$ equals
\begin{equation*}
    \beta X\sum_\sumstack{a_1'',b_1'',c,\ell_1,\ell_1',\ell_1'',\ell_2, \ell_2',\ell_2''\text{ coprime}\\ a''_1b''_1c\ell_1^{2/3}\ell_1'\ell_1''\ell_2\ell_2'\ell_2'' > Q}\mu(a_1''\ell_1'\ell_1''\ell_1c)\frac{\mu(a_1''b_1'')}{a_1''^5\ell_1^2\ell_1''^2b_1''^7\ell_2^4\ell_2''^4c^4\ell_1'^5\ell_2'^5}\prod_{p\mid a_1''b_1''\ell_1'\ell_2'}\left(1-\frac{1}{p}\right).
\end{equation*}
Writing $a_1=a_1''\ell_1''\ell_1'$, $b_1=b_1''\ell_2''\ell_2'$ and using multiplicativity, the above becomes
\begin{equation*}
     \beta X\sum_\sumstack{a_1,b_1,c,\ell_1,\ell_2 \text{ coprime}\\a_1b_1c\ell_1^{2/3}\ell_2 > Q }\mu(a_1\ell_1c)\frac{1}{a_1^2\ell_1^2b_1^4\ell_2^4c^4},
\end{equation*}
which is precisely the complementary sum of \eqref{redsmallsum}. Together, these two sums combine into the main term for the counting function of quadratic fields.

We have now found the main and secondary terms in the counting function of cubic fields, together with an error (up to a factor $X^\epsilon$)
\begin{equation*}
    \ll X^{1/2} + \frac{X}{Q^{3/2}}+\frac{X^{5/6}}{Q}+X^{3/5}Q^{1/5}.
\end{equation*}
Letting $Q=X^{4/17}$ gives a total error
\begin{equation*}
    X^{11/17+\epsilon}=X^{2/3-1/51+\epsilon},
\end{equation*}
as claimed in Theorem \ref{sharpgrhthm}.

Finally, to conclude the argument, we prove Lemma \ref{redshintlemma}.
\begin{proof}[Proof of Lemma \ref{redshintlemma}]
    Without loss of generality, we may modify the definition of $\xi^{\pm}_{\text{quad}}(s,\lambda'_{a_1''b_1''})$ to sum over $\SLtwo$-orbits instead of $\GLtwo$-orbits. The zeta function defined using such orbits is simply equal to $2$ times the zeta function defined using $\GLtwo$-orbits.

    As remarked in connection to \eqref{redshape}, a reducible form $x$ can be brought into the shape
\begin{equation}
    u(x_1u^2+x_2uv+x_3v^2),
\end{equation}
with $0 \leq x_2<2x_3$ and $x_3 >  0$. Moreover, if $x$ only has one root, then this representation is unique and its $\SLtwo$-stabiliser is trivial. Note also that $x$ having more than $1$ root is equivalent to $x_2^2-4x_1x_3$, i.e. the discriminant of a corresponding quadratic form, being square.

The discriminant of a reducible form $x$ of the above shape equals
\begin{equation*}
    x_3^2(x_2^2-4x_1x_3).
\end{equation*}
Moreover, for a prime $p$,
\begin{equation*}
    \lambda'_p(x)=\begin{cases}
        0, \text{ if $p\mid x_1,x_2,x_3$},\\ 
        -1, \text{ if $p\mid x_2,x_3$, $p\nmid x_1$},\\
        0, \text{ if $p\mid x_3$, $p\nmid x_2$},        
    \end{cases}
\end{equation*}
and if $p\nmid x_3$ it equals the Kronecker character
\begin{equation*}
    \left(\frac{x_2^2-4x_1x_3}{p}\right).
\end{equation*}

Hence, from the definition, it follows that
\begin{equation*}
    \xi^{\pm}_{\text{quad}}(s,\lambda'_{d}) = \frac{1}{2}\sum_{AB= d}\frac{\mu(A)}{A^{3s}}\sum_\sumstack{x_3=1\\ (x_3,B)=1}^\infty \frac{1}{x_3^{2s}}\sum_\sumstack{m=1\\ mA\neq \square }^{\infty}\frac{\left(\frac{mA}{B}\right)}{m^s}\sum_\sumstack{Ax_2^2-4x_1x_3=\pm m\\ 0\leq x_2< 2x_3\\(A,x_1)=1}1.
\end{equation*}
We may remove the condition that $x_1$ be coprime to $A$ using Möbius inversion so that the above becomes
\begin{equation*}
    \xi^{\pm}_{\text{quad}}(s,\lambda'_{d}) = \frac{1}{2}\sum_{AB= d}\frac{\mu(A)}{A^{3s}}\sum_{C\mid A}\frac{\mu(C)}{C^s}\sum_\sumstack{x_3=1\\ (x_3,B)=1}^\infty \frac{1}{x_3^{2s}}\sum_\sumstack{m=1\\ mAC\neq \square }^{\infty}\frac{\left(\frac{mAC}{B}\right)}{m^s}\sum_\sumstack{(A/C)x_2^2-4x_1x_3=\pm m\\ 0\leq x_2< 2x_3}1.
\end{equation*}

For the innermost sum, we note that once $x_2$ has been chosen, one may solve for 
\begin{equation*}
    x_1 = \frac{(A/C)x_2^2- (\pm m)}{4x_3}.
\end{equation*}
Hence, the sum over $x_1,x_2$ simply counts half the number of solutions to $(A/C)y^2\equiv \pm m \nsmod{4x_3}$. Before proceeding, we note that up to a negligible error, we may remove the condition that $mAC$ be nonsquare. Next, we expand the zeta function as
\begin{equation}\label{inclsqsum}
\begin{split}
     \frac{1}{2}&\sum_{AB= d}\frac{\mu(A)}{A^{3s}}\sum_{C\mid A}\frac{\mu(C)}{C^s}\\&\times \sum_{x_3=1}^\infty \frac{1}{x_3^{2s}}\left(\sum_\sumstack{D\mid 4x_3\\ D < 4x_3}\frac{1}{D^s}\sum_\sumstack{m=1 \\(4x_3/D,m)=1}^{\infty}\frac{\left(\frac{mACD}{B}\right)}{m^s}\sum_\sumstack{(A/C)x_2^2-4x_1x_3=\pm Dm\\ 0\leq x_2< 2x_3}1+\frac{1}{(4x_3)^s}\sum_\sumstack{m=1}^{\infty}\frac{\left(\frac{4x_3mAC}{B}\right)}{m^s}\sum_\sumstack{(A/C)x_2^2-4x_1x_3=\pm 4x_3m\\ 0\leq x_2< 2x_3}1\right).
\end{split}
\end{equation}
Now, for fixed $A,C,x_3$ as above, write $E = (A/C,4x_3)$ and define $F$ by $A/C=EF$. Then,
\begin{equation*}
    \sum_\sumstack{(A/C) x_2^2-4x_1x_3=\pm Dm\\ 0\leq x_2< 2x_3}1 = \frac{1}{2}\#\{x_2 \nsmod{4x_3/(E,D): (A/C)x_2^2/(E,D)\equiv \pm mD/(E,D)}\}.
\end{equation*}
If $(E,D)\neq E$, then the above is zero, as the left-hand side in the equivalence reduces to zero modulo $E/(E,D)$, which is squarefree. In particular, only the $D$ satisfying $E\mid D$ contribute. For $D$ divisible by $E$, we have that
\begin{equation*}
    \frac{1}{2}\#\{x_2 \nsmod{4x_3/E: Fx_2^2\equiv \pm mD/E}\} = \frac{G_D}{2}\#\{x_2 \nsmod{4x_3/D}: F(G_D^2E/D)x_2^2\equiv \pm m\}
\end{equation*}
where $G_D$ is the smallest integer such that $D/E\mid G_D^2$. If we denote by $q_e$ the size of the quadratic character group $\mathbb{Z}/e\mathbb{Z}[2]$ of $\mathbb{Z}/e\mathbb{Z}$, then $q_e$ is multiplicative in $e$ and
\begin{equation*}
    \frac{G_D}{2}\#\{x_2 \nsmod{4x_3/D}: F(G_D^2E/D)x_2^2\equiv \pm m\}  = \frac{G_D}{2}\sum_{\chi\in \mathbb{Z}/(4x_3/D)\mathbb{Z}[2]} \chi(\pm m)\overline{\chi}(FG_D^2E/D).
\end{equation*}
We also have the uniform bound $q_p\ll 1$ so that $q_e\ll e^\epsilon$. Using quadratic reciprocity \eqref{quadrec}, we see that the sum over $m$ in \eqref{inclsqsum} is a sum of Dirichlet $L$-functions $L(s,\Tilde{\chi})$, for quadratic characters $\Tilde{\chi}$. Moreover, $\Tilde{\chi}$ is the trivial character if and only if $d=A$ and $\chi\in \mathbb{Z}/e\mathbb{Z}[2]$ is trivial. This finishes the proof of \eqref{redLindelöf}, using Lindelöf for Dirichlet $L$-functions.

To find the value of the residue, we begin by computing the residue for the function defined by the $D$ and $m$-sum in \eqref{inclsqsum}, when $B=1$. The value of this residue is simply half of the average
\begin{equation*}
    \frac{1}{4x_3}\sum_{m=1}^{4x_3}\#\{(A/C)y^2\equiv \pm m \nsmod{4x_3}\} = 1.
\end{equation*}
Hence, the total residue equals \eqref{redres}.

\end{proof}

\subsection{Counting sharply without the GRH}\label{sharpnoGRHsect}
We now outline how to unconditionally obtain savings in the error term of the counting function of cubic fields, proving Theorem \ref{uncondsharpctthm}. The purpose of this section is to illustrate  that it is possible to go beyond an error of size $\mathcal{O}\left(X^{2/3+\epsilon}\right)$, even unconditionally. Our argument is not optimised and can certainly be improved at the cost of reduced clarity. We present it in the form given below so that the most important steps are highlighted.

The starting point of the argument is the sieve \eqref{countingSieve}, and its decomposed form \eqref{firstdecompsieve}. As in the previous section, we split the summation into two parts, depending on a parameter $Q$. The error-term contribution from the small range is once again bounded by
\begin{equation*}
    \ll X^{3/5+\epsilon}Q^{1/5},
\end{equation*}
as in \eqref{sharpbadsmallerror}.

Similar to Section \ref{sharpGRHsect}, the source of the savings in the error term is the large range. The key idea is to obtain a nontrivial bound for the sum \eqref{grhoneoverL}. As we shall see, without access to the GRH, this nontrivial bound is very close to the trivial one.

If $K$ is the field associated with $x_\text{max}$ from \eqref{grhoneoverL}, then the sum in \eqref{grhoneoverL} is a partial sum of the coefficients of $1/L(s,\rho_K)$. We obtain a power-saving estimate by using the fact that not too many $L(s,\rho_K)$ can have a zero which lies close to the point $s=1$. 

To be more precise, given a field $K$, and real numbers $R,T\geq 1$, and $\sigma\geq 0$ we let 
\begin{equation}\label{zeroboxdef}
    N_K(\sigma,T)=\#\{\rho: L(s,\rho_K)=0, \abs{\Im(\rho)}\leq T, \Re(\rho)\geq \sigma  \}.
\end{equation}
We then have the zero-density result, which follows from \cite[Theorem 1.1]{TZ}.
\begin{prop}\label{zerodenspropcube}
    Given $R,T$ and $\sigma$ as above, we have that
    \begin{equation}\label{cubiczerodens}
        \sum_\sumstack{K: \abs{\mathrm{Disc}(K)}\leq R\\ K \text{ $S_3$-cubic}} N_K(\sigma,T)\ll_\epsilon R^{1/2+\epsilon}(R^3T)^{B(1-\sigma)+\epsilon},
    \end{equation}
    where $B=216\cdot 10^7$.
\end{prop}
We postpone the proof until the end of this section. 

Given the proposition above, we may fix $T=R$, and choose $\sigma$ such that
\begin{equation}\label{sigmachoice}
    B(1-\sigma)=1/16,
\end{equation}
say. Then, the right-hand side in \eqref{cubiczerodens} is $\ll R^{3/4+\epsilon}$. In particular, at most $R^{3/4+\epsilon}$ of the $L$-functions $L(s,\rho_K)$, with $K$ of absolute discriminant bounded by $R$ can have a zero in the box defined by \eqref{zeroboxdef}. In particular, most $L(s,\rho_K)$ have no zero in this box.

Now, given a field $K$, there are at most $\left(R/\abs{\mathrm{Disc}(K)}\right)^{1/2+\epsilon}$ irreducible cubic forms of discriminant bounded by $R$ with $K$ as the fraction field of the corresponding overring. Hence, the number of irreducible cubic forms $x$, whose maximal overring has fraction field $K$ with $L(s,\rho_K)$ having a zero in the above box is bounded by
\begin{equation*}
    \ll \sum_{\abs{\mathrm{Disc}(K)}\leq R^{3/4}}\frac{R^{1/2+\epsilon}}{\abs{\mathrm{Disc}(K)}^{1/2}} + R^{3/4+\epsilon}\frac{R^{1/2+\epsilon}}{\left(R^{3/4}\right)^{1/2}} \ll R^{7/8+\epsilon}.
\end{equation*}

On the other hand, for an $L$-function $L(s,\rho_K)$ with no zero in the box above, we have the following nontrivial bound for its partial sums.
\begin{lemma}\label{uncondpsavelemma}
    Let $K$ be an $S_3$-cubic field such that $L(s,\rho_K)$ has no zeros in the box
    \begin{equation}\label{actualboxdef}
        \{s:  \abs{\Im(s)}\leq T, \Re(s)\geq \sigma  \},
    \end{equation}
    where $1/2 <\sigma < 1$. Then, for $A$ with $\log T \asymp \log A$, we have that
    \begin{equation}\label{savesum}
        \sum_{a\leq A}\mu_K(a)\ll \left(A^{\sigma+\epsilon}+\frac{A^{1+\epsilon}}{T}\right)\abs{\mathrm{Disc}(K)}^{\epsilon},
    \end{equation}
    where $\mu_K(a)$ is the $a$th coefficient of $1/L(s,\rho_K)$.
\end{lemma}
\begin{proof}
    By the truncated Perron's formula, we have that the partial sum above equals
    \begin{equation*}
        \frac{1}{2\pi i}\int_{1+\epsilon-iT/2}^{1+\epsilon+iT/2} \frac{A^s}{L(s,\rho_K)}\frac{ds}{s} + \mathcal{O}\left(\frac{A^{1+\epsilon}}{T}\right).
    \end{equation*}
Now, in the box 
\begin{equation*}
    \{s:  \abs{\Im(s)}\leq T/2, \Re(s)\geq \sigma+\epsilon  \},
\end{equation*}
we have the bound
\begin{equation}\label{growthbd}
    \frac{1}{\abs{L(s,\rho_K)}}\ll \abs{\mathrm{Disc}(K)}^\epsilon(\abs{\Im(s)}+1)^\epsilon,
\end{equation}
as $L(s,\rho_K)$ has no zeros in the box \eqref{actualboxdef}. Indeed, this follows from similar methods as those used to prove \cite[Proposition 5.16, Theorem 5.19]{IK}. Hence, shifting the contour to $\Re(s)=\sigma+\epsilon$ yields the desired bound.
\end{proof}
\begin{remark}
    We have the same bound, with a very similar proof, if one restricts the sum in \eqref{savesum} to squarefree integers $a$. In this case, $\mu_K(a)=\mu(a)\lambda_K(a)$.
\end{remark}

Using Lemma \ref{uncondpsavelemma}, we can bound the sum \eqref{sharplargesum} as being 
\begin{equation*}
    \ll \frac{X^{1+\epsilon}}{Q^{2-\sigma}} = \frac{X}{Q^{1+1/(16B)}},
\end{equation*}
with $\sigma$ as in \eqref{sigmachoice}. Indeed, to prove this, we consider a box as in \eqref{actualboxdef} and use Lemma \ref{uncondpsavelemma} to control partial sums associated with $L(s,\rho_K)$ with no zeros in such a box, and use the trivial bound
\begin{equation*}
    \sum_{a\leq A}\mu(a)\lambda_K(a)\ll A^{1+\epsilon}
\end{equation*}
for the few $K$ such that $L(s,\rho_K)$ has a zero in such a box.

We also need to consider the contribution to the large range coming from reducible forms. We proceed similarly to the proof of Lemma \ref{redshintlemma}, except that we use the convexity bound for Dirichlet $L$-functions instead of the GRH. We obtain a main term and an error term, which is eclipsed by the one above.

Finally, balancing the error terms
\begin{equation*}
    X^{3/5+\epsilon}Q^{1/5}+\frac{X^{1+\epsilon}}{Q^{1+1/(16B)}}
\end{equation*}
yields the optimal choice $Q= X^{a}$, with 
\begin{equation}\label{sharpnogrhopta}
    a= \frac{2}{5}\left(\frac{6}{5}+\frac{1}{16B}\right)^{-1},
\end{equation}
giving a total error
\begin{equation*}
    \ll X^{2/3-c_B+\epsilon},
\end{equation*}
with
\begin{equation*}
    c_B=\frac{1}{288B+15},
\end{equation*}
as desired.

We remark that Theorem \ref{quasiapplthm} is proven in the same way. Assuming that $\zeta_K(s)$, and thus $L(s,\rho_K)$, has no zeros with real part larger than $1-\tau$, Lemma \ref{uncondpsavelemma} implies a bound $\ll A^{1-\tau+\epsilon}$ for the partial sums \eqref{savesum}. Balancing the resulting error terms 
\begin{equation}\label{quasibalance}
    X^{3/5+\epsilon}Q^{1/5}+\frac{X^{1+\epsilon}}{Q^{1+\tau}}
\end{equation}
yields the exponent $c(\tau)$ from Theorem \ref{quasiapplthm}.

We end this section by proving Proposition \ref{zerodenspropcube}.
\begin{proof}[Proof of Proposition \ref{zerodenspropcube}]
For the family $\mathcal{G}$ of Galois $S_3$-sextic fields $M$, let $\mathcal{G}(R) = \{M\in \mathcal{G}: \abs{\mathrm{Disc}(M)}\leq R\}$. Then, it is known that
\begin{equation*}
    \sum_{M\in \mathcal{G}(R)}N_M(\sigma,T) \ll m_\mathcal{G}(R) (RT)^{10^76^3(1-\sigma)+\epsilon},
\end{equation*}
see \cite[Theorem 1.1]{TZ}. Here, 
\begin{equation*}
    m_\mathcal{G}(R) = \max_{M\in \mathcal{G}(R)}\#\{M'\in \mathcal{G}(R): M\cap M'\neq \mathbb{Q}\},
\end{equation*}
and $N_M(\sigma,T)$ counts zeros of $\zeta_M/\zeta$. Moreover, \cite[Theorem 1.1]{TZ} asserts that the same inequality holds for every subfamily of $\mathcal{G}$ as well.

Now, $\zeta_M/\zeta$ factors into a product of Artin $L$-functions associated with characters of the nontrivial irreducible representations of $S_3$. One of these factors is $L(s,\rho_{K_M})$, where $K_M$ is the unique (up to isomorphism) $S_3$-cubic subfield of $M$ so that $M$ is the Galois closure of $K_M$. It follows that every zero of $L(s,\rho_{K_M})$ is also a zero of $\zeta_M(s)/\zeta(s)$.

Let $L_M$ denote the quadratic resolvent of $M$. Then, the discriminant 
\begin{equation*}
    \abs{\mathrm{Disc}(M)} = \abs{\mathrm{Disc}(L_M)}\abs{\mathrm{Disc}(K_M)}^2.
\end{equation*}
The idea is now to split $\mathcal{G}$ into subfamilies depending on the size of the discriminant of $L_M$. Fix a small $\delta > 0$ and consider a partition $I_0 = [0,\delta)$, $I_1=[\delta, 2\delta),...,$ $I_{N_\delta}=[N_{\delta}\delta, 1/3]$ of the interval $[0,1/3]$. For each $0\leq i\leq N_\delta$, we define a subfamily $\mathcal{G}_i$ of $\mathcal{G}$ by
\begin{equation*}
    \mathcal{G}_i=\big\{M\in \mathcal{G}: \frac{\log\abs{\mathrm{Disc}(L_M)}}{\log\abs{\mathrm{Disc}(M)} } \in I_i\big\}.
\end{equation*}
Then, for $M\in \mathcal{G}_i$, we have that
\begin{equation*}
    \abs{\mathrm{Disc}(M)}^{(1-(i+1)\delta)/2} \leq \abs{\mathrm{Disc}(K_M)}\leq \abs{\mathrm{Disc}(M)}^{(1-i\delta)/2}.
\end{equation*}
In particular, for such an $M$, if $\abs{\mathrm{Disc}(K_M)}\leq R$, then $\abs{\mathrm{Disc}(M)}\leq R^{2/(1-(i+1)\delta)}$. We write $a_i= 2/(1-(i+1)\delta)$.

We now bound $m_{\mathcal{G}_i}(R^{a_i})$. For two non-isomorphic $S_3$-sextic $M,M'$, the condition $M\cap M'\neq \mathbb{Q}$ translates into $M$ and $M'$ being cubic extensions of the same quadratic subfield. Equivalently, the cubic subfields $K_M$ and $K_{M'}$ share the same quadratic resolvent $L_M$. Write $\abs{\mathrm{Disc}(K_{M'})} = \abs{\mathrm{Disc}(L_{M})}f(K_{M'})^2$, for a positive integer $f(K_{M'})$, which up to bounded factors equals the product of all totally ramified primes in $K_{M'}$. For $M'\in \mathcal{G}_i$, we note that
\begin{equation*}
    f(K_{M'})^2\leq \abs{\mathrm{Disc}(M')}^{(1-3i\delta)/2}.
\end{equation*}
Let $d$ be a fundamental discriminant and let $f\geq 1$. Then, by the Brauer--Siegel theorem, the number of cubic fields of discriminant equal to $df^2$ is $\ll d^{1/2+\epsilon}f^{\epsilon}$. It follows that
\begin{equation*}
    m_{\mathcal{G}_i}(R^{a_i})\leq R^{a_i\left((i+1)\delta/2+(1-3i\delta)/4\right) +\epsilon}= R^{\frac{1}{2}+o_\delta(1)+\epsilon},
\end{equation*}
where $o_\delta(1)\to 0$ as $\delta \to 0$.

Finally, we bound the left-hand side of \eqref{cubiczerodens} asymptotically  by
\begin{equation*}
    \sum_{i=0}^{N_\delta} \sum_{M\in \mathcal{G}_i(R^{a_i})}N_M(\sigma,T) \ll_\delta \sum_{i=0}^{N_\delta} m_{\mathcal{G}_i}(R^{a_i})(R^3T)^{10^76^3(1-\sigma)+\epsilon}\ll_\delta R^{1/2+o_\delta(1)+\epsilon}(R^3T)^{10^76^3(1-\sigma)+\epsilon}.
\end{equation*}
Choosing $\delta >0$ small enough yields the desired bound.
\end{proof}

\subsection{The smoothed counting function}\label{smoothctallfieldsect}

We now turn to the related problem of counting cubic number fields with a smooth weight, where we will be able to obtain significantly improved results. Let $h: \mathbb{R}_{\geq 0}\to \mathbb{C}$ be a smooth function with compact support. We then study the smoothed counting function
\begin{equation*}
    B^\pm(h;X) = \sum_{m\geq 1}L^{\pm}_mh\left(\frac{m}{X}\right),
\end{equation*}
where $L^\pm_m$ is the number of $S_3$-cubic, and quadratic fields, with discriminant equal to $\pm m$, with the quadratic fields weighted by a factor $1/2$. The contribution from quadratic fields can later be removed by means of summation by parts to recover $N^\pm(h;X)$.

An inclusion-exclusion sieve shows that 
\begin{equation}\label{smoothinclexcl}
    B^\pm(h;X)  = \sum_{n\geq 1}\mu(n)\sum_{m\geq 1}a^{\pm}_1(\Psi_{n},m)h\left(\frac{m}{X}\right) + \mathcal{O}(X^{1/2+\epsilon}),
\end{equation}
where $a^{\pm}_1(\Psi_{n},m)$ counts the number of cubic forms which are nonmaximal at $n$ with discriminant equal to $\pm m$. We remark that all error terms are allowed to depend on $h$.

We write the innermost sum above as a contour integral without first separating $n$ into large and small. If we let $\Tilde{h}$ denote the Mellin transform of $h$, then by Mellin inversion
\begin{equation}\label{smoothfirstcont}
    \sum_{m\geq 1}a^{\pm}_1(\Psi_{n},m)h\left(\frac{m}{X}\right)= \frac{1}{2\pi i}\int_{(1+\epsilon)}X^s\xi^\pm(s,\Psi_n)\Tilde{h}(s)ds,
\end{equation}
see e.g. \cite[Section 4.1]{SST}. The above integral is absolutely convergent. Indeed, as
\begin{equation*}
    \Tilde{h}(s)= \int_{0}^\infty h(u)u^{s-1}du,
\end{equation*}
integration by parts shows that $\Tilde{h}(s)$ is rapidly decaying as $\abs{\Im(s)}\to \infty$ along a fixed vertical strip. 

Applying the discriminant-reducing identity \eqref{Shintanidiscred} shows that \eqref{smoothfirstcont} equals
\begin{equation*}
    \frac{1}{2\pi i}\sum_{abc=n}\mu(b)\int_{(1+\epsilon)}\left(\frac{X}{a^2b^4c^4}\right)^s\xi^\pm(s,\eta_{ab})\Tilde{h}(s)ds.
\end{equation*}
We now use the expansion
\begin{equation*}
    \eta_{ab}=\sum_{d\mid ab}\eta'_{d},  
\end{equation*}
and write $a=a_1\ell_1$, $b=b_1\ell_2$ as in the previous section to expand the above into
\begin{equation*}
    \frac{1}{2\pi i}\sum_{a_1b_1\ell_1\ell_2c=n}\mu(b_1\ell_2)\int_{(1+\epsilon)}\left(\frac{X}{a_1^2b_1^4\ell_1^2\ell_2^4c^4}\right)^s\xi^\pm(s,\eta'_{a_1b_1})\Tilde{h}(s)ds.
\end{equation*}
If we include the sum over $n$, this becomes
\begin{equation*}
\frac{1}{2\pi i}\sum_{a_1,b_1,\ell_1,\ell_2,c \text{ coprime}}\mu(a_1\ell_1c)\int_{(1+\epsilon)}\left(\frac{X}{a_1^2b_1^4\ell_1^2\ell_2^4c^4}\right)^s\xi^\pm(\eta'_{a_1b_1},s)\Tilde{h}(s)ds.    
\end{equation*}
As before, we separate into a large and a small range. We make the small range slightly smaller when compared to Section \ref{sharpGRHsect}, which will be convenient when estimating certain dual sums. To be precise, we define the two ranges by $a_1b_1^{2}\ell_1^{2/3}\ell_2c\leq Q$ and $a_1b_1^2\ell_1^{2/3}\ell_2c > Q$. Note that in the large range, $b_1$ can now be significantly smaller than before. However, the rapid decay of $1/b_1^4$, compared to $1/a_1^2$, ensures that this does not affect the error terms.
\subsubsection{The smooth large range}
The range where $a_1b_1^2\ell_1^{2/3}\ell_2c > Q$ is handled very similarly to the corresponding range in the estimates for the sharp counting function whence we provide only a sketch.

First, we separate the contribution of the various types of rings to $\xi^\pm(s,\eta'_{a_1b_1})$. The quadratic contribution is handled as in Section \ref{sharplargesect}, using Lemma \ref{redshintlemma} to shift the contour to $\Re(s)=1/2+\epsilon$. The contribution from forms with three roots is $\ll X^{1/2+\epsilon}$, as in Section \ref{redringssect}. Finally, the irreducible contribution is handled by keeping the contour at $(1+\epsilon)$ and using absolute convergence to expand the irreducible part of $\xi^\pm(s,\eta'_{a_1b_1})$. Interchanging the order of summation and using the GRH to control the sum over $a_1$ yields a total bound $\ll X^{1+\epsilon}Q^{-3/2}$ for the irreducible contribution to the large range. 

In total, we extract from the large range the expected main term coming from quadratic rings, as well as a total error
\begin{equation}\label{largesmootherror}
    \ll X^{1/2+\epsilon}+\frac{X^{1+\epsilon}}{Q^{3/2}},
\end{equation}
just as in the sharp case.

\subsubsection{The smooth small range}\label{smoothsmallallfldctsect}
We turn our attention to the small range. In the sharp case, we isolated the main and secondary terms with an error $\ll X^{3/5+\epsilon}Q^{1/5}$. In the smooth case, we will see that we can significantly improve this by applying the GRH in this range as well.

For convenience, we replace $\xi^\pm(s,\eta'_{a_1b_1})$ by the diagonalised zeta functions
\begin{equation*}
    \xi_\pm(s,\eta'_{a_1b_1}) = \sqrt{3}\xi^+(s,\eta'_{a_1b_1})\pm \xi^{-}(s,\eta'_{a_1b_1}),
\end{equation*}
so that we may apply the functional equation \eqref{genfcnal} without difficulties. We then consider
\begin{equation*}
\frac{1}{2\pi i}\sum_\sumstack{a_1,b_1,\ell_1,\ell_2,c \text{ coprime}\\ a_1b_1^2\ell_1^{2/3}\ell_2c\leq Q}\mu(a_1\ell_1c)\int_{(1+\epsilon)}\left(\frac{X}{a_1^2b_1^4\ell_1^2\ell_2^4c^4}\right)^s\xi_\pm(s,\eta'_{a_1b_1})\Tilde{h}(s)ds.    
\end{equation*}
Next, we shift the contour to $\Re(s)=-\epsilon$. From the pole at $s=1$ we pick up a residue contribution 
\begin{equation*}
    X\Tilde{h}(1)\sum_\sumstack{a_1,b_1,\ell_1,\ell_2,c \text{ coprime}\\ a_1b_1^2\ell_1^{2/3}\ell_2c\leq Q}\frac{\mu(a_1\ell_1c)}{a_1^3b_1^5\ell_1^2\ell_2^4c^4}\big(\sqrt{3}\alpha^+ \pm \alpha^{-}\big) +X\Tilde{h}(1)\beta \sum_\sumstack{a_1,b_1,\ell_1,\ell_2,c \text{ coprime}\\ a_1b_1^2\ell_1^{2/3}\ell_2c\leq Q}\frac{\mu(a_1\ell_1c)}{a_1^2b_1^4\ell_1^2\ell_2^4c^4}. 
\end{equation*}
The second of these sums matches up against the quadratic contribution from the large range, just as in the sharp case. The first sum can be extended to infinity at the cost of an error $X^{1+\epsilon}/Q^{3/2}$.

The function $\xi_-(s,\eta'_{a_1b_1})$ does not have a pole at $s=5/6$, however the function $\xi_+(s,\eta'_{a_1b_1})$ does have a pole. The contribution from the residue of this function is
\begin{equation*}
    2\sqrt{3}\gamma^+X^{5/6}\Tilde{h}(5/6)\sum_\sumstack{a_1,b_1,\ell_1,\ell_2,c \text{ coprime}\\ a_1b_1^2\ell_1^{2/3}\ell_2c\leq Q}\frac{\mu(a_1\ell_1c)}{a_1^2b_1^{11/3}\ell_1^{5/3}\ell_2^{10/3}c^{10/3}}.
\end{equation*}
As before, this sum may be extended to infinity at the cost of an error $X^{5/6+\epsilon}/Q$.

Now, $\xi_\pm(s,\eta'_{a_1b_1})$ has no additional poles, however the function $\Tilde{h}(s)$ potentially has poles at the non-positive integers. In particular, the residue at $s=0$ of $\Tilde{h}(s)$ equals $h(0)$. It will not be necessary to explicitly write down the contribution of this residue. Indeed, as in the end of Section \ref{dualS3onelevsect}, we shall see that this residue eventually cancels against other terms coming from the shifted integral
\begin{equation}\label{shiftsmoothint}
\frac{1}{2\pi i}\sum_\sumstack{a_1,b_1,\ell_1,\ell_2,c \text{ coprime}\\ a_1b_1^2\ell_1^{2/3}\ell_2c\leq Q}\mu(a_1\ell_1c)\int_{(-\epsilon)}\left(\frac{X}{a_1^2b_1^4\ell_1^2\ell_2^4c^4}\right)^s\xi_\pm(s,\eta'_{a_1b_1})\Tilde{h}(s)ds.  
\end{equation}

After applying the functional equation \eqref{genfcnal} to \eqref{shiftsmoothint}, we obtain
\begin{equation*}
    \frac{\pm 3}{2\pi i}\sum_\sumstack{a_1,b_1,\ell_1,\ell_2,c \text{ coprime}\\ a_1b_1^2\ell_1^{2/3}\ell_2c\leq Q}\mu(a_1\ell_1c)\int_{(-\epsilon)}\left(\frac{X}{a_1^2b_1^4\ell_1^2\ell_2^4c^4}\right)^s (a_1b_1)^{3-4s}\frac{\Delta^\pm(1-s)}{\Delta^\pm(s)}\xi_\pm(1-s,\mathbf{1}_{a_1b_1,\big((0)\cup (1^3)\big)'})\Tilde{h}(s)ds.
\end{equation*}
Here, $\big((0)\cup (1^3)\big)'$ simply denotes $(0)\cup (1^3)$, unless the modulus $a_1b_1$ is divisible by three, in which case $\mathbf{1}_{3,\big((0)\cup (1^3)\big)'} =\mathbf{1}_{3, \mathcal{O}_0\cup \mathcal{O}_1}$, consistent with Lemma \ref{eta3fourier}. Writing $a_1=a_1'\ell_1'$ and $b_1=b_1'\ell_2'$ to decompose the union above, we obtain
\begin{equation}\label{lastintsmoothct}
\begin{split}
    \frac{\pm 3}{2\pi i}&\sum_\sumstack{a_1',b_1',\ell_1,\ell_2,\ell_1',\ell_2',c \text{ coprime}\\ a_1'b_1'^2\ell_1^{2/3}\ell_2\ell_1'\ell_2'^2c\leq Q}\mu(a_1'\ell_1\ell_1'c)
    \\&\times\int_{(-\epsilon)}\left(\frac{X}{a_1'^2b_1'^4\ell_1^2\ell_2^4\ell_1'^2\ell_2'^4c^4}\right)^s (a_1'b_1')^{3-4s}(\ell_1'\ell_2')^{-1}\frac{\Delta^\pm(1-s)}{\Delta^\pm(s)}\xi_\pm(1-s,\mathbf{1}_{a_1'b_1'(3,\ell_1'\ell_2'),(1^3)''})\Tilde{h}(s)ds,
\end{split}
\end{equation}
with the convention that if $3\mid \ell_1'\ell_2'$, then one should consider $(1^3)''$ as equalling $(1^3)$, and if $3\mid a_1'b_1'$ then one uses $(1^3)''= (1^3)'$ as defined above. In the sequel, for the sake of brevity, we only consider the part of the sum where $3\nmid a_1'b_1'\ell_1'\ell_2'$. The part of the sum where this is not the case is handled similarly.

First, we note that the integral above can be absolutely bounded by
\begin{equation}\label{absintbdsmooth}
    \ll X^\epsilon \frac{(a_1'b_1')^{3}}{\ell_1'\ell_2'}\sup_{\Re(s)\leq -\epsilon} \abs{\xi_\pm(\mathbf{1}_{a_1'b_1'(3,\ell_1'\ell_2'),(1^3)''},1-s)}
\end{equation}

We are then led to study the zeta functions
\begin{equation*}
    \xi^\pm(1-s,\mathbf{1}_{d,(1^3)}) = \sum_\sumstack{f\in \GLtwo(\mathbb{Z})\setminus V(\mathbb{Z})\\ \pm \mathrm{Disc}(f) > 0} \frac{\mathbf{1}_{d,(1^3)}(f)}{\#\mathrm{Stab}(f)\abs{\mathrm{Disc}(f)}^{1-s}},
\end{equation*}
with $d=a_1'b_1'$, inside its domain of absolute convergence. For this, we proceed as in Section \ref{dualS3onelevsect}. As before, in \eqref{redcount}, the contribution to \eqref{absintbdsmooth} from reducible forms is bounded by
\begin{equation}
    \ll X^\epsilon \frac{1}{\ell_1'\ell_2'},
\end{equation}
which after summing yields a bound $\ll X^{\epsilon}Q$. Similarly, the contribution from forms corresponding to subrings of $C_3$-fields is negligible, cf. \eqref{c3sumbd}, giving a total contribution $\ll X^{1/2+\epsilon}$. It remains to handle the contribution corresponding to subrings of $S_3$-cubic fields. Given such a field $K$, we let $\eta_m(d,K)$ denote the number of index $m$-suborders of $K$, where $d$ is of type $(1^3)$. We remark that this number can be computed explicitly using Lemma \ref{totramsubordct} and \eqref{suborderctdw}. More precisely, it depends only on $m,d$ and the splitting types of the prime factors of $md$ in $K$. We remark that, as $d$ is not prime, we may not assume that $K$ is maximal at $d$, as we did in Section \ref{dualS3onelevsect}.

Now, as in \eqref{s3fcndecomponelev}, we decompose the $S_3$-part of the zeta function above as
\begin{equation*}
    \xi^{S_3}(s,\mathbf{1}_{d,(1^3)}) = \sum_{K \text{ $S_3$-cubic}}\frac{1}{\abs{\mathrm{Disc}(K)}^s}\sum_{m=1}^\infty \frac{\eta_m(d,K)}{m^{2s}},
\end{equation*}
where we have ignored the weighting depending on the sign of the discriminant of $K$, for the sake of simplicity. Recall also that $\eta_m(d,K)$ is multiplicative in $m$. Next, we sum over the possible splitting types of $m$ and interchange the order of summation
\begin{equation}\label{interordersmoothct}
    \sum_{m=1}^\infty m^{-2s} \sum_{(S_i)_{q\mid m}}\eta_m(d,(S_i)_{q\mid m})\sum_\sumstack{K \text{ $S_3$-cubic}}\frac{1}{\abs{\mathrm{Disc}(K)}^s}\mathbf{1}_{\text{$m$ of type $(S_i)_{q\mid m}$ in $K$}} \,\,\mathbf{1}_{\text{$e$ of type $(1^3)$ in $K$}},
\end{equation}
where $e=d/(d,m)$. Assuming $3\nmid m$, we write $\mathrm{sqf}(m)=P_{(111)}P_{(21)}P_{(3)}P_{(1^21)}P_{(1^3)}$ depending on the prescribed splitting behaviour at the various prime divisors of $m$. Then, in the notation of \eqref{Shankarthorneappl}, we use \cite[Theorem 11]{ShankarThorne} to write
\begin{equation}\label{secondshankarthorne}
\begin{split}
    &\sum_\sumstack{K \text{ $S_3$-cubic}}\frac{\mathbf{1}_{\text{$m$ of type $(S_i)_{q\mid m}$ in $K$}}\,\,\mathbf{1}_{\text{$e$ of type $(1^3)$ in $K$}}}{\abs{\mathrm{Disc}(K)}^s}  = \sum_\sumstack{\text{$P_{(1^21)}\mid D$ fundamental} \\ \text{ $K_D$ inert at $P_{(21)}$}\\K_D \text{ split at $P_{(111)}P_{(3)}$}\\ (D,eP_{(1^3)})=1}\frac{1}{\abs{D}^s}\sum_{r\mid P_{(3)}}\mu(r)\frac{1}{c_D3^{\omega(rP_{(111)})}}
    \\ \times &\left(\frac{1}{2}M_{1,D}(s)\prod_{q\mid P_{(1^3)}e}\left(\frac{1+\chi_{-3D}(q)}{q^{2s}}\right)\prod_{q\nmid 3Dme}\left(1+\frac{1+\chi_{-3D}(q)}{q^{2s}}\right)+\sum_{E\in \mathcal{L}_3(rP_{(111)},D) }M_{2,E}(s)\frac{\theta_E(eP_{(1^3)})}{(eP_{(1^3)})^{2s}}\prod_{ q\nmid 3Dme}\left(1+\frac{\theta_E(q)}{q^{2s}}\right)\right).
\end{split}
\end{equation}
We now wish to average this over $(d,m)e$. For this, we note the following facts, which follow from Lemma \ref{totramsubordct}. First, for primes $p\mid e$, $\eta_m(d,K)$ is zero unless $p$ is totally ramified in $K$. Next, if $p\mid d/e$ is partially ramified in $K$, then $\eta_m(d,K)$ is zero unless $p\mid m$. Finally, if $p\mid d/e$ is unramified in $K$, then $\eta_m(d,K)$ is zero unless $p^3\mid m$. As in Section \ref{dualS3onelevsect}, we always have the bound $\eta_m(d,K)\ll m^{1/3+\epsilon}$ irrespective of $d$.

Recall that $\theta_E(p)$ is nonzero only if $p$ is split or inert in $E$. In the first case, it equals $2$ and in the latter case it equals $-1$. Hence, $\theta_E$ is essentially equal to the Artin character of $E$, differing only when $p$ is partially ramified. The key is now to use the GRH to control averages over $\theta_E$, cf. \eqref{lastGRHonelev} and \eqref{grhoneoverL}. We obtain the exact same bound as in \eqref{grhoneoverL}. 

To perform the averaging, we will also need input from class field theory. More specifically we use that
\begin{equation*}
    \#\mathcal{L}_3(rP_{(111)},D)\ll \abs{rP_{(111)}}^\epsilon (\# \mathcal{L}_3(1,D)+1),
\end{equation*}
which follows from taking $3$-Sylow subgroups of the exact sequence \cite[Eq. (3.7)]{CohenThorne}. Next, as the primes which ramify in a field all divide the discriminant, we have that
\begin{equation}\label{aveclassgpbd}
    \sum_{f\leq F}\sum_\sumstack{ \abs{D}\leq tf\\ f\mid D} \#\mathcal{L}_3(rP_{(111)}, D)\ll \abs{rP_{(111)}}^\epsilon (tF)^{1+\epsilon}.
\end{equation}
Noting that this is $\ll t^{1+\epsilon}F^{3/2}$ is sufficient for our purposes.

We now estimate \eqref{lastintsmoothct} using \eqref{interordersmoothct} and \eqref{secondshankarthorne}. We first interchange the order of summation between the sum over $a_1'$ and that over $m$ and $(S_i)$. Next, we write $a_1' = ef_1f_2f_3f_4f_5$, where $e=a_1'/(a_1',m)$ and $f_1\mid (a_1',m)$ is the part of $(a_1',m)$ which is prescribed to be totally split, $f_2$ is prescribed to be partially split, $f_3$ inert, etc. Combining this with the observations above on divisibility of $m$, the GRH, and the bound \eqref{aveclassgpbd}, we obtain a total contribution
\begin{equation*}
    \ll Q^{3/2+\epsilon}
\end{equation*}
coming from the part of \eqref{secondshankarthorne} involving the sum over $E$. Similarly, we may bound the contribution from the rest of the $D$-sum involving the various characters to also be $\ll Q^{3/2+\epsilon}$, where we also write the condition that $K_D$ be inert at $P_{(21)}$ and split at $P_{(111)}P_{(3)}$ in terms of characters. Finally, just as in the end of Section \ref{dualS3onelevsect}, we obtain a contribution from the part of the $D$-sum which does not involve any characters. After expressing this in terms of standard zeta functions, we may shift this contribution to $\Re(s)=1/2-\epsilon$. The contribution from the residue at $s=0$ cancels against the previously obtained residue contribution for a total contribution
\begin{equation*}
    \ll X^{1/2+\epsilon}.
\end{equation*}
Optimising these error terms against those coming from the large range, see \eqref{largesmootherror}, we find the optimal choice $Q=X^{1/3}$ giving a total error
\begin{equation*}
    \ll X^{1/2+\epsilon},
\end{equation*}
which proves Theorem \ref{smoothctthm}.

\subsection{Splitting conditions}\label{splitcondsupsect}
We now describe how to accommodate prescribed splitting conditions in the sharp, and smooth, counting functions of cubic fields. To simplify the presentation, we consider the case of an arbitrary splitting condition at a single prime, although our methods work for any finite number of primes. Thus, we fix a prime $p$ and impose the splitting condition $S$.

\subsubsection{The sharp counting function}\label{splitcondsectsharp}
To study the sharp counting function, the relevant sieve to study is
\begin{equation*}
    \sum_{n\geq1}\mu(n)N^\pm(W_n, \mathbf{1}_{p,S};X),
\end{equation*}
where $N^\pm(W_n, \mathbf{1}_{p,S};X)$ counts the number of cubic forms which are nonmaximal at $n$ and split according to $S$ at $p$, together with a condition on the sign of the discriminant. If $S$ is one of the unramified splitting types, then nonmaximality at $p$ is impossible, whence we need only sum over $n$ such that $p\nmid n$. Using the discriminant-reducing identity \eqref{discredsum}, one then has
\begin{equation*}
    \sum_{n\geq1}\mu(n)N^\pm(W_n, \mathbf{1}_{p,S};X) = \sum_\sumstack{a,b,c \text{ coprime}\\ p\nmid abc}\mu(ac)N^\pm\big(\eta_{ab}\otimes \mathbf{1}_{p,S};X/(a^2b^4c^4)\big).
\end{equation*}

If $S$ is the totally ramified splitting type $(1^3)$, we instead use a slightly different discriminant-reducing identity, which follows from \cite[Eq. (69)]{BST}, which implies that
\begin{equation}\label{totramdiscred}
        \sum_{n\geq1}\mu(n)N^\pm(W_n, \mathbf{1}_{p,(1^3)};X) = \sum_\sumstack{a,b,c \text{ coprime}\\ p\nmid abc}\sum_{de\mid p}\mu(ace)N^\pm\big(\eta_{abe}\otimes \eta_{2,d} \otimes \mathbf{1}_{p/de,(1^3)};X/(a^2b^4c^4d^2e^4)\big),
\end{equation}
where $\eta_{2,p}(x)$ counts the number of roots of $x\nsmod{p}$ which have multiplicity at least $2$. Said differently, 
\begin{equation*}
    \eta_{2,p}(x) = \mathbf{1}_{p,(1^3)}+\mathbf{1}_{p,(1^21)}+(p+1)\mathbf{1}_{p,(0)}.
\end{equation*}
For the partially ramified splitting type, we instead have that \cite[Eq. 67]{BST}
\begin{equation}\label{partramdiscred}
        \sum_{n\geq1}\mu(n)N^\pm(W_n, \mathbf{1}_{p,(1^21)};X) = \sum_\sumstack{a,b,c \text{ coprime}\\ p\nmid abc}\sum_{de\mid p}\mu(acd)N^\pm\big(\eta_{abe}\otimes \eta_{2,d} \otimes \mathbf{1}_{p/de,(1^21)};X/(a^2b^4c^4d^2e^2)\big).
\end{equation}
To consider prescribed splitting types at multiple primes, one modifies the sieves above appropriately at each such prime.

Next, one proceeds as in Sections \ref{sharpGRHsect} and \ref{sharpnoGRHsect} to decompose the sum and split into a small and large range. To handle the small range, we use exact formulas for the various Fourier transforms, due to Mori \cite{Mori}, see also \cite[Theorem 11]{TTOrb}. In particular,
\begin{equation*}
    \widehat{\eta_{2,p}}\ll \frac{\mathbf{1}_{p,(0)}}{p}+\frac{\mathbf{1}_{\mathrm{Disc}\equiv 0 \nsmod p}}{p^2}, \,\,\, \widehat{\mathbf{1}}_{p, (1^3)}\ll \frac{\mathbf{1}_{p,(0)}}{p^2}+\frac{1}{p^3},\,\,\,\widehat{\mathbf{1}}_{p,(1^21)}\ll \frac{\mathbf{1}_{p,(0)}}{p}+\frac{\mathbf{1}_{\mathrm{Disc}\equiv 0 \nsmod p}}{p^2}+ \frac{1}{p^3},
\end{equation*}
and for unramified splitting types $S$, we have that
\begin{equation*}
    \widehat{\mathbf{1}}_{p,S}\ll \mathbf{1}_{p,(0)}+ \frac{\mathbf{1}_{p,(1^3)}}{p} +\frac{1}{p^2}.
\end{equation*}
Using these bounds, we may bound the contribution from the small range, with or without the GRH, similar to \cite[Section 6]{BTT}, except that we need not average to obtain a bound for the splitting type $(1^3)$, which seems to be a feature of our choice of discriminant-reducing identity. 

We obtain contributions
\begin{equation*}
    X^{3/5}Q^{1/5}p^{4/5}, \text{ and } X^{3/5}Q^{1/5}p^{-4/5}
\end{equation*}
for the unramified splitting types, and the totally ramified splitting type respectively. For the partially ramified splitting type, we obtain after averaging over $p\leq P$, a bound
\begin{equation*}
    X^{3/5}Q^{1/5}P^{4/5}.
\end{equation*}

The contribution from the large range is handled as in previous sections, taking advantage of oscillations in the sum, and bounding the reducible contribution separately. For the unramified splitting types, we obtain a bound of the same quality as in Sections \ref{sharpGRHsect} and \ref{sharpnoGRHsect}. For the partially ramified splitting type, we obtain a bound of this quality, even after averaging the error term as in \cite[Theorem 1.4]{BTT}. 

For the totally ramified splitting type, the situation is slightly more complex. Under the GRH, the bound from Section \ref{sharpGRHsect} simply gets multiplied by $p^{-2}$. When not assuming the GRH, we would like to conclude that we may simply multiply the bound from Section \ref{sharpnoGRHsect} by $p^{-2}$. However, Proposition \ref{zerodenspropcube} is not sufficiently strong to imply such a bound for arbitrarily large $p$. The issue is that fields of discriminant bounded by some $Y$, which are totally ramified at $p$ are quite rare among all fields of discriminant bounded by $Y$, for large $p$. Instead, we restrict our attention to small $p$, say $p\leq X^{1/100}$ in Theorem \ref{uncondsharpctthm}. Just as under the GRH, there is no uniformity problem under the quasi-Riemann hypothesis for $\zeta_K(s)$ and the second term from \eqref{quasibalance} is simply multiplied by $p^{-2}$. 

After optimising the cut-off variables, we arrive at the stated bounds for $E_1(\Sigma_p)$ and $E_2(\Sigma_p)$ from Theorem \ref{splitcondthmQ}, and the stated bound for $E_4(\Sigma_p)$ from Theorem \ref{quasiapplthm}.

\subsubsection{The smooth counting function}\label{smoothsplitsect}
We now describe how to incorporate splitting conditions in the smooth counting function studied in Section \ref{smoothctallfieldsect}, under the GRH. 

We begin by considering \eqref{smoothfirstcont}, and use the discriminant-reducing sieve \eqref{Shintanidiscred}, together with versions of the sieves \eqref{totramdiscred} and \eqref{partramdiscred} restated in the language of Shintani zeta functions. As before, we split into a large and small range and bound the contribution from the large range as in Section \ref{splitcondsectsharp}. 

In the small range, we shift the contour to $\Re(s)=-\epsilon$, and consider an analogue of \eqref{lastintsmoothct}. One first switches to maximal overrings, which requires an analogue of Lemma \ref{totramsubordct} for the splitting type $(1^21)$, proven in the same way. Next, by applying \cite[Theorem 11]{ShankarThorne}, we obtain an error of size $\ll Q^{3/2+\epsilon}$, multiplied by a factor which essentially equals $\widehat{\delta}_1(\mathbf{1}_{p,S})$ from \cite[Theorem 3.1]{BTT}. For unramified splitting types, this factor is $\ll p^2$. For the totally ramified splitting type, it is $\ll p$ and for the partially ramified splitting type, its average over $p\leq P$ is $\ll P^2$. In the ramified cases, we should actually use that $\widehat{\delta}_1(\eta_{p})\ll p$ and $\widehat{\delta}_1(\eta_{2,p})\ll p^2$ as well, but as the discriminants from \eqref{totramdiscred} and \eqref{partramdiscred} are decaying with respect to $e$ and $d$, the error-term contribution coming from this part of the sieve can be subsumed in the error term arising from $\widehat{\delta}_1(\mathbf{1}_{p,S})$.

Optimising the value of $Q$ yields the claimed bound for $E_3(\Sigma_p)$ in Theorem \ref{splitcondthmQ}.

\subsection{Function fields}\label{fcnfldctsect}

We end by briefly describing how to generalise our results to function fields. For the sake of brevity, we restrict our proof sketch to the counting function without splitting conditions.

We are then interested in counting field extensions of $\fq(T)$, where $q$ is coprime to $2$ and $3$. Now, as there are no Archimedean primes in $\fq(T)$, counting problems over this ground field closely resemble smooth counting problems over $\mathbb{Q}$. Moreover, as the GRH is proven for Dedekind zeta functions associated with extensions of $\fq(T)$, we shall see that we will obtain results of the same quality as those obtained for the corresponding smooth problem under GRH for extensions of $\mathbb{Q}$. Indeed, we count cubic field extensions of $\fq(T)$ by extending the methods of Section \ref{ffields}, similar to how we obtained the results of Section \ref{smoothctallfieldsect} by extending the methods of Section \ref{dualS3onelevsect}.

Let $X=q^{2n}$ for an integer $n\geq 1$. To study cubic fields of discriminant equal to $X$, where the prime $P_\infty$ splits according to the splitting type $S_\infty$, we use the inclusion-exclusion sieve
\begin{equation*}
    \sum_{F}\mu(F)N^{S_\infty}(W_F;X),
\end{equation*}
see \cite[Section 5]{Ahlquist2}. Here the summation ranges over monic squarefree $F$, and $N^{S_\infty}(W_F;X)$ counts the number of $\fq[T]$-integral binary cubic forms with discriminant equal to $X$, of type $S_\infty$ at $P_\infty$, which are nonmaximal at all primes dividing $F$. In the sequel, we write $R=\fq[T]$. 

Now, just as over $\mathbb{Q}$, we may decompose this sieve into 
\begin{equation*}
    \sum_{f_1,g_1,\ell_1,\ell_2,h \text{ coprime}}\mu(f_1\ell_1h) N^{S_\infty}(\eta'_{f_1g_1};X/\abs{f_1^2g_1^4\ell_1^2\ell_2^4h^4}).
\end{equation*}
We separate the range of summation into a large and small range according to whether $\abs{f_1g_1^2\ell_2h}\abs{\ell_1}^{2/3}\leq Q$ or not. In the large range, one proceeds analogously to Section \ref{sharpGRHsect} to isolate the reducible contribution and obtain an error
\begin{equation*}
    \ll \frac{X^{1+\epsilon}}{Q^{3/2}}.
\end{equation*}

In the small range, we forego the use of Shintani zeta functions and instead use the geometry-of-numbers methods from \cite{Ahlquist2} to study the counting functions $ N^{S_\infty}(\eta'_{f_1g_1};X/\abs{f_1^2g_1^4\ell_1^2\ell_2^4h^4})$. As described in Section \ref{dualfcnfldonelev}, one uses the finite Fourier transform to reduce the study of this counting function to a study of certain sums over degenerate and non-degenerate dual forms $y$.

The contribution from the non-degenerate forms $y$ is handled as in Section \ref{dualfcnfldonelev}, using a function field generalisation of \cite[Theorem 11]{ShankarThorne}. The contribution from such forms is analogous to the shifted integral from Section \ref{smoothsmallallfldctsect}, and its total contribution to the error term is bounded by
\begin{equation*}
    \ll Q^{3/2+\epsilon}.
\end{equation*}
By optimising $Q$, one obtains the desired error-term bound $\ll X^{1/2+\epsilon}$.

Now, there is also a contribution coming from degenerate forms $y$. In \cite{Ahlquist2}, this contribution is shown to equal the expected main and secondary terms, up to an error which sums to 
\begin{equation*}
    \ll X^{2/3+\epsilon},
\end{equation*}
see \cite[Eq. (5.9)]{Ahlquist2} and the subsequent computations. An error term of this size is clearly unacceptable for our purposes, whence we now describe how to better handle the contribution from degenerate forms. In fact, we show that the error term can be removed completely from the degenerate contribution. Our computations are analogous to the computations for determining the residues of the Shintani zeta functions $\xi^\pm(s)$ in \cite{Shintani}.

\subsubsection{The degenerate contribution}\label{degencontribsect}
In \cite[Section 5.2]{Ahlquist2} the contribution from degenerate $y$ is handled by isolating the contribution from $y$ of the form $(0,0,0,y_4)$ and discarding the contribution from other $y$. It turns out that if one instead considers the contribution from all degenerate $y$, then one obtains precisely the expected main and secondary terms, with no error term.

First, cf. \cite[Eq. (5.5)]{Ahlquist2}, the Fourier transform
\begin{equation*}
    \widehat{\eta}'_P = \frac{\mathbf{1}_{P,(0)}+\mathbf{1}_{P,(1^3)}}{\abs{P}}. 
\end{equation*}
The contribution coming from $y\equiv 0 \nsmod{P}$ is handled using inclusion-exclusion as in \cite[Eq. (5.6)]{Ahlquist2} so that we may focus on the contribution from forms of type $(1^3)$ modulo a divisor $r\mid f_1g_1$. 

Now, just as in \cite[Proposition 2.10]{Shintani}, we have an orbital description of the degenerate forms. Indeed, all $R$-integral degenerate forms with a triple root in $\fq(T)$ lie in the $\GLtwo(R)$ orbit of a unique form $(0,0,0,f)$, with $f\in R$ monic. Moreover, an $R$-integral degenerate form with a double root in $\fq(T)$ lies in the orbit of a unique form $(0,0,f',f)$, with $f,f'$ monic, $f'\neq 0$, and $\abs{f}< \abs{f'}$. Hence, the degenerate forms $y$ with a triple root at $r$ lie in the orbit of $(0,0,0,f)$, $(r,f)=1$ or $(0,0,rf',f)$, with $(f',r)=1$ and $\abs{f}<\abs{rf'}$. 

The degenerate contribution can now be studied explicitly by using the unfolding method cf. \eqref{oscintff}, and a certain $KAN$-decomposition for $\GLtwo(R)$, see e.g. \cite[Section 4.5]{Ahlquist2}. For simplicity, we consider the case $r=P$. One proceeds by decomposing the open compact set $B$ defined prior to \eqref{ffonelevint} into a union of disjoint boxes
\begin{equation*}
    B=\bigcup_{i=1}^{i_0} (v^{(i)} +C),\,\,\, C=\{(x_1,x_2,x_3,x_4): \abs{x_i} < c\},
\end{equation*}
for a sufficiently small $c$. Similarly, we define the inverted box $C^{-1}$ by replacing the inequality by $\abs{x_i}< c^{-1}$. After applying the unfolding method, the contribution from degenerate $y$ with a triple root in $R$ becomes up to a constant, cf. \eqref{oscintff},
\begin{equation*}
    \frac{1}{q-1}\sum_{i=1}^{i_0}c^4\abs{\lambda_0}^4\int_{g\in \GLtwo_Y(K_\infty)/(N\cap \GLtwo(R))}\sum_{y}\chi_\infty([v^{(i)},g^Ty]/P)\mathbf{1}_{gy\in rC^{-1}}dg,
\end{equation*}
where $N$ is the set of lower-triangular matrices with both diagonal elements equalling $1$ and $y$ ranges over nonzero dual forms $(0,0,0,a)$. For technical reasons, we must drop the condition that $P\nmid a$ when using the unfolding method. Define
\begin{equation*}
    a(t)= \begin{pmatrix}
        t &0\\0 & t^{-1}
    \end{pmatrix}.
\end{equation*}
Then, after taking quotients by $N$ and using $K$-invariance of $B$, we see that our integral is essentially
\begin{equation*}
    \frac{1}{q-1}\sum_{i=1}^{i_0}c^4\abs{\lambda_0}^4\int_{g=\lambda a(t)}\sum_{a}\chi_\infty([v^{(i)},g(0,0,0,a)]/P)\mathbf{1}_{gy\in PC^{-1}}\frac{d\lambda dt}{\abs{\lambda t^3}}.
\end{equation*}
Here, the absolute value $\abs{\lambda}=\abs{\lambda_0}$ is fixed, and $t$ is monic, in the sense that its Laurent expansion has leading coefficient equal to $1$. The measure $d\lambda$ is normalised so that it gives the set of $\lambda$ with $\abs{\lambda }<1$  measure $1$. The same holds true for $dt$. Technically, instead of considering the integral which is written above, one should consider limits of integrals with the added restriction that $\abs{t}<q^m$, and let $m\to \infty$. In the sequel the condition $\abs{t} < q^m$ is implicit in our integrals. Now, acting by $g$ above yields
\begin{equation*}
    \frac{1}{q-1}\sum_{i=1}^{i_0}c^4\abs{\lambda_0}^4\int_{g=\lambda a(t)}\sum_{a}\chi_\infty(v_4^{(i)}\lambda a/Pt^3)\mathbf{1}_{\abs{\lambda a/t^3}< \abs{P}c^{-1}}\frac{d\lambda dt}{\abs{\lambda t^3}},
\end{equation*}
where $v_4^{(i)}$ is the fourth coordinate of $v^{(i)}$.

We now split into two cases. First, if $\abs{v_4^{(i)}}< c$, then without loss of generality, $v_4^{(i)}=0$, as the box $v_4^{(i)}+C=C$. In this case, the character above equals $1$ for all choices of $g$ and we see that the integral above may be computed explicitly. More specifically, one computes
\begin{equation*}
    \sum_{\abs{a}<\abs{Pt^3}c^{-1}/\abs{\lambda_0}} 1= (\abs{Pt^3}c^{-1}/\abs{\lambda_0}-1)/(q-1),
\end{equation*}
assuming that $\abs{Pt^3}c^{-1}/\abs{\lambda_0} > 1$. One also notes that the integral over $\lambda$ provides a factor $q-1$ in this case. The full $t$-integral over the first of these terms is divergent, however we shall see that because of cancellation, this is not a problem if one interprets all integrals as the appropriate limits described above. Integrating the above, we obtain, excluding the factor $(q-1)^{-1}$,
\begin{equation*}
\vol(B_0)\abs{P}\abs{\lambda_0}^3\int_{\abs{t}^3>c\abs{\lambda_0}/\abs{P}}dt-c\abs{\lambda_0}^4\vol(B_0)\int_{\abs{t}^3>c\abs{\lambda_0}/\abs{P}}\frac{dt}{\abs{t}^3}
\end{equation*}

We should also consider the contribution from the more difficult case $\abs{v_4^{(i)}}\geq c$. Then, the box $v^{(i)}+C$ does not contain any point with last coordinate zero. From the definition of $\chi_\infty$ one sees after averaging over $\lambda$ and summing over $a$ that such a sum is nonzero only if $\abs{v_4^{(i)}}< \abs{Pt^3}/\abs{\lambda_0}$, in which case it equals $-1$. Indeed, this follows from the relation
\begin{equation}\label{lambdaintrel}
    \int_{\abs{\lambda}=1}\chi_\infty(\lambda a)d\lambda =\begin{cases}
        q-1, &\text{ if $\abs{a}< q^{-1}$,}\\
        -1, &\text{ if $\abs{a}= q^{-1}$,}\\
        0, &\text{ if $\abs{a}> q^{-1}$,}
    \end{cases}
\end{equation}
and the fact that there are exactly $q^k$ monic polynomials of degree equal to some $k$, and that there are $(q^k-1)(q-1)^{-1}$ nonzero monic polynomials of degree strictly less than $k$. 

Let $B_u$ denote the set $\{(x_1,x_2,x_3)\in K_\infty^3: (x_1,x_2,x_3,u)\in B\}$. Then, using that $\vol(C)=c^4$, we see that the contribution from the above term equals, up to the factor $(q-1)^{-1}$,
\begin{equation*}
    -\sum_{i=1}^{i_0}c^4\abs{\lambda_0}^4\int_{\abs{t}^3 > c\abs{\lambda_0}/\abs{P}}\mathbf{1}_{\abs{v_4^{(i)}}< \abs{Pt}^3/\abs{\lambda_0}}\frac{dt}{\abs{t^3}}=-\abs{\lambda_0}^4\int_{\abs{t}^3 > c\abs{\lambda_0}/\abs{P}}\int_{\abs{u} \geq c}\vol(B_u)\mathbf{1}_{\abs{u}< \abs{Pt}^3/\abs{\lambda_0}}du\frac{dt}{\abs{t^3}}. 
\end{equation*}
Interchanging the order of integration, we obtain
\begin{equation*}
\begin{split}
    -\abs{\lambda_0}^4\int_{\abs{u}\geq c}&\vol(B_u)\int_{\abs{t}^3 > \abs{u\lambda_0}/\abs{P}}\frac{dt}{\abs{t^3}}du =-\abs{\lambda_0}^4\int_{u}\vol(B_u)\int_{\abs{t}^3 > \abs{u\lambda_0}/\abs{P}}\frac{dt}{\abs{t^3}}du\\&+\vol(B_0)\abs{\lambda_0}^4\int_{\abs{u} < c}\int_{\abs{t}^3 > \abs{u\lambda_0}/\abs{P}}\frac{dt}{\abs{t^3}}du. 
\end{split}
\end{equation*}
Interchanging the order of integration in the last term yields
\begin{equation*}
    -\abs{\lambda_0}^4\int_{u}\vol(B_u)\int_{\abs{t}^3 > \abs{u\lambda_0}/\abs{P}}\frac{dt}{\abs{t^3}}du+c\vol(B_0)\abs{\lambda_0}^4\int_{\abs{t}^3 > c\abs{\lambda_0}/\abs{P}}\frac{dt}{\abs{t}^3} +\vol(B_0)\abs{\lambda_0}^3\abs{P}\int_{\abs{t}^3\leq c\abs{\lambda_0}/\abs{P}}dt. 
\end{equation*}

One studies the contribution from degenerate forms with a double root similarly. Let 
\begin{equation*}
    n(f)=\begin{pmatrix}
        1 & 0 \\ f & 1
    \end{pmatrix}.
\end{equation*}
Then, we consider the integral
\begin{equation*}
    c^4\abs{\lambda_0}^4\int_{\abs{\lambda}=\abs{\lambda_0},\,\, \abs{f} <1}\sum_{a,b}\chi_\infty([v^{(i)},(0,0,b\lambda/t, (fb+a/P) \lambda/t^3)])\mathbf{1}_{(0,0,b\lambda/t, (fb+a/P) \lambda/t^3)\in C^{-1}}df\frac{d\lambda}{\abs{\lambda}}\frac{dt}{\abs{t}^3},
\end{equation*}
where the sum ranges over monic $a,b$ with $\abs{a} < \abs{Pb}$. Note that as $b\neq 0$, the condition in the indicator function implies that $\abs{t} >c\abs{\lambda_0}$. Moreover, for every fixed $b$, after interchanging order of summation, the indicator function implies that one should integrate over $\abs{fb} < c^{-1}\abs{t^3}/\abs{\lambda_0}$. Using \eqref{lambdaintrel}, this shows that the integral over $f$ and $\lambda$ is zero, unless $\abs{v_4^{(i)}} <c$, i.e. $v_4^{(i)}=0$, so we may assume that this is the case. Then $f$ simply integrates to $c^{-1}\abs{t}^3/\abs{b\lambda_0}$ and the integral becomes
\begin{equation*}
    \frac{\abs{P}\abs{\lambda_0}^3c^3}{q-1}\sum_{b}\int_{ \abs{\lambda}=\abs{\lambda_0},\, \abs{t}> c\abs{\lambda_0}}\chi_\infty(v_3^{(i)}b\lambda/t)\mathbf{1}_{\abs{b\lambda}/\abs{t}<c^{-1}}\frac{d\lambda}{\abs{\lambda}}dt.
\end{equation*}
Now, $B$ contains no degenerate points so that $\abs{v_3^{(i)}}\geq c$. Hence, the $b$-sum and the $\lambda$-integral yield $-1$, if $\abs{v_3^{(i)}} < \abs{t}/\abs{\lambda_0}$, and $0$ else. Thus, after summing over all boxes we obtain a contribution
\begin{equation*}
    -\abs{P}\abs{\lambda_0}^3\int_{u}\vol(B_{u,0})\int_{\abs{t} > \abs{u\lambda_0}}dtdu =-\abs{P}\abs{\lambda_0}^3\vol(B_0)\int_{t}dt+\abs{P}\abs{\lambda_0}^3\int_{u}\vol(B_{u,0})\int_{\abs{t} \leq \abs{u\lambda_0}}dtdu,
\end{equation*}
excluding the factor $(q-1)^{-1}$

Adding together the various integrals, we obtain a total contribution
\begin{equation*}
\begin{split}
    &\frac{1}{q-1}\left(\abs{P}\abs{\lambda_0}^4\int_{u}\vol(B_{u,0})\abs{u}du -\abs{\lambda_0}^4\int_{u}\vol(B_u)\int_{\abs{t}^3 > \abs{u\lambda_0}/\abs{P}}\frac{dt}{\abs{t^3}}du\right).
\end{split}
\end{equation*}
The first of these integrals is precisely the reducible main term, and the second of the integrals is the irreducible secondary term. The irreducible main term comes from the contribution $y=0$, as desired.

\end{document}